\documentclass[a4paper,12pt]{article}
\usepackage{amsmath,amsthm,amssymb,enumerate}
\usepackage{lipsum}
\usepackage{bm}
\usepackage{cancel}
\usepackage[affil-it]{authblk}
\usepackage{gensymb}
\usepackage[square,sort&compress,comma,numbers]{natbib}
\usepackage{mathptmx} 
\usepackage{hyperref}
\usepackage{subfig}
\usepackage{graphicx}
\usepackage[margin=3cm]{geometry}
\usepackage{tikz}
\usepackage{esint}
\usepackage{caption}
\usetikzlibrary{shapes,calc}
\usepackage{verbatim}
 \usepackage{stmaryrd}
\usepackage{bm}
\usepackage{marginnote}
\chardef\bslash=`\\ 

\newtheorem{thm}{Theorem}[section]
\newtheorem{cor}[thm]{Corollary}
\newtheorem{lem}[thm]{Lemma}

\theoremstyle{definition}

\theoremstyle{remark}
\newtheorem{rem}{Remark}[section]
\newtheorem{example}{Example}[section]

\numberwithin{equation}{section}

\newcommand{\bN}{\mathbb N}
\newcommand{\bR}{\mathbb R}

\newcommand{\bT}{\mathbb T}

\newcommand{\bS}{\mathbb S}

\newcommand{\cC}{\mathcal C}
\newcommand{\cD}{\mathcal D}
\newcommand{\cE}{\mathcal E}

\newcommand{\cT}{\mathcal T}
\newcommand{\cL}{\mathcal L}

\newcommand{ \cN}{\mathcal N}

\newcommand{\cM}{M}

\newcommand{\lt}{L^2(\Omega)}

\newcommand{\hto}{H^2_0(\Omega)}

\newcommand{\integ}{\int_\Omega}
\newcommand{\sit}{\sum_{T\in\mathcal{T}}\int_T}
\newcommand{\sik}{\sum_{K\in\mathcal{T}}\int_K}

\newcommand{\divc}{\mathrm{div}}

\newcommand{\sie}{\sum_{E\in \mathcal{E}}\int_E}

\newcommand{\fl}{\quad \text{for all}\:}

\newcommand{\half}{\frac{1}{2}}
\newcommand{\trinl}{\ensuremath{\left| \! \left| \! \left|}}
\newcommand{\trinr}{\ensuremath{\right| \! \right| \! \right|}}

\newcommand{\dx}{{\rm\,dx}}
\newcommand{\ds}{{\rm\,ds}}

\newcommand{\cof}{{\rm cof}}

\newcommand{\Holder}{H\"{o}lder~}
\usepackage[normalem]{ulem}
\definecolor{violet}{rgb}{0.580,0.,0.827}
\newcommand{\ccnew}[1]{#1}
\begin{document}

\newpage
\setcounter{page}{1}

\title{Nonconforming Finite Element Discretisation for 
Semilinear Problems with Trilinear Nonlinearity}
\author{Carsten Carstensen\footnote{Department of Mathematics, Humboldt-Universit\"{a}t zu 
Berlin, 10099 Berlin, Germany.  Distinguished Visiting Professor, Department of Mathematics, 
Indian institute of Technology Bombay, Powai, Mumbai-400076. Email cc@math.hu-berlin.de}        
$\boldsymbol{\cdot}$ Gouranga Mallik
\footnote{Department of Mathematics,
Indian Institute of Science, Bangalore 560012
India. Email gourangam@iisc.ac.in}
$\boldsymbol{\cdot}$ Neela Nataraj
\footnote{Department of Mathematics, Indian Institute of Technology Bombay, Powai, Mumbai 400076, India. Email neela@math.iitb.ac.in}
}
\maketitle
\begin{abstract} 
The Morley finite element method (FEM) is attractive for  semilinear problems 
with the biharmonic operator as a leading term in  the stream function vorticity formulation of 2D Navier-Stokes problem and in the von  K\'{a}rm\'{a}n equations. This paper establishes a 
best-approximation a~priori error analysis and an a~posteriori error analysis of discrete
solutions close to an arbitrary  regular solution on the continuous level to semilinear problems with a trilinear nonlinearity.
The analysis avoids any smallness assumptions on the data and so has to 
provide discrete stability by a perturbation analysis before the Newton-Kantorovic 
theorem can provide the existence of discrete solutions. An abstract framework 
 for the stability analysis in terms of discrete operators from the medius analysis 
 leads to 
 new results
 on the nonconforming  Crouzeix-Raviart FEM for 
 second-order linear non-selfadjoint and indefinite elliptic problems 
 with $L^\infty$ coefficients. The  paper  identifies  six parameters 
 and sufficient conditions for the local a~priori and  a~posteriori error control of 
 conforming and nonconforming  discretisations of a class of semilinear 
 elliptic problems first in an abstract framework and then in the
 two semilinear  applications. This leads 
 to new best-approximation error estimates  
 and to 
  a~posteriori error estimates in terms of explicit residual-based 
 error control for the conforming and Morley FEM.
\end{abstract}


\section{Introduction}

\subsection{Motivation}
The nonconforming finite element methods (FEMs) have recently been rehabilitated
by the medius analysis, which combines arguments from traditional a~priori and a~posteriori error analysis \cite{Gudi10}. {In particular, nonconforming finite element schemes can be equivalent \cite{CCDGNN15,CC_DP_MS12} or superior to conforming finite element 
schemes \cite{CCKKDPMS15}.} The conforming FEMs for fourth-order problems require $C^1$ conformity and lead to cumbersome implementations, while the nonconforming Morley FEM 
is as simple as quadratic Lagrange finite elements; the reader may consider the 
finite element program in 
\cite[Sec. 6.5]{CCDGJH14} with less than 30 lines of Matlab for a proof of its simplicity. {The second-best scheme of easy implementations for  fourth-order problems is the 
$C^0$ interior penalty method (C0IP) \cite{BNRS17,CCGMNN18} with the benefit of higher-order 
variants and the disadvantage of a (critical) stability parameter choice.  The optimal convergence rates 
are known for the  adaptive Morley FEM \cite{CCDGJH14,CCNN19} in fourth-order problems, 
but open for C0IP;  cf.
\cite{kreuzer2019convergence, BN10, KP07}  for the state of the art in second-order applications.  
Hence the advantage of higher-order schemes is not guaranteed for C0IP and leaves 
the Morley FEM as the method of choice. 

\medskip

{ This relevance of the nonconforming Morley FEM  for 
fourth-order problems is not reflected in} the contributions  in the literature on the attractive application 
to semilinear problems with the linear biharmonic operator as the leading term 
{ (}plus quadratic lower-order contributions{)}.  
There are important  model applications of this problem in the stream-function formulation of the incompressible 2D Navier-Stokes equations \cite{BrezziRappazRaviart80,CN86,CN89} 
and in the von K\'{a}rm\'{a}n equations \cite{CiarletPlates,GMNN_NCFEM} for 
nonlinear plates in solid mechanics.  
This paper enriches the general theory of semilinear problems for trilinear low-order terms from conforming FEMs \cite{BrezziRappazRaviart80} to nonconforming FEMs with the medius analysis. This overcomes the smallness assumption (on the load $f$) in  \cite{CN89} and adds 
 a~posteriori error control beyond   \cite{CCGMNN18} for a dG discretisation.}
The Morley FEM allows for additional benefits and leads, for instance, to guaranteed
lower eigenvalue bounds \cite{CCDG14_eigenvalues}. 

\subsection{Discrete Stability}
This paper considers the local approximation of a general regular solution $u$
to a nonlinear function $N(u)$ without any
extra conditions. 
The invertible Frech\'et derivative $DN(u)$ of the  nonlinear function $N:X\to Y^*$ 
at a regular solution $u$ is by definition 
a linear bijection between Banach spaces $X$ and $Y^*$;  
this is equivalent to  an $\inf$-$\sup$ condition on the associated 
bilinear form $DN(u;\bullet,\bullet)=a+b: X\times Y\to\bR$ (split into two contributions
$a$ and $b$  in Section~2). For a nonconforming finite element  discretisation with some finite element 
space $X_h\times Y_h \not\subset X\times Y$, in the 
absence of further conditions,  the  $\inf$-$\sup$ condition for 
$a+b: X\times Y\to\bR$ does {\em not}
imply an $\inf$-$\sup$ condition  for the discrete bilinear form
$a_h+b_h:X_h\times Y_h\to\bR$.  Section~2 studies two general 
bilinear forms $\widehat{a}$ and $\widehat{b}$ defined on a superspace 
$\widehat{X}\times 
\widehat{Y}$ of $X_h\times Y_h$ and $X\times Y$, and introduces 
four parameters in  {\bf (H1)}-{\bf (H4)} with a sufficient  condition for an
$\inf$-$\sup$ condition  to hold for $a_h+b_h: X_h\times Y_h\to\bR$
to enable a Petrov-Galerkin scheme and is the first contribution of this paper. 

\medskip

There will be three applications of this abstract framework in this paper. The first
of which is on former results in  \cite{CCADNNAKP15}
on a nonconforming Crouzeix-Raviart FEM for well-posed 
second-order linear self-adjoint and indefinite elliptic problems: Since the framework
applies the medius analysis tools, there are no smoothness assumptions and the
feasibility and best-approximation property  for sufficiently small mesh-sizes 
is newly established for the Crouzeix-Raviart FEM for $L^\infty$ coefficients in this paper 
(compared to  piecewise Lipschitz continuous coefficients in  \cite{CCADNNAKP15}).

\subsection{Fourth-order semilinear problems}
The second and third applications of this discrete stability framework of Section~2
is on semilinear problems with a trilinear nonlinearity:  The stream function formulation of the incompressible 2D 
Navier-Stokes problem \cite{BrezziRappazRaviart80} in Section~4
and the von K\'{a}rm\'{a}n equations \cite{CiarletPlates,Brezzi} in Section~5 with conforming and Morley FEM.  
The abstract stability result (a) overcomes the high regularity assumptions 
$u \in H^2_0(\Omega) \cap H^3(\Omega)$ and (b) is not restricted to small 
data  as in \cite{CN86,CN89}.

\medskip

\subsection{Overview of further results} 
\noindent{The main abstract results are stated in {\bf (A)}-{\bf (D)} below.} Here and throughout the paper, {it is  assumed that the mesh needs to be sufficiently fine to well approximate the solutions to the linear problems associated with the leading elliptic differential operator. 

Throughout this subsection let $N:V\to V^*$ be a differentiable function in a Hilbert space $V$ with dual $V^*$
with one fixed  regular solution $u$ to $N(u)=0$. The Hilbert space is a Sobolev space $H_0^m(\Omega)$
associated to some  polyhedral bounded Lipschitz domain 
$\Omega\subset\mathbb{R}^n$ that is partitioned by arbitrarily fine 
shape-regular triangulations into simplices. The latter form  a  family $\bT$ and given  any $\delta>0$,
 let $ \bT(\delta)$ denote the (nonempty) 
subset of all triangulations $\cT$ of maximal mesh-size smaller than or equal to $\delta$.
For each  $\cT\in \bT$, suppose there is a conforming or nonconforming finite element space $V_h(\cT)$ and a 
differentiable function $N_h:V_h(\cT)\to V^*_h(\cT)$ with additional conditions; in particular, there 
is a norm $\| \bullet \|_{\widehat{V}}$ on $V+V_h(\cT)$ that extends the norm in $V$.
(Notice the simplified notation $N_h\equiv N_h(\cT)$.) 
This paper discusses conditions in  {\bf (H1)}-{\bf(H6)} sufficient for the subsequent consequences.} 

\medskip

{\bf (A).} 
 {\em There exist  $\epsilon,\delta>0$ such that, for all
$\displaystyle\cT\in\bT(\delta)$, there exists a unique discrete solution $u_h\in 
V_h(\cT) $ to $N_h(u_h)=0$ with $\| u-u_h\|_{\widehat{V}} \le \epsilon$.}

{\bf (B).} 
 {\em  There exist $\epsilon,\delta,\rho>0$ such that  {\bf (A)} holds and, 
for all $\displaystyle\cT\in\bT(\delta)$ and  for any  initial iterate  
$u_h^{(0)}\in  V_h(\cT)$ with $\|u_h-u_h^{(0)}\|_{\widehat{V}}  \le \rho$,
 the Newton scheme converges
quadratically to  $u_h$ .}

{\bf (C).}  {\em There exist $\epsilon,\delta,C_{\text{\rm qo}}>0$ such that {\bf (A)} holds and, for all 
$\cT\in\bT(\delta)$,
$$
\| u-u_h\|_{\widehat{V}}\leq 
C_{\text{\rm qo}}\left(\min_{v_h\in  V_h(\cT)}\| u-v_h\|_{\widehat{V}}+{\rm apx}(\cT)\right)
$$
with some approximation term ${\rm apx}(\cT)$ to be specified in the particular 
application.}

A local reliable and efficient
a~posteriori error control holds even for inexact solve (owing to a 
termination in an iterative solver)  in the   sense of

{\bf (D).} {\em 
There exist   $\epsilon,\delta,C_{\text{\rm rel}},C_{\text{\rm eff}}>0$ such that  
any approximation   $v_h\in V_h(\cT) $ 
with $\| u-v_h\|_{\widehat{V}}\leq\epsilon$ and $\displaystyle\cT\in\bT(\delta)$
satisfies
$$
C_{\text{\rm rel}}^{-1}\| u-v_h\|_{\widehat{V} } \le 
\| N(v_h)\|_{\widehat V^*} + \min_{v\in V}  \| v_h-v \|_{\widehat{V}}
\le C_{\text{\rm eff}}\| u-v_h\|_{\widehat{V} }.
$$
}

It is part of the abstract results in Section~2 and 3 to identify the reliability and efficiency constants in the above displayed estimate and prove that the positive constants 
$\epsilon$, $\delta$, $\rho$, $C_{\rm qo}$, $C_{\text{\rm rel}}$, and 
$C_{\text{\rm eff}}$ are  mesh-independent. 

The abstract error control in {\bf (D)} is the point of departure in the applications to the  stream function formulation of the incompressible 2D Navier-Stokes problem \cite{BrezziRappazRaviart80} in Section~4
and the von K\'{a}rm\'{a}n equations \cite{CiarletPlates,Brezzi} in Section~5.  
This paper  establishes the first  reliable  estimate of $\| N(v_h)\|_{\widehat V^*} $ and 
$\min_{v\in V}  \| v_h-v \|_{\widehat{V}}$ in terms of an explicit residual-based
error estimator for the conforming and Morley FEM and discusses its efficiency. 

\subsection{Outlook}

This  presentation is restricted to quadratic problems in which the weak formulation involves  a trilinear form for a simple outline to cover  two important semilinear fourth-order problems. The generalisation to more general and stronger nonlinearities, however, requires appropriate  growth conditions in various norms and involves a more technical framework. The presentation matches exactly the nonconforming applications 
(Crouzeix-Raviart and Morley finite elements); other schemes
like  the discontinuous Galerkin schemes \cite{CCGMNN18} with their discrete norms and various jump conditions could be included with more additional technicalities.

\medskip
{The ideas developed in this paper extend to other semilinear problems,
 to optimal control and obstacle problems \cite{BSZZ13,BSZ12}  (governed by fourth-order plates and very thin plates), fully nonlinear Monge-Ampere equations based on vanishing moment method \cite{N10}. 
Moreover, the  lowest-order version of   skeletal or polytopal, hybridizable  discontinuous Galerkin and   
higher-order hybrid methods   \cite{DDE15, DEL16,dd19}  is a perturbation of the
nonconforming Crouzeix-Raviart finite element method for the Poisson problems. It is therefore expected that
the  Morley FEM is related to the lowest-order variant of skeletal schemes  for PDEs governed by fourth-order elliptic equations \cite{BDGK17}. In this way this paper stimulates the development  of 
the {\it a priori} and {\it a posteriori} error analysis of  those schemes.}

\subsection{General notation}
Standard notation on Lebesgue and Sobolev spaces applies throughout the paper and $\|\bullet\|$ abbreviates $\|\bullet\|_{L^2(\Omega)}$ with a $L^2$ scalar product $(\bullet,\bullet)_{L^2(\Omega)}$, while the duality brackets $<\bullet,\bullet>_{V^*\times V}$ 
are reserved for a dual pairing in $V^*\times V$; 
$\|\bullet\|_{\infty}$ abbreviates the norm in $L^\infty(\Omega)$; 
$H^m(\Omega)$ denotes the  Sobolev  spaces of order $m$ with  
norm $\|\bullet\|_{H^m(\Omega)}$; $H^{-1}(\Omega)$ 
(resp. $H^{-2}(\Omega)$) is the dual space of 
$H^1_0(\Omega):=\{v\in H^1(\Omega): v|_{\partial \Omega}=0\}$ 
(resp. $H^2_0(\Omega):=\{v\in H^2(\Omega): v|_{\partial \Omega}=\frac{\partial v}{\partial \nu}|_{\partial \Omega}=0\}$). With a regular triangulation $\cT$ of the polygonal Lipschitz domain $\Omega\subset\bR^n$ into simplices, associate 
its piecewise constant mesh-size $h_\cT\in P_0(\cT)$ with 
$h_\cT|_T:=h_T:=\text{diam}(T) \approx  |T|^{1/n}$ for all $T\in\cT$ and its 
maximal mesh-size  $h_{\max}:=\max h_\cT$.
Here and throughout, 
$$
P_k(\cT):=\left\{v\in L^2(\Omega):\,\forall T\in\cT,v|_{T}\in P_k(T)\right\}
$$
denotes the piecewise polynomials of degree at most $k\in\mathbb{N}_0$
and let  $\Pi_k$ denote the  $L^2(\Omega)$ (resp. $L^2(\Omega;\bR^n)$ or 
$L^2(\Omega;\bR^{n\times n})$) orthogonal projection onto $P_k(\cT)$ 
(resp. $P_k(\cT;\bR^m)$ or $P_k(\cT;\bR^{m \times m})$). Oscillations of 
degree $k$ read 
\[
{\rm osc}_k(\bullet,\cT):=\|h_{\cT}^{p}(I-\Pi_k)\bullet\|_{L^2(\Omega)}
\]
with its square ${\rm osc}_k^2(\bullet,\cT):={\rm osc}_k(\bullet,\cT)^2$ for 
$p=1$ for second-order in Section~2
and $p=2$ for fourth-order problems in Sections~4 and 5. 
The notation $A\lesssim B$ means there exists a generic 
$h_{\cT}$-independent constant $C$ such that $A\leq CB$; $A\approx B$ abbreviates $A\lesssim B\lesssim A$. In the sequel, $C_{\text {rel}}$ and $C_{\text {eff}}$ denote generic reliability and efficiency constants. The set of all $n\times n$ real symmetric matrices is $\bS:=\bR^{n\times n}_{sym}$. 
\section{Well-posedness of the discrete problem}\label{infsup}
This section presents sufficient conditions for the stability of nonconforming discretizations of a well-posed linear problem. Subsection 2.1 introduces four parameters {\bf (H1)}-{\bf (H4)} and a condition on them  sufficient for a discrete inf-sup condition for the sum $a+b$ of two bilinear forms $a, b: X\times Y\to \bR$ extended to {superspaces} $\widehat{X}\supset X+X_h$ and  
$\widehat{Y}\supset Y+Y_h$. Subsection 2.2 discusses a first application to second-order non-self adjoint and indefinite elliptic problems \cite{CCADNNAKP15}. 

\subsection{Abstract discrete inf-sup condition}\label{sec:abs_result}
Let $\widehat{X}$ (resp. $\widehat{Y}$) be a real Banach space with norm $\|\bullet\|_{\widehat{X}}$ (resp. $\|\bullet\|_{\widehat{Y}}$) and suppose $X$ and $X_h$ (resp. $Y$ and $Y_h$) are two complete linear subspaces of $\widehat{X}$ (resp. $\widehat{Y}$) with inherited norms $\|\bullet\|_{X}:=\left(\|\bullet\|_{\widehat{X}}\right)|_{X}$ and $\|\bullet\|_{X_h}:=\left(\|\bullet\|_{\widehat{X}}\right)|_{X_h}$ (resp. $\|\bullet\|_{Y}:=\left(\|\bullet\|_{\widehat{Y}}\right)|_{Y}$ and $\|\bullet\|_{Y_h}:=\left(\|\bullet\|_{\widehat{Y}}\right)|_{Y_h}$). Let $\widehat{a},\widehat{b}:\widehat{X}\times \widehat{Y}\to \bR$ be bounded bilinear forms and abbreviate
\begin{align}\label{defn_ab}
&a:=\widehat{a}|_{X\times Y},\: a_h:=\widehat{a}|_{X_h\times Y_h}\text{ and } b:=\widehat{b}|_{X\times Y},\: b_h:=\widehat{b}|_{X_h\times Y_h}.
\end{align}

Let the bilinear forms $a$ and $b$ be associated to the linear operators $A$ and 
$B\in L(X;Y^*)$, e.g., $Ax:=a(x,\bullet)\in Y^*$ for all $x\in X$. Suppose that the linear operator   
$\widehat{A}\in L(\widehat{X}; \widehat{Y}^*)$ (resp. 
$A+B\in L(X;Y^*)$)  associated to the  
bilinear form $\widehat{a}$ 
(resp.  $a+b$)  is invertible and 
\begin{align} 
0<\widehat{\alpha}&:=\inf_{\substack{\widehat{x}\in \widehat{X} \\ \|\widehat{x}\|_{\widehat{X}}=1}} \sup_{\substack{\widehat{y}\in \widehat{Y}\\ \|\widehat{y}\|_{\widehat{Y}}=1}}\widehat{a}(\widehat{x},\widehat{y});\label{dis_Ah_infsup}\\
0<\beta&:=\inf_{\substack{x\in X\\ \|x\|_{X}=1}} \sup_{\substack{y\in Y\\ \|y\|_{Y}=1}}(a+b)(x,y).\label{cts_infsup}
\end{align}
Suppose that three linear operators $P\in L(\widehat{Y};Y_h)$, $Q\in L(X_h; X)$, 
$\cC\in L(Y_h;Y)$ exist and  lead to parameters 
$\delta_1,\delta_2,\delta_3,\Lambda_4\ge 0$ in 
\begin{itemize}
\item[{\bf (H1)}]
$\displaystyle\delta_1:=\sup_{\substack{x_h\in X_h\\ \|x_h\|_{X_h=1}}} \sup_{\substack{y_h\in Y_h\\ \|y_h\|_{Y_h}=1}}\widehat{a}\left(A^{-1}\left(\widehat{b}(x_h,\bullet)|_{Y}\right),y_h-\cC y_h\right);$
\item[{\bf (H2)}]
$\displaystyle\delta_2:=\sup_{\substack{x_h\in X_h\\ \|x_h\|_{X_h=1}}} \sup_{\substack{\widehat{y}\in \widehat{Y}\\ \|\widehat{y}\|_{\widehat{Y}}=1}}\widehat{a}\left(x_h{ + }A^{-1}\left(\widehat{b}(x_h,\bullet)|_{Y}\right),\widehat{y}-P\widehat{y}\right);$
\item[{\bf (H3)}]
$\displaystyle\delta_3:=\sup_{\substack{x_h\in X_h\\ \|x_h\|_{X_h}=1}}\left\|\widehat{b}\big{(}x_h,(1-\cC)\bullet\big{)}\right\|_{ Y_h^*};$
\item[{\bf (H4)}] $\exists\Lambda_4<\infty\,\forall x_h\in X_h\;\;
\|(1-Q)x_h\|_{\widehat{X}}\leq \Lambda_4\;{\rm dist}_{\|\bullet\|_{\widehat{X}}}\left(x_h,X\right).$
\end{itemize}
Abbreviate the bound 
$\| \widehat{b}\|_{\widehat{X}\times Y^*}$ 
 of the bilinear form $ \widehat{b}|_{\widehat{X}\times Y^*}$
 simply by  $\| \widehat{b}\|$
 and  set $\|a\|:=\|A\|_{L(X;Y^*)}$  as well as 
 $\|A^{-1}\|:=\|A^{-1}\|_{L(Y^*;X)}$ --- whenever there is no risk of confusion 
 (e.g. with the $L^2$ norm $\|\bullet\|$ of a Lebesgue function). 
If  {\bf (H4)} holds with $0\leq\Lambda_4<\infty$, set
\begin{equation}\label{AsmpCondn}
\widehat{\beta}:=\frac{\beta}{\Lambda_4\beta+\|a\|\left(1+\Lambda_4\left(1+\|A^{-1}\|\|\widehat{b}\|\right)\right)}>0.
\end{equation}
In the applications discussed in this paper, $\delta_1+\delta_2+\delta_3$ 
from {\bf (H1)}-{\bf (H3)} will be smaller than $\widehat{\alpha}\widehat{\beta}$ so that the subsequent result provides a discrete inf-sup condition with $\beta_h>0$.

\begin{thm}[discrete inf-sup condition]\label{dis_inf_sup_thm}
Under the aforementioned notation, \eqref{dis_Ah_infsup}-{\eqref{AsmpCondn}} 
and {\bf (H1)}-{\bf (H4)} imply 
\begin{equation}\label{eqdis_inf_sup_defbetah}
\widehat{\alpha}\widehat{\beta}-\left(\delta_1+\delta_2+\delta_3\right)\leq \beta_h:=\inf_{\substack{x_h\in X_h\\ \| x_h\|_{X_h}=1}}\sup_{\substack{y_h\in Y_h\\ \| y_h\|_{Y_h}=1}} 
(a_h+b_h)(x_h,y_h).
	\end{equation}
\end{thm}

\begin{proof}
Given any $x_h\in X_h$ with $\| x_h\|_{X_h}=1$, define $$x:=Qx_h,\: \xi=A^{-1}\left(\widehat{b}(x_h,\bullet)|_Y\right)\in X,\text{ and }\eta=A^{-1}\left(b(x,\bullet)|_Y\right)\in X .$$ 

\noindent The inf-sup condition \eqref{cts_infsup} and $A\eta=Bx$ lead to
\begin{equation*}
\beta\|x\|_X\leq \|Ax+Bx\|_{Y^*}=\|A(x+\eta)\|_{Y^*}\leq \|a\|\|x+\eta\|_{X}.
\end{equation*}
This and triangle inequalities imply
\begin{equation}\label{xbound}
{\beta}/{\|a\|}\,\|x\|_X\leq\|x+\eta\|_{X}\leq \|x-x_h\|_{\widehat{X}}+\|x_h+\xi\|_{\widehat{X}}+\|\eta-\xi\|_X.
\end{equation}
The definition of $\xi$ and $\eta$, the boundedness of the operator $A^{-1}$ and 
of the bilinear form $\widehat{b}|_{\widehat{X}\times Y}$ show
\begin{equation}\label{xi_eta_bdd}
\|\xi-\eta\|_X=\|A^{-1}\left(\widehat{b}(x-x_h,\bullet)|_Y\right)\|_{X}\leq \|A^{-1}\|\|\widehat{b}\|\|x-x_h\|_{\widehat{X}}.
\end{equation}
The combination of \eqref{xbound}-\eqref{xi_eta_bdd} reads
\begin{align}\label{xfirst_bdd}
{\beta}/{\|a\|}\,\|x\|_{X}\leq \|x_h+\xi\|_{\widehat{X}}+\left(1+\|A^{-1}\|\|\widehat{b}\|\right)\|x-x_h\|_{\widehat{X}}.
\end{align}
Since {\bf (H4)} implies 
\begin{equation}\label{A4app}
\| x-x_h\|_{\widehat{X}}\leq\Lambda_4\| x_h+\xi\|_{\widehat{X}},
\end{equation}
the estimate \eqref{xfirst_bdd} results in 
\begin{align}
\|x\|_X\leq {\|a\|}/{\beta}\left(1+\Lambda_4\left(1+\|A^{-1}\|\|\widehat{b}\|\right)\right)\|x_h+\xi\|_{\widehat{X}}\label{xnew_bdd}.
\end{align}
The triangle inequality and \eqref{A4app}-\eqref{xnew_bdd} lead to
\begin{align*}
1=\|x_h\|_{X_h}&\leq\|x-x_h\|_{\widehat{X}}+\|x\|_X\\
&\leq \left(\Lambda_4+{\|a\|}/{\beta}\left(1+\Lambda_4\left(1+\|A^{-1}\|\|\widehat{b}\|\right)\right)\right)\|x_h+\xi\|_{\widehat{X}}.
\end{align*}
With the definition of $\widehat{\beta}$ in \eqref{AsmpCondn}, this reads
\begin{equation}\label{beta_bdd}
\widehat{\beta}\leq \|x_h+\xi\|_{\widehat{X}}.
\end{equation}

For given $x_h+\xi\in\widehat{X}$ and for any $0<\epsilon<\widehat{\alpha}$, the inf-sup condition \eqref{dis_Ah_infsup} implies the existence of some $\widehat{y}\in \widehat{Y}$ with $\|\widehat{y}\|_{\widehat{Y}}=1$ and
\begin{equation}\label{alpha_bdd}
(\widehat{\alpha}-\epsilon)\|x_h+\xi\|_{\widehat{X}}\leq \widehat{a}(x_h+\xi,\widehat{y})=\widehat{a}(x_h+\xi,\widehat{y}-P\widehat{y})
+\widehat{a}(x_h+\xi,P\widehat{y}).
\end{equation}
Since $\widehat{a}(\xi,\cC y_h)=\widehat{b}(x_h,\cC y_h)$ for $y_h:=P\widehat{y}$, the {latter} term is equal to
\begin{equation*}
\widehat{a}(x_h+\xi,y_h)=a_h(x_h,y_h)+b_h(x_h,y_h)+\widehat{a}(\xi,y_h-\cC y_h)+\widehat{b}(x_h,\cC y_h-y_h).
\end{equation*}
Let $\gamma_h:=a_h(x_h,y_h)+b_h(x_h,y_h)$, then {\bf (H1)}-{\bf (H3)} and \eqref{alpha_bdd} lead to
\begin{align}\label{bdd_deltas}
\widehat{a}(x_h+\xi,\widehat{y})\leq \gamma_h+\delta_1+\delta_2+\delta_3.
\end{align}
The combination of \eqref{beta_bdd}-\eqref{bdd_deltas} and 
$\epsilon\searrow0$ in the end result in
\begin{equation*}
\widehat{\alpha}\widehat{\beta}-(\delta_1+\delta_2+\delta_3)\leq \gamma_h\leq \|a_h(x_h,\bullet)+b_h(x_h,\bullet)\|_{Y_h^*}.
\end{equation*}
The last estimate holds for an arbitrary $x_h$ with $\|x_h\|_{X_h}=1$ and so proves the discrete inf-sup condition  $\widehat{\alpha}\widehat{\beta}-\left(\delta_1+\delta_2+\delta_3\right)\leq \beta_h$. 
\end{proof}

It is well known that  a  positive $\beta_h>0$ in  \eqref{eqdis_inf_sup_defbetah} implies the 
best-approximation for the Petrov-Galerkin scheme 
\cite{MR3097958,MR2373954,Braess,DiPetroErn12}
in the following sense.

\begin{cor}[best-approximation]\label{rembestapproximation}  
Suppose $(\widehat{X},\widehat{a})$ is a  Hilbert space  and 
$u\in X$, $u_h\in X_h$, and  $\widehat{F}\in \widehat{Y}^*$ satisfy 
$(a+b)(u,\bullet)=F:=\widehat{F}|_Y\in Y^*$ and 
$(a_h+b_h)(u_h,\bullet)=F_h:=\widehat{F}|_{Y_h}\in Y_h^*$. Then  
\[
\beta_h \|  u- u_h \|_{\widehat{X}} 
\le M \min_{x_h\in X_h}   \| u - x_h \|_{\widehat{X}}
+
\sup_{\substack{y_h\in Y_h\\ \| y_h\|_{Y_h}=1}}
\left(F_h(y_h)- (\widehat{a}+\widehat{b})( u ,y_h)\right)
\]
with the bound $M:=  \|  \widehat{a}+\widehat{b} \|_{ \widehat{X} \times Y_h }  
\le \|  \widehat{a}+\widehat{b} \|$ of the bilinear form 
$ (\widehat{a}+\widehat{b})|_{\widehat{X} \times Y_h}$.
\end{cor}

The proof of the quasi-optimal convergence 
for a stable discretisation is nowadays standard  in all finite element textbooks 
in the context of the Strang-Fix lemmas.  

\subsection{Second-order linear non-selfadjoint \\  and indefinite elliptic problems}%
\label{sec:SelfAdjoint}
This subsection applies {\bf (H1)}-{\bf (H4)} to second-order linear self-adjoint and indefinite elliptic problems and establishes a priori estimates for conforming and nonconforming FEMs under more general conditions on the smoothness of the coefficients of the elliptic operator and for $\Omega \subset {\mathbb R}^n$  vis-\`{a}-vis \cite{CCADNNAKP15}.

\subsubsection{Mathematical model}\label{subsectMathematicalmodelinSection2}
The strong form of a second-order problem with $L^\infty$ coefficients ${\bf A},$  {\bf b} and $\gamma$  
reads: Given $f\in L^2(\Omega)$ seek $u\in V:=H^1_0(\Omega)$ such that 
\begin{eqnarray}\label{eq1}
 \mathcal{L}u :=  -\nabla \cdot (\boldmath A\nabla u+u {\mathbf b})+ \gamma \: u=f.
\end{eqnarray}
The coefficients ${\bf A}\in L^\infty(\Omega;\bS),\; 
{\mathbf b}\in L^\infty(\Omega;\bR^n),\gamma\in L^\infty(\Omega)$ 
satisfy $0<\underline{\lambda}\leq \lambda_1({\bf A}(x))\leq \cdots \leq \lambda_n({\bf A}(x))
\leq \overline{\lambda}<\infty$ for the eigenvalues $\lambda_j({\bf A}(x))$ 
of the SPD ${\bf A}(x)$  for a.e. $x\in \Omega$. 

\noindent For $u,v\in V,$ the expression
\begin{equation}\label{defn_a}
a(u,v):=\int_{\Omega}({\bf A}\nabla u)\cdot\nabla v\dx
\end{equation}
defines a scalar product on $V$ (and $V$ is endowed with this scalar product in the sequel)
equivalent to the standard scalar product in the sense { that} 
the $H^1$-seminorm $|\bullet|_{H^1(\Omega)}:=\|\nabla\bullet\|$ in $V$ satisfies
\begin{equation}\label{lambda_bdd}
\underline{\lambda}^{1/2}| \bullet|_{H^1(\Omega)} \leq \|\bullet\|_{a}
:=a(\bullet,\bullet)^{1/2}\leq \overline{\lambda}^{1/2}| \bullet|_{H^1(\Omega)}.
\end{equation}
Given the bilinear form $b:V\times V\to \bR$ with
\begin{align}\label{defn_b}
b(u,v)&:=\int_{\Omega}\left(u {\mathbf b}\cdot\nabla v+\gamma uv\right)\dx \; {\rm for \; all \;  } u,v \in V
\end{align}
and the linear form $F\in L^2(\Omega)^*\subset H^{-1}(\Omega)=: V^*$ with $F(v):=\int_{\Omega}fv\dx$ for all $v\in V$, the weak formulation of \eqref{eq1} seeks the solution $u\in V$ to
\begin{align}\label{WeakFrmSec}
(a+b)(u,v):=a(u,v)+b(u,v)=F(v)\quad\text{ for all } v\in V.
\end{align}
In the absence of further conditions on the smoothness of the coefficients, any 
higher regularity of the weak solution $u\in H^1_0(\Omega)$ \eqref{eq1}
in the form  $u\in H^s(\Omega)$ for any $s>1$ 
 is not guaranteed even for  $f\in C^\infty(\Omega)$
\cite[p. 20]{SchatzWang96}.

\subsubsection{Triangulations}
Throughout this paper, $\bT$ is a set of shape-regular triangulations of the polyhedral 
bounded Lipschitz domain $\Omega\subset \bR^n$ into simplices.
Given an initial  triangulation $\cT_0$ of 
$\Omega$, let
 the newest-vertex bisection define local  mesh-refining that leads to a set of  
shape-regular triangulations $\cT\in\bT$. 

Shape-regularity means that there exists  a universal constant 
$\kappa>0$  such that the maximal diameter ${\rm diam}(B)$ 
of a ball $B\subset K$ satisfies 
$\kappa\, h_K\leq {\rm diam}(B)\leq {\rm diam}(K)=:h_K$ for any $K\in\cT \in \bT$.
Given $\cT\in\bT$, let $h_{\cT}\in P_0(\cT)$ be piecewise constant with $h_{\cT}|_K=h_K
={\rm diam} (K)$ for $K\in\cT$ and let $h_{\max} :=h_{\max}(\cT):=\max h_{\cT}$;  recall  
$\bT(\delta):=\left\{\cT\in\bT :  h_{\max}(\cT)\le  \delta\right\}$ for any  $\delta >0$.

The set of all sides of {the shape-regular triangulation $\cT$ of 
$\Omega$ into simplices} is denoted by $\cE$.
The set of all internal vertices (resp. boundary vertices) and  interior sides (resp.  boundary sides) 
of $\cT$ are denoted by $\cN (\Omega)$ (resp.  $\cN(\partial\Omega)$) and 
$\cE (\Omega)$ (resp. $\cE (\partial\Omega)$).

\subsubsection{Conforming FEM}
Let $P_1(\cT)$ denote the piecewise affine functions in $L^\infty(\Omega)$ with respect to the triangulation $\cT$ so that the associated $P_1$ conforming finite element function spaces without and with (homogeneous) boundary conditions read
\begin{align*}
S^1(\cT):=P_1(\cT)\cap C(\bar\Omega)\text{ and } S^1_0(\cT):=\left\{v_C\in S^1(\cT):v_C=0 \text{ on }\partial\Omega\right\}.
\end{align*}
The interior nodes $\cN(\Omega)$ label the nodal basis functions $\varphi_z$ with patch {$\omega_z=\{\varphi_z>0\}={\rm int}( {\rm supp}\,\varphi_z)$} around $z\in\cN(\Omega)$.

Given some finite-dimensional finite element space $V_h$ with 
$S_0^1(\cT)\subseteq V_h\subset V\equiv H^1_0(\Omega)$, the discrete formulation 
of \eqref{WeakFrmSec} seeks $u_h\in V_h$ with 
\begin{equation}\label{weakdisSec}
a(u_h,v_h)+b(u_h,v_h)=F(v_h)\text{ for all }v_h\in V_h . 
\end{equation}

The arguments of \cite{SchatzWang96}  are rephrased in the following lemma
(proven in the appendix) 
that allows the application of Theorem~\ref{dis_inf_sup_thm}  in the subsequent 
theorem. 

\begin{lem}\label{SpecLem}
For any $\epsilon>0$ there exists some $\delta>0$ such that 
the solution $z\in V\equiv H^1_0(\Omega)$ to $a(z,\bullet)=g\in \lt\subset H^{-1}(\Omega)$ 
for  $g\in \lt$  satisfies, for all $\cT\in\bT(\delta)$, that 
\[
\min_{z_C\in S^1_0(\cT)} \| z-z_C\|_a
+ \min_{Q_0\in P_0(\cT;\bR^n)}\| A\nabla z - Q_0  \|
\leq\epsilon \|g\|.\]
\end{lem}

\begin{thm}
Adopt the aforementioned assumptions on $a$ and $b$  in 
\eqref{defn_a} -\eqref{defn_b} and suppose that \eqref{eq1} 
is well-posed in the sense that it allows for a unique solution $u$ 
for all right-hand sides $f\in L^2(\Omega)$. Then 
\begin{equation*}
0<\beta:=\inf_{\substack{x\in V\\ \| x_h\|_{a}=1}}
\sup_{\substack{y\in V \\ \| y\|_{a}=1}}(a+b)(x,y)
\end{equation*}
and for any positive  $\beta_0<\beta$, there exist $\delta>0$ such that
\begin{equation*}
\beta_0\leq\beta_h:=\inf_{\substack{x_h\in V_h\\ \| x_h\|_{a}=1}}
\sup_{\substack{y_h\in V_h \\ \| y_h\|_{a}=1}}(a+b)(x_h,y_h)
\end{equation*}
holds for all $S_0^1(\cT)\subset V_h:=X_h=Y_h\subset V$ with respect to $\cT\in\bT(\delta)$. Moreover,  the solution $u$ to \eqref{eq1} and $u_h$ to  \eqref{weakdisSec} satisfy 
\begin{align}\label{Kato_arg}
\| u-u_h\| _a\leq \frac{\|a+b\|}{\beta_0}\min_{v_h\in V_h}\| u-v_h\|_a.
\end{align}
\end{thm}

\begin{proof}
The invertibility of a linear operator from one Banach space into the dual 
of another is equivalent to an $\inf$-$\sup$ condition 
\cite{MR3097958,MR2373954,Braess,DiPetroErn12}; in particular, 
the well-posedness of the theorem implies $\beta>0$. The remaining assertions 
follow from Theorem~\ref{dis_inf_sup_thm} with $\widehat{a}=a$, $\widehat{b}=b$, 
$S_0^1(\cT)\subset V_h=X_h=Y_h \subset 
X=Y=V=H^1_0(\Omega) $ endowed with the norm $\| \bullet \|_a$. 
Then  $\alpha=\widehat{\alpha}=1=\| \widehat{a}\| $ and $\beta$ is the constant in  \eqref{cts_infsup}.

To conclude the discrete inf-sup condition, it is sufficient to verify that the parameters involved in {\bf (H1)}-{\bf (H4)} can be chosen such that the discrete inf-sup constant in Theorem~\ref{dis_inf_sup_thm} is positive.  Moreover, the discrete inf-sup constants of $a+b$ are equal to those of the dual problem $a+ b^*$ with $b^*(u,v)=b(v,u)$. Therefore,
Theorem~\ref{dis_inf_sup_thm} is applied to $a$ and $b^*$ (rather than $a$ and $b$).

 Let $Q$ and $\cC$ be the identity, while $P\in L(V;V_h)$ denotes
  the Galerkin projection onto $V_h$ with respect to $a$, i.e. $a(v-Pv,\bullet)=0$ in $V_h$ for all $v\in V$. Then {the parameters in}  {\bf (H1)}, {\bf (H3)}, and {\bf (H4)} are $\delta_1=\delta_3=\Lambda_4=0$. 
The {choice} of {the parameter $\delta_2$ in} {\bf (H2)} concerns $v\in V$ and 
$u_h\in V_h$ with $\| v\|_a=1=\| u_h\|_a$ and the solution 
$z:=A^{-1}(b^*(u_h,\bullet))\in V$ to $a(z,\bullet)=b(\bullet, u_h)$. 
Notice that $g:={\bf b}\cdot\nabla u_h+\gamma u_h\in\lt$ satisfies
 \[
 b(\varphi, u_h)=\int_{\Omega}\left({\bf b}\cdot\nabla u_h+\gamma u_h\right)\varphi\dx=\integ g\varphi\dx  \quad {\rm for \; all \;} \varphi \in V
 \]
and (with the Friedrichs constant $C_{F} $ for $\|\bullet\|  
\le \underline{\lambda}^{-1/2}C_F\| \bullet\|_a$ in $V$)
\begin{align}
\|g\|\leq \| u_h\|_a \,\|{\bf A}^{-1/2} {\bf b}\|_{\infty}+\|\gamma\|_{\infty}\|u_h\|
\leq   (\|{\bf b}\|_{\infty}+C_F\|\gamma\|_{\infty})\underline{\lambda}^{-1/2}
=: C
.\label{Frid}
\end{align}
The Galerkin orthogonality with  $P$, the definition of $z$,  and a Cauchy 
inequality with $\| v\|_a=1$ in the end show
\[
a\left(u_h+A^{-1}(b^*(u_h,\bullet)),v- P v\right)=a(z,v- P v)=a(z- P z,v) \le \|z-Pz\|_a.
\]
Given any $\epsilon>0$, Lemma~\ref{SpecLem} leads to  {$\delta >0$} such that
for all $\cT\in\bT(\delta)$ there exists some  $z_C \in S_0^1(\cT)$  with 
\[
  \|z-Pz\|_a\le   \| z-z_C\|_a \leq \epsilon \|g\|\le  \epsilon C
\]
with  \eqref{Frid}  in the last step. The combination of the previous inequalities 
proves  {\bf (H2)} with  ${\delta_2} := \epsilon C$.

Theorem~\ref{dis_inf_sup_thm} applies with $\beta=\widehat{\beta}$ and
$\beta_h\ge \beta- \epsilon C$. This proves the assertion  on $\beta_h\ge\beta_0$
for sufficiently small $\epsilon$ and $\delta$. 

The quasi-optimal convergence \eqref{Kato_arg} follows from  
Corollary~\ref{rembestapproximation} without the second term in the conforming
discretisation.
\end{proof}

\begin{rem}
The proof requires that the discrete space $V_h$ solely satisfies $S_0^1(\cT)\subset V_h\subset H^1_0(\Omega)$ and so allows for conforming $hp$ finite element  spaces. 
The condition $\cT\in\bT(\delta)$ allows for local mesh-refining as long as
max $h_0$ is sufficiently small.
\end{rem}

\subsubsection{Nonconforming FEM}
This subsection establishes the {\it first} best-approximation-type a~priori 
error estimate for the  lowest-order nonconforming
FEM in any space dimension $\ge 2$ under the 
assumptions on the coefficients of Subsubsection~\ref{subsectMathematicalmodelinSection2}
as an application of Theorem~\ref{dis_inf_sup_thm}.
This  generalises  \cite[Thm 3.3]{CCADNNAKP15} from piecewise Lipschitz
continuous to $L^\infty$ coefficients. 

The nonconforming Crouzeix-Raviart (CR) finite element spaces read
\begin{eqnarray*}
&& CR^1(\mathcal T): =\{v \in P_1(\mathcal{T}): \forall E \in \mathcal E,~ v~ \text{ is continuous at mid($E$) }\},\\
&& CR^1_0(\mathcal T):=\{v \in CR^1(\mathcal T):v(\text{mid} (E))
=0 ~~\text{for all} ~E \in \mathcal E ({\partial\Omega}) \}.
\end{eqnarray*}
Here $\text{mid} (E)$ denotes the  mid operator for a simplex obtained by taking the arithmetic mean of all vertices.
The CR finite element spaces give rise to the  bilinear forms 
 $a_{\text{pw}},b_{\text{pw}}:CR_0^1(\cT)\times CR_0^1(\cT)\to \bR$ defined, for all $u_{\text{CR}},v_{\text{CR}}\in CR_0^1(\cT)$, by
\begin{align}
&a_{\text{pw}}(u_{\text{CR}},v_{\text{CR}}):=\sum_{T\in\cT}\int_T(\mathbf A \nabla u_{\text{CR}})\cdot\nabla v_{\text{CR}}\dx,\\
&b_{\text{pw}}(u_{\text{CR}},v_{\text{CR}}):=\sum_{T\in\cT}\int_T\left(u_{\text{CR}} {\mathbf b}\cdot\nabla v_{\text{CR}}+\gamma u_{\text{CR}}v_{\text{CR}}\right)\dx.
\end{align}
The nonconforming FEM seeks the discrete solution $u_{\text{CR}}\in CR_0^1(\cT)$ to
\begin{align}\label{dis_weak_Se}
a_{\text{pw}}(u_{\text{CR}},v_{\text{CR}})+b_{\text{pw}}(u_{\text{CR}},v_{\text{CR}})=F(v_{\text{CR}})\fl v_{\text{CR}}\in CR_0^1(\cT).
\end{align}
Notice that $\|\nabla_{\text{pw}} \bullet \|$ with the piecewise action 
$\nabla_{\text{pw}}$ of the gradient $\nabla $ is a norm on $CR_0^1(\cT)$ and so is 
$\trinl \bullet\trinr_{\text{pw}}:=\|  {\bf A}^{1/2} \nabla_{\text{pw}}\bullet\|$.
The subsequent theorem implies the unique solvability and boundedness of discrete solutions for sufficiently fine meshes. 

\begin{thm}\label{dis_inf_sup_Se}
Suppose that $\cL$ is a bijection and so $\cL^{-1}$ is bounded and 
\eqref{cts_infsup} 
holds with $\beta=\|\cL^{-1}\|>0$. Then 
there exist positive $\delta$ and $\beta_0$  such that any $\cT\in\bT(\delta)$ satisfies
\begin{align}\label{eqnewlabelccforinfsupbeta0}
\beta_{0}\leq \beta_h:=\inf_{\substack{w_{{\rm CR}}\in CR_0^1(\cT)\\ \trinl w_{{\rm CR}}\trinr_{{\rm pw}}=1}} \sup_{\substack{ v_{{\rm CR}}\in CR_0^1(\cT)\\ 
\trinl v_{{\rm CR}}\trinr_{{\rm pw}}=1}}(a_{{\rm pw}}+b_{{\rm pw}})(w_{{\rm CR}},v_{{\rm CR}}).
\end{align}
\end{thm}

\begin{proof}
Let $\displaystyle H^1(\cT):=\left\{v\in L^2(\Omega)\,|\,\forall\,T\in\cT,\, v|_{T}\in H^1(T)\right\}$ and endow the vector space
$$\widehat{V}:=\widehat{X}:=\widehat{Y}:=\left\{\widehat{v}\in H^1(\cT)\,|\,\forall\,E\in\cE, \int_E[\widehat{v}]_E\ds=0\right\}\supset V+CR_0^1(\cT)$$
with the norm
$\trinl\bullet\trinr_{\text{pw}}:=\|{\bf A}^{1/2}\nabla_{\text{pw}}\bullet\|$. 
Here and throughout the paper, the jump of $\widehat{v}\in \widehat{V}$ 
across any interior face 
$E=\partial K_+\cap \partial K_-\in \cE(\Omega)$ shared by 
two simplices $K_+$ and $K_-$  reads
\begin{align*}
[\widehat{v}]_E:=
\widehat{v}|_{K_+}-\widehat{v}|_{K_-}\quad \text{on } E=\partial K_+\cap \partial K_-
\end{align*}
(then  $\omega_E:=\text{int}( K_+\cup K_-)$),  
while $[\widehat{v}]_E:=\widehat{v}|_E$ along any boundary face 
$E\in \cE(\partial\Omega)$ according to the homogeneous boundary condition on 
$\partial\Omega$ (and then 
$\omega_E:=\text{int}(K)$ for $K\in\cT$ with $E\in \cE(K)$).

\noindent The boundedness of $\widehat{a}+\widehat{b}$ follows from a piecewise Friedrichs inequality
\[
\|  \hat v \| \le C_\text{pwF} \left( \sum _{E\in\cE}  |  \omega_E|^{-1} \,  |\int_E [\hat v]ds |^2  
+ \| \nabla_{\text{pw}} \hat v \|^2 \right)^{1/2}
\]
known for all $\hat v\in \widehat{V}$ with the volume $|\omega_E|$ of the side-patch
$|\omega_E|\approx h_E^n$. For  $\hat v\in \widehat{V}$ and   $E\in\cE$, 
the integral  $\int_E [\hat v]ds=0$ vanishes; hence the piecewise Friedrichs 
inequality reduces to $\|  \hat v \| \le C_\text{pwF} \| \nabla_{\text{pw}} \hat v \|$.
This enables a proof that $(\widehat{V},\widehat{a})$ is a Hilbert space that 
$\widehat{b}$ is a  bounded bilinear form with respect to those norms. Consequently,
$\widehat{\alpha}=1=\|\widehat{a}\|$ and 
 \eqref{cts_infsup}  holds with some $\beta=\|\cL^{-1}\|>0$.

Define the nonconforming 
interpolation operator  $I_{\text{CR}}\in L(\widehat{V};CR_0^1(\cT))$   by 
\begin{align}\label{integ_mean}
I_{\text{CR}}v &:= \sum_{E \in {\mathcal E}}  \left(\fint_E v~ds\right) \psi_E 
\quad\text{for all }v \in \widehat{V}
\end{align}
with the side-oriented basis functions $\psi_E$ of $CR_0^1(\cT)$ with  
$\psi_E({\rm mid}(F))=\delta_{EF}$, the Kronecker symbol,   for all sides
$E,F \in {\mathcal E}.$
 For any $v_{\text{CR}}\in CR_0^1(\cT)$, the conforming companion operator $Q:=\cC:= J \in L\left(CR_0^1(\cT);V\right) $ with  $J  v_{\text{CR}}\in P_4(\cT)\cap C^0(\bar\Omega)$ from \cite[p. 1065]{CCDGMS15} satisfies (a) that $w:=v_{\text{CR}}-J  v_{\text{CR}}\perp P_1(\cT)$ 
 is $L^2$ orthogonal to the space $P_1(\cT)$ of piecewise first-order polynomials, 
 (b)  the integral mean property of the gradient
\begin{equation}\label{integ_mean_ortho}
\Pi_0\left(\nabla_{\text{pw}}(v_{\text{CR}}-J v_{\text{CR}})\right)=0,
\end{equation}
and  (c) the approximation and stability property  (with a universal constant 
$\Lambda_{\text{CR}}$) 
\begin{equation}\label{J4_enrich}
\|h_{\cT}^{-1}(v_{\text{CR}}-J v_{\text{CR}})\|  
+\| \nabla_{\text{pw}}(  v_{\text{CR}}-J v_{\text{CR}}) \|  \le
\Lambda_{\text{CR}}  \min_{v\in H^1_0(\Omega)}\|\nabla_{\text{pw}} ( v_{\text{CR}}-v) \|.
\end{equation}
(The proofs in \cite{CCDGMS15} are in 2D, but can be generalised to any dimension). Note that $J $ is a right inverse to $I_{\text{CR}}$ in the sense that $I_{\text{CR}}J  v_{\text{CR}}=v_{\text{CR}}$ holds for all $v_{\text{CR}}\in CR_0^1(\cT)$. The inequality \eqref{J4_enrich} implies {\bf (H4)}
with $\Lambda_4=(\overline{\lambda}/\underline{\lambda})^{1/2}\Lambda_{\text{CR}}$.

The bilinear forms  
$\widehat{a}\equiv a_{\text{pw}}, \widehat{b}\equiv b_{\text{pw}}:\widehat{V}\times \widehat{V}\to\bR$
read, for all $\widehat{u},\widehat{v}\in\widehat{V}$, as 
\begin{equation} \label{ab}
\widehat{a}(\widehat{u},\widehat{v}):=\sum_{T\in\cT}\int_T(\mathbf A \nabla \widehat{u})\cdot\nabla \widehat{v}\dx\quad \text{and} \quad
\widehat{b}(\widehat{u},\widehat{v}):=\sum_{T\in\cT}\int_T(\widehat{u} {\mathbf b}\cdot\nabla\widehat{v} +\gamma \widehat{u}\widehat{v})\dx.
\end{equation}
As in the stability proof of the conforming FEM, Theorem~\ref{dis_inf_sup_thm} applies
to $\widehat{a}$ and  $\widehat{b}^*$ (rather than to $\widehat{a}$ and  $\widehat{b}$).

The proof of {\bf (H1)} concerns  $u_{\text{CR}},v_{\text{CR}}\in CR_0^1(\cT)$ with 
$\trinl u_{\text{CR}}\trinr_{\text{pw}}=1=\trinl v_{\text{CR}}~\trinr_{\text{pw}}$ and  the solution 
$z=A^{-1}(\widehat{b}(\bullet, u_{\text{CR}}) |_{V})\in V$ to 
$a(z,\bullet)= \widehat{b}(\bullet, u_{\text{CR}})$ in $V$. The right-hand side is the $L^2$
scalar product of the test function in $V$ with
$g:={\bf b}\cdot\nabla_{\text{pw}} u_{\text{CR}}+\gamma u_{\text{CR}} \in\lt$ bounded with 
the discrete Friedrichs inequality $\|\bullet\| \leq C_{dF}\|\nabla_{\text{pw}}\bullet\|$
in $CR_0^1(\cT)$ \cite[p. 301]{Brenner} by
 $\|g\|\leq (\|{\bf b}\|_{\infty} +C_{dF}  
 \|{\gamma}\|_{\infty})\underline{\lambda}^{-1/2}=: C_0$.
 Since $\nabla_{\text{pw}} w\perp  P_0(\cT;\bR^n)$  in $L^2(\Omega;\bR^n)$ for 
 $w:= v_{\text{CR}}-Jv_{\text{CR}}$, 
 Lemma~\ref{SpecLem} applies for any $\epsilon>0$ 
 and leads to $\delta>0$ so that, for $\cT\in\bT(\delta)$ 
 with the $L^2$ projection $\Pi_0$,   
\begin{align}
&\widehat{a}\left(A^{-1}\left(\widehat{b}^*(u_{\text{CR}},\bullet)|_V\right),v_{\text{CR}}-J v_{\text{CR}}\right)=
a_{\text{pw}}(z,w)  =\int_{\Omega} ((1-\Pi_0){\bf A}\nabla z )\cdot\nabla_{\text{pw}}w \dx  \notag\\
&\quad\leq  \| (1-\Pi_0){\bf A}\nabla z \|\,\| \nabla_{\text{pw}}w \|   
\le \epsilon \|g \|  \Lambda_{\text{CR}} \underline{\lambda}^{-1/2} \le C_1\epsilon =:\delta_1
\label{H1assp}
\end{align}
with \eqref{J4_enrich} for $v=0$ and $\| \nabla_{\text{pw}} v_{\text{CR}}\|\le  
\underline{\lambda}^{-1/2}  \trinl v_{\text{CR}} \trinr_{\text{pw}} =\underline{\lambda}^{-1/2}$ 
in the end for $C_1:=  C_0  \Lambda_{\text{CR}} \underline{\lambda}^{-1/2}$. This 
concludes the proof of  {\bf (H1)}. 

The proof of {\bf (H2)} concerns  $u_{\text{CR}}\in CR_0^1(\cT)$, 
 $\widehat{v}\in\widehat{V}$ with $\trinl u_{\text{CR}}\trinr_{\text{pw}}=1
 =\trinl \widehat{v} \trinr_{\text{pw}}$, 
 and   the solution  $z\in V$ to $a(z,\bullet)= \widehat{b}(\bullet, u_{\text{CR}})$ in $V$ as before.
The operator $P:\widehat{V}\to CR^1_0(\cT)$, however, is not $I_{\text{CR}}$ because 
the oscillating coefficients ${\bf A}$ prevent the immediate cancellation property for
$\widehat{a}(u_{\text{CR}},\widehat{v}-P\widehat{v} )=0$. The latter  is a consequence of 
the best-approximation $P$ in the Hilbert space $\widehat{V}$ onto its linear and closed subspace  $ CR_0^1(\cT)$;   so let   $P\widehat{v}\in CR_0^1(\cT) $ be the unique
minimiser in 
\[
 \trinl \widehat{v} -P\widehat{v} \trinr_{\text{pw}}=\min_{v_{\text{CR}}\in CR_0^1(\cT)}
 \trinl \widehat{v} - v_{\text{CR}}\trinr_{\text{pw}}\le  \trinl \widehat{v}  \trinr_{\text{pw}} =1.
\]
Lemma~\ref{SpecLem} applies for any $\epsilon>0$ and leads to $\delta>0$ so that, for 
each $\cT\in\bT(\delta)$, there exists some $z_C\in S^1_0(\cT)\subset CR_0^1(\cT)$ 
with $  \trinl z-z_C\trinr_{\text{pw}}  \le \epsilon C_0$. This,  
$\widehat{a}(u_{\text{CR}}+z_C,\widehat{v}-P\widehat{v} )=0$,  and 
$ \trinl \widehat{v}- P \widehat{v}  \trinr_{\text{pw}} \le 1$ in the end 
 provide  {\bf (H2)} (with $\widehat{b}^*$ replacing $\widehat{b}$):
\begin{align*}
&\widehat{a}(u_{\text{CR}}+ A^{-1}\left(\widehat{b}^*(u_{\text{CR}},\bullet)|_V\right), 
\widehat{v}-P\widehat{v}) \\
&=\int_{\Omega}({\bf A}\nabla_{\text{pw}}(z-z_C)\cdot\nabla_{\text{pw}}(\widehat{v}-P\widehat{v})\dx\notag\\
&\leq 
\trinl z-z_C\trinr_{\text{pw}} \trinl \widehat{v}- P \widehat{v}  \trinr_{\text{pw}} 
\le    C_0 \epsilon  =:\delta_2
\end{align*}
The proof of {\bf (H3)} concerns $u_{\text{CR}},v_{\text{CR}}\in CR_0^1(\cT)$ with $\trinl u_{\text{CR}}\trinr_{\text{pw}}=1=\trinl v_{\text{CR}}\trinr_{\text{pw}}$ and $w:= v_{\text{CR}}- J  v_{\text{CR}}$. 
This and  \eqref{J4_enrich} (with $v=0$) 
prove
\begin{align*}
&\widehat{b}^*\left(u_{\text{CR}},v_{\text{CR}}- J  v_{\text{CR}}\right)
=\int_{\Omega}\left({\bf b}\cdot\nabla_{\text{pw}}u_{\text{CR}}+\gamma u_{\text{CR}}\right) w\dx \\
&= \int_{\Omega} g w\dx \le  h_{\max} \|g\|\Lambda_{\text{CR}} \|\nabla_{\text{pw}} v_{\text{CR}}\|
\le C_0\Lambda_{\text{CR}}  \underline{\lambda}^{-1/2}\delta.
\end{align*}
Without loss of generality, assume $\delta\le\epsilon$. Then 
 {\bf (H3)} follows with  
 $\delta_3:= C_3\epsilon$ for 
 $C_3:=C_0\Lambda_{\text{CR}}  \underline{\lambda}^{-1/2}$.
 (It is remarkable that  in the last inequalities,  the extra property 
 $w:=v_{\text{CR}}-J  v_{\text{CR}}\perp P_1(\cT)$  leads to the bound 
 $  \underline{\lambda}^{-1/2}\Lambda_{\text{CR}}\text{osc}_1(g,\cT)$,
 but  that can easily be exploited solely for piecewise smooth 
 or at least piecewise continuous ${\bf b}$ and $\gamma$).
  
 Since  {\bf (H1)}-{\bf (H4)} hold for $\widehat{a}$ and  $\widehat{b}^*$,  
 Theorem~\ref{dis_inf_sup_thm} proves 
 $\beta_h\ge \widehat{\beta} - (C_1+C_2+C_3)\epsilon$
 with  positive $\widehat{\beta}<\beta$ defined in   \eqref{AsmpCondn}.
Any positive  $\epsilon<  \widehat{\beta}/(C_1+C_2+C_3) $ concludes the proof;
in fact, any  constant $\beta_0$ with $0<\beta_0< \widehat{\beta}$ can be realised
in \eqref{eqnewlabelccforinfsupbeta0}
by small $\delta>0$.
\end{proof}

The following best-approximation-type error estimate generalises a result
in  \cite{CCADNNAKP15}.

\begin{thm}\label{SA_err_est}
Let $u\in H^1_0(\Omega)$ solve \eqref{WeakFrmSec} and set $p:={\bf A}\nabla u+u{\bf b}\in H({\rm div},\Omega)$. There exists $\delta>0$ such that for all $\cT\in\bT(\delta)$, the discrete problem \eqref{dis_weak_Se} has a unique solution $u_{{\rm CR}}\in CR_0^1(\cT)$ and $u, \: u_{{\rm CR}}, \:p$ and its piecewise integral mean $\Pi_0 p$ satisfy
\begin{align}
\trinl u-u_{{\rm CR}}\trinr_{{\rm pw}}&\lesssim \trinl u-I_{{\rm CR}}u\trinr_{{\rm pw}}+\|p-\Pi_0p\|+{\rm {\rm osc}}_1(f-\gamma u,\cT).
\end{align}
\end{thm}
 
\begin{proof}
Given  $e_{\text{CR}}:=I_{\text{CR}}u-u_{\text{CR}}$, the discrete inf-sup condition of Theorem~\ref{dis_inf_sup_Se} implies the existence of $v_{\text{CR}}\in CR_0^1(\cT)$ with $\trinl v_{\text{CR}}\trinr_{\text{pw}}\le 1/\beta_0$ and
\begin{align}\label{bhapp3}
\trinl e_{\text{CR}}\trinr_{\text{pw}}
= a_{\text{pw}}(e_{\text{CR}},v_{\text{CR}})+b_{\text{pw}}(e_{\text{CR}},v_{\text{CR}}).
\end{align}
Recall from (a)-(b) in the proof of Theorem \ref{dis_inf_sup_Se} that 
$v:= J  v_{\text{CR}}$ satisfies $I_{\text{CR}}v=v_{\text{CR}}$ and $\Pi_1v=\Pi_1v_{\text{CR}}$. 
Since $a(u,v)=-b(u,v)+F(v)$ and $u_{\text{CR}}$ solves \eqref{dis_weak_Se},  $w:=v-v_{\text{CR}}$ 
satisfies  
\begin{align*}
a_{\text{pw}}(e_{\text{CR}},v_{\text{CR}})&=a_{\text{pw}}(u,v_{\text{CR}})-a_{\text{pw}}(u_{\text{CR}},v_{\text{CR}})\\
&=F(w) -a_{\text{pw}}(u,w)-b(u,v)+b_{\text{pw}}(u_{\text{CR}},v_{\text{CR}}).
\end{align*}
This leads in \eqref{bhapp3} to 
\begin{align*}
\trinl e_{\text{CR}}\trinr&= F(w) -a_{\text{pw}}(u,w)-b_{\text{pw}}(u,w) {-b_{\text{pw}}(u-I_{\text{CR}}u, v_{\text{CR}})}\notag\\
&=\int_{\Omega}(f-\gamma u) w\dx - \int_{\Omega} p\cdot\nabla_{\text{pw}}w\dx { -b_{\text{pw}}(u-I_{\text{CR}}u, v_{\text{CR}})}.
\end{align*}
Since $\nabla_{\text{pw}}w\,\bot\, P_0(\cT;\bR^n)$ in $L^2(\Omega;\bR^n)$ and $w\,\bot\, P_1(\cT;\bR^n)$, {an upper bound for the first two terms on the right-hand side is}
\begin{align}
&\integ(I-\Pi_1)(f-\gamma u)w\dx - \integ(p-\Pi_0 p)\cdot \nabla_{\text{pw}} w\dx\notag\\
&\leq \left(\|p-\Pi_0 p\|+{\rm osc}_1(f-\gamma u,\cT)\right)
\underline{\lambda}^{-1/2} \trinl w\trinr_{ {\text{pw}}}\notag\\
&\leq \Lambda_{\text{CR}}\underline{\lambda}^{-1/2}\beta_0^{-1}
\left(\|p-\Pi_0 p\|+{\rm osc}_1(f-\gamma u,\cT)\right).
\end{align}
This and a triangle inequality conclude the proof.
\end{proof}

{
\subsubsection{A modified CR-FEM for general right-hand sides} \label{gen_rhs}
The nonconforming scheme of the previous subsection allows for a right-hand side in $L^2(\Omega)$, 
while conforming variants directly apply to $f \in H^{-1}(\Omega)$. This subsection briefly discusses the modification for  $f \in H^{-1}(\Omega)$ and a surprising 
analog to Theorem \ref{SA_err_est} without  oscillation terms. 
Given  $f \in H^{-1}(\Omega)$,  the modified CR-FEM 
seeks  $\widetilde{u}_{\text{CR}}\in CR_0^1(\cT)$ such that 
\begin{align}\label{dis_weak_Se1}
 a_{\text{pw}}(\widetilde{u}_{\text{CR}},v_{\text{CR}})+\widetilde{b}_{\text{pw}}(\widetilde{u}_{\text{CR}},v_{\text{CR}})&=<f,Jv_{\text{CR}}>_{H^{-1}(\Omega) \times H^1_0(\Omega)} 
 \quad\text{for all } v_{\text{CR}}\in CR_0^1(\cT)
\end{align}
with the duality brackets $<\bullet ,\bullet >_{H^{-1}(\Omega) \times H^1_0(\Omega)} $ on the right-hand side of \eqref{dis_weak_Se1} acting on 
 $f \in H^{-1}(\Omega)$ and the test function $Jv_{\text{CR}}\in H^1_0(\Omega)$. The bilinear form $b(\bullet,\bullet)$ is replaced in \eqref{dis_weak_Se1} by a modification 
 $ \widetilde{b}_{\text{pw}}(\bullet,\bullet)$ defined, for 
 $u_{\text{CR}},v_{\text{CR}}\in CR_0^1(\cT)$,   by 
\begin{align}
&\widetilde{b}_{\text{pw}}(u_{\text{CR}},v_{\text{CR}}):=\sum_{T\in\cT}\int_T\left(u_{\text{CR}} {\mathbf b}\cdot\nabla_{\text{pw}}  v_{\text{CR}}+\gamma u_{\text{CR}}Jv_{\text{CR}}\right)\dx.
\end{align}
The difference to  $b_{\text{pw}} (u_{\text{CR}},v_{\text{CR}})$ from  \eqref{ab} is in the final 
application of $Jv_{\text{CR}}$ rather than $v_{\text{CR}}$ with the conforming companion
operator $J$ from the proof of Theorem~\ref{dis_inf_sup_Se}.

\begin{thm}[best approximation in modified CR-FEM]\label{SA_err_estmodified}
Let $u\in H^1_0(\Omega)$ solve \eqref{WeakFrmSec} with the right-hand side $f\equiv f\in H^{-1}(\Omega)$ and 
 set $p:={\bf A}\nabla u+u{\bf b}\in H({\rm div},\Omega)$. There exists $\delta>0$ such that for all $\cT\in\bT(\delta)$, the discrete problem \eqref{dis_weak_Se1} has a unique solution $\widetilde{u}_{\text{CR}}\in CR_0^1(\cT)$ 
 and $u, \: \widetilde{u}_{{\rm CR}}, \:p$ and its piecewise integral mean $\Pi_0 p$ satisfy
\[
\trinl u-\widetilde{u}_{{\rm CR}}\trinr_{{\rm pw}}\lesssim \trinl u-I_{{\rm CR}}u\trinr_{{\rm pw}}+\|p-\Pi_0p\|.
\]
\end{thm}

\begin{proof}
The stability of the modified bilinear form $ a_{\text{pw}}(\bullet, \bullet) +\widetilde{b}_{\text{pw}}(\bullet, \bullet)$ follows from the methodology of this section. An immediate 
proof follows from the stability \eqref{eqnewlabelccforinfsupbeta0} and a perturbation argument: For any $ v_{\text{CR}}, w_{\text{CR}}\in CR_0^1(\cT)$ with
$ \trinl v_{\text{CR}}\trinr_{\text{pw}}=1= \trinl w_{\text{CR}}\trinr_{\text{pw}}$, 
\[
 | b_{\text pw}(v_{\text{CR}},   w_{\text{CR}} ) - \widetilde{b}_{\text pw}(v_{\text{CR}},  w_{\text{CR}}) |
 \le \|\gamma\|_{\infty} \| v_{\text{CR}}\|\, \|   w_{\text{CR}} - J w_{\text{CR}}  \|
 \le \|\gamma\|_{\infty} C_\text{pwF}  \Lambda_{\text{CR}}    h_{\max}
 \]
with a piecewise  Friedrichs inequality and  \eqref{J4_enrich} in the last step. 
The combination with   \eqref{eqnewlabelccforinfsupbeta0} and a triangle inequality 
prove stability of the modified scheme
\[
\beta_0/2\le 
\widetilde{ \beta}_h:=\inf_{\substack{w_{\text{CR}}\in CR_0^1(\cT)\\ \trinl w_{\text{CR}}\trinr_{\text{pw}}=1}} \sup_{\substack{ v_{\text{CR}}\in CR_0^1(\cT)\\ 
\trinl v_{\text{CR}}\trinr_{\text{pw}}=1}}(a_{\text{pw}}+\widetilde{b}_{\text{pw}})(w_{\text{CR}},v_{\text{CR}})
\]
for any $\cT\in\mathbb{T}(   \beta_0/(2\|\gamma\|_{\infty} C_\text{pwF}  \Lambda_{\text{CR}}))$. 
The proof of the a priori error estimate follows the arguments of the proof of Theorem \ref{SA_err_est}.  Given  $e_{\text{CR}}:=I_{\text{CR}}u-\widetilde{u}_{\text{CR}}$, the stability of the modified scheme leads to some $v_{\text{CR}}\in CR_0^1(\cT) $ with norm
$\trinl v_{\text{CR}}\trinr_{\text{pw}}\le 2/\beta_0$ and 
\[
\trinl e_{\text{CR}}\trinr_{\text{pw}}
= a_{\text{pw}}( I_{\text{CR}}u- \widetilde{u}_{\text{CR}} ,v_{\text{CR}})+\widetilde{b}_{\text{pw}}(I_{\text{CR}}u- \widetilde{u}_{\text{CR}},v_{\text{CR}}). 
\]
Let  $v:=Jv_{\text{CR}}$ and $w:= v-v_{\text{CR}}$. Then   \eqref{WeakFrmSec}  and  \eqref{dis_weak_Se1}  imply  
\[ 
\trinl e_{\text{CR}}\trinr_{\text{pw}}
 =  a_{\text{pw}}( I_{\text{CR}}u-u ,v_{\text{CR}})  
 -  (p, \nabla_{\text{pw}} w)_{L^2(\Omega)} 
- (u- I_{\text{CR}}u ,   \gamma v +   b\cdot\nabla_{\text{pw}} v_{\text{CR}}  )_{L^2(\Omega)} .
\]
The point is that all terms with $v$ disappear and no oscillation terms remain. In fact, all other terms
are controlled by $\trinl u- I_{\text{CR}}u  \trinr_{\text{pw}}$ or by $\|p-\Pi_0 p\|$ as in the proof
of  Theorem \ref{SA_err_est}; further details are omitted.
\end{proof}}

\section[Semilinear problems with trilinear nonlinearity]{A class of 
semilinear problems with\\ trilinear nonlinearity} \label{error}
This section is devoted to an abstract framework for an a~priori and a~posteriori analysis to solve a class of semilinear problems that includes the applications in Section~4 and 5.

\subsection{A priori error control}
Suppose  $X$ and $Y$ are real Banach spaces and let the 
quadratic function $N:X\to Y^*$ be of the form
\begin{equation}\label{eqccdefN}
N(x):=\cL x+\Gamma(x,x,\bullet)
\end{equation}
with a leading linear operator $A\in L(X;Y^*)$ and  $F\in Y^*$ for the affine operator $\cL x:=Ax-F$ for all $x\in X$ and a bounded trilinear form $\Gamma: X\times X\times Y\to \bR$.

To approximate a regular $u$ solution to $N(u)=0$, the discrete version involves some discrete spaces $X_h$ and $Y_h$ plus  a discrete function  $F_h\in Y_h^*$, 
$\cL_hx_h:=A_hx_h-F_h$, and a 
{bounded trilinear form} $\Gamma_h:X_h\times X_h\times Y_h\to \bR$ with $N_h(x_h)=\cL_h x_h+\Gamma_h(x_h,x_h,\bullet)$. The discrete problem seeks 
$u_h\in X_h$ such that
\[\cL_hu_h+\Gamma_h(u_h,u_h,\bullet)=0\quad \text{ in } Y_h^*.\]
The $a~priori$ error analysis is based on the Newton-Kantorovich theorem and 
adapts  the abstract discrete inf-sup results of Subsection~\ref{sec:abs_result}.
Some further straightforward notation is required for this. Suppose that there exists some   invertible bounded linear operator 
operator $\widehat{A}$ (i.e. $\widehat{A}v=\widehat{a}(v,\bullet)$ in $\widehat{Y}$ for all $v\in\widehat{X}$) on extended Banach spaces $\widehat{X}$ and $\widehat{Y}$ and 
suppose that  there exists   a bounded  extension 
\[
\widehat{\Gamma}:\widehat{X}\times \widehat{X}\times \widehat{Y}\to \bR
\quad\text{with}\quad \| \widehat{\Gamma}\| 
:=\| \widehat{\Gamma}\|_{\widehat{X}\times \widehat{X}\times \widehat{Y}}:=
\sup_{\substack{\widehat{x}\in \widehat{X}\\ \|\widehat{x}\|_{\widehat{X}}=1}}
\sup_{\substack{\widehat{\xi}\in \widehat{X}\\ \|\widehat{\xi}\|_{\widehat{X}}=1}}
\sup_{\substack{\widehat{y}\in \widehat{Y}\\ \|\widehat{y}\|_{\widehat{Y}}=1}}
\widehat{\Gamma}(\widehat{x},\widehat{\xi},\widehat{y})<\infty 
\]
of $\Gamma=\widehat{\Gamma}|_{X\times X\times Y}$ with 
$\Gamma_h=\widehat{\Gamma}|_{X_h\times X_h\times Y_h}$.
Given the regular solution $u\in X$ to $N(u)=0$ in \eqref{eqccdefN}, let the 
bilinear form $\widehat{b}:\widehat{X}\times \widehat{Y}\to \bR$  
be the linearisation of  $\widehat{\Gamma}$ at the solution $u$ , i.e.,
\[
\widehat{b}(\bullet,\bullet):= \widehat{\Gamma}(u,\bullet, \bullet)+\widehat{\Gamma}(\bullet,u,\bullet),
\]
and be bounded by $\|\widehat{b}\|:=\|\widehat{b}\|_{ \widehat{X}\times \widehat{Y} } \le 
2 \|u\|_X  \| \widehat{\Gamma}\|$. 
Adopt the notation  \eqref{defn_ab} for the bilinear forms $a,a_h,b$, and $b_h$
as respective restrictions of $\widehat{a}$ and $\widehat{b}$ and suppose 
$\widehat{F}\in \widehat{Y}^*$ exists with $F:=\widehat{F}|_Y$ and 
$F_h:=\widehat{F}|_{Y_h}$.

Recall 
that the bounded linear operator  $\widehat{A}$ is invertible and the so the associated 
bilinear form $\widehat{a}$ is bounded and satisfies \eqref{dis_Ah_infsup} 
with some positive $\widehat{\alpha}$.

Recall  that $u$ is a regular solution to $N(u)=0$ in the sense that 
$N(u)=0$ and $DN(u)\in L(X;Y^*)$ with $DN(u)=(a+b)(\bullet,\bullet)$ satisfies the 
 inf-sup condition \eqref{cts_infsup}.
 
Suppose all the aforementioned bilinear forms satisfy {\bf (H1)}-{\bf (H4)} with some operators  $P\in L(\widehat{Y};Y_h)$, $ Q\in L(X_h; X)$, and $ \cC\in L(Y_h;Y)$.
In addition to  {\bf (H1)}-{\bf (H4)}  suppose 
that $\delta_5,\delta_6\ge 0$ satisfy 
\\
${\bf (H5)}\quad\displaystyle \delta_5:=\left\|\left(\widehat{F}-\widehat{A}u\right)(1-\cC)\bullet\right\|_{Y_h^*}$;
\\
${\bf (H6)} \quad \displaystyle \exists\,x_h\in X_h $ such that $\delta_6:=\|u-x_h\|_{\widehat{X}}.$
\\
The non-negative  parameters $\delta_1,\delta_2,\delta_3,\delta_5,\delta_6$ 
and $\widehat{\alpha}$,  $\beta$, $\|\widehat{b}\|$ all
depend  on the fixed regular solution $u$ to $N(u)=0$
and this dependence is suppressed in the notation for simplicity. 

Under the present assumptions and with the additional 
smallness assumption $4 {\delta} \|\widehat{\Gamma}\| < \beta_0$
(in the notation of \eqref{ccdefn_beta0}-\eqref{defn_delta})
the properties 
{\bf (A)}-{\bf (B)} hold for the  fixed discretisation at hand 
in the following sense. Suppose that $\|\widehat{\Gamma}\|>0$ for 
otherwise $N$ is a linear equation with a unique solution
and the results of Section~2 apply. 

\begin{thm}[existence and uniqueness of a discrete solution]\label{thm3.1}
Given a regular solution  $u\in X$ to $N(u)=0$, assume the existence of 
extended  bilinear forms $\widehat a$ and $\widehat{b}$ with 
 \eqref{defn_ab}-\eqref{dis_Ah_infsup} and $\widehat{\alpha}>0$ (resp. $\beta>0$ in  \eqref{cts_infsup} 
 and $\widehat{\beta}>0$ in  \eqref{AsmpCondn}). 
Suppose that 
{\bf (H1)}-{\bf (H6)} hold with parameters $\delta_1,\ldots,\delta_6\ge 0$ and 
that  $x_h\in X_h$  satisfies  {\bf (H6)}. Suppose that  
\begin{align}
\label{ccdefn_beta0}
\beta_0&:=\widehat{\alpha}\widehat{\beta}-(\delta_1+\delta_2+\delta_3+2\|\widehat{\Gamma}\|\delta_6) >0\quad\text{and } \\
\label{defn_delta}
{\delta}&:=\beta_0^{-1} \left(\delta_5+ \|\widehat{a}\| \delta_6+\delta_6 \big{(}\|x_h\|_{X_h}+ \|\cC\|\, \|u\|_{X} \big{)}\|\widehat{\Gamma}\|\,+\delta_3/2 \right)\ge 0
 \end{align}
satisfy   $4 {\delta} \|\widehat{\Gamma}\| < \beta_0$.
Then  $ \epsilon:=\delta_6 +\delta + r_-  $ with $m:=2 \|\widehat{\Gamma}\|/\beta_0 >0$,
$h:= \delta m\ge 0$, 
 \begin{equation}\label{rminus}
 r_-:= (1-\sqrt{1-2h })/m  - {\delta}\ge 0 , \quad\text{and}\quad 
  \rho:=(1+\sqrt{1-2h})/ m>0 
 \end{equation}
 satisfy  (i) there exists a solution $u_h\in X_h$ to $N_h(u_h)=0$
 with $\| u-u_h \|_{\widehat{X}}\le \epsilon$ and (ii)  given any $v_h\in X_h$
 with $\|v_h-u_h\|_{{X_h}}\le \rho$, the Newton scheme 
 with initial iterate  $v_h$ converges R-quadratically
 to the discrete solution  $u_h$ in (i). 
 If even $4 \epsilon \|\widehat{\Gamma}\| \le  \beta_0$, then (iii)
 there is at most one solution $u_h\in X_h$ to $N_h(u_h)=0$
 with $\| u-u_h \|_{\widehat{X}}\le \epsilon$.
\end{thm}

The proof is based on the Newton-Kantorovich convergence theorem found, e.g., in \cite[Subsection~5.5]{MR1344684} for \( X= Y=\bR^n\) and in \cite[Subsection~5.2]{MR816732} for Banach spaces. The notation is adopted to the setting of 
Theorem~\ref{thm3.1}.

\begin{thm}[Kantorovich (1948)] \label{kantorovich}
Assume the Frech\'et-derivative $DN_h(x_h)$ of $N_h$ at some \(x_h\in X_h\) satisfies
\begin{equation}\label{Kanto_Condn}
 \|D N_h(x_h)^{-1}\|_{L( Y_h^*; X_h)} \leq 1/\beta_0
    \quad\text{and}\quad
    \|D N_h(x_h)^{-1}N_h(x_h)\|_{X_h} \leq {\delta}.
\end{equation}
Suppose that \(D N_h\) is Lipschitz continuous with Lipschitz constant
$2 \|\widehat{\Gamma}\|$ 
and that  {\(4 {\delta} \|\widehat{\Gamma}\| \leq \beta_0\)}.
Then there exists a unique root \(u_h\in
  \overline{ B(x_1,r_-)} \) to  \(N_h\) in the ball around the first iterate \(x_1 := x_h - D N_h(x_h)^{-1}N_h(x_h)\) 
  and this is the only root in \(\overline{B(x_h,\rho)}\)
with $r_-, \rho$ from \eqref{rminus}. If even {\(4 {\delta} \|\widehat{\Gamma}\| < \beta_0\)},
then 
the Newton scheme with   initial iterate \(x_h\) leads to a sequence in   {\(B(x_h,\rho)\)}  
that  converges R-quadratically to \(u_h\). \qed
\end{thm}

\noindent{\it Proof of Theorem \ref{thm3.1}.}
Suppose that $\delta\ge 0$ and $\|\widehat{\Gamma}\|>0 $  so 
that  $r_-\ge 0$ in \eqref{rminus}  well defined.
The bounded trilinear form 
$\Gamma_h=\widehat{\Gamma}|_{ X_h\times  X_h\times  Y_h}$ 
leads to the Frech\'et-derivative $DN_h( x_h)\in L(X_h;Y_h^*)$ with 
\[
DN_h( x_h;\xi_h,\eta_h)=a_h(\xi_h,\eta_h)+\Gamma_h(x_h, \xi_h,\eta_h)
+\Gamma_h( \xi_h,x_h,\eta_h)\quad\text{for all } x_h, \xi_h\in X_h,\;  
\eta_h\in Y_h.
\]
The definitions of $a$ and $b$ and their extensions and discrete versions with 
 {\bf (H1)}-{\bf (H4)} allow  in  Theorem~\ref{dis_inf_sup_thm} for a positive 
 inf-sup constant 
 $\beta_1:=\widehat{\alpha}\widehat{\beta}-(\delta_1+\delta_2+\delta_3)$  in \eqref{eqdis_inf_sup_defbetah} for the bilinear form 
 \[
 D\widehat{N}(u)|_{X_h\times Y_h}= 
 a_h+\widehat{\Gamma}(u,\bullet,\bullet)+\widehat{\Gamma}(\bullet,u,\bullet)
 = a_h+b_h 
 \] 
 for the extended nonlinear form $\widehat{N}(\widehat{x})=\widehat{A}(\widehat{x})
 -\widehat{F}+ \widehat{\Gamma}(\widehat{x},\widehat{x},\bullet) \in \widehat{Y}^*$
 for $\widehat{x}\in \widehat{X}$ and its derivative  $D\widehat{N}(u)$ at $u$.
 This discrete inf-sup condition \eqref{eqdis_inf_sup_defbetah}
and  a triangle inequality with $x_h$ from {\bf (H6)}  lead to an inf-sup constant
\begin{align*}
0< \beta_0:=
\beta_1 - 2\|\widehat{\Gamma}\|\delta_6
\leq \beta_h:=\inf_{\substack{\xi_h\in X_h\\ \|\xi\|_{X_h}=1}}
\sup_{\substack{\eta_h\in Y_h\\ \|\eta_h\|_{Y_h}=1}}DN_h(x_h;\xi_h,\eta_h)
\end{align*}
for the bilinear form 
$DN_h(x_h;\bullet,\bullet)=a_h+\Gamma_h(x_h,\bullet,\bullet)
+\Gamma_h(\bullet,x_h,\bullet)$. 
The discrete inf-sup constant is a singular value and equal to the norm of the inverse operator; 
$1/\beta_0$ is an upper bound of the operator norm of the 
discrete inverse. This proves  the first estimate of \eqref{Kanto_Condn}. 
It also proves  in the second estimate of
\eqref{Kanto_Condn} that
\begin{equation}\label{Kant_cond1}
\|DN_h(x_h)^{-1}N_h(x_h)\|_{X_h}\leq \beta_0^{-1}\|N_h(x_h)\|_{Y_h^*}
\end{equation}
and it remains to estimate $N_h(x_h)$ in the norm of $Y_h^*$.
Given any  $y_h\in Y_h$ with $\|y_h\|_{Y_h}=1$ and   $y:=\cC y_h\in Y$,  an exact Taylor expansion with  $N(u;y)=0$ shows 
\begin{align}
&N_h(x_h;y_h)= N_h(x_h;y_h)-N(u;y)\notag\\
&= \widehat{F}(y-y_h) +a_h( x_h,y_h)-a(u,y)
+\Gamma_h( x_h, x_h,y_h)-\Gamma(u,u,y)\notag\\
&=\widehat{F}(y-y_h) -\widehat{a}( u,y-y_h)
+\widehat{a}(x_h-u,y_h)+\Gamma_h( x_h, x_h,y_h)
-\Gamma(u,u,y).\label{nonlin_exp}
\end{align}
In  abbreviated duality brackets, the  first two terms in \eqref{nonlin_exp}
are equal to
\begin{align*}
\widehat{F}(y-y_h) -\widehat{a}( u,y-y_h)
=\langle \widehat{F}-\widehat{A}u,(\cC-I)y_h\rangle\leq \delta_5
\end{align*}
with {\bf (H5)}. The definition of $\delta_6$ in {\bf (H6)} proves
\begin{align*}
\widehat{a}(x_h-u,y_h)\leq\|\widehat{a}\|\delta_6.
\end{align*}
Up to the factor $2$, the last two terms in \eqref{nonlin_exp} are equal to
\begin{align*}
&2\Gamma_h( x_h, x_h,y_h)-2\Gamma(u,u,y)
=\widehat{\Gamma}( x_h-u, x_h,y_h)+\widehat{\Gamma}( x_h,x_h-u,y_h)\notag\\
&\quad+\widehat{\Gamma}(u,x_h-u,y)+\widehat{\Gamma}(x_h-u,u,y)-\widehat{b}(x_h,y-y_h).\notag\\
&\leq2\delta_6\left(\|x_h\|_{X_h}+\|\cC\| \: \|u\|_{X}\right)\|\widehat{\Gamma}\|\,+\delta_3.
\end{align*}
The combination of the preceding three displayed estimates with 
\eqref{nonlin_exp} implies 
\begin{equation}\label{eqproofofKant_cond1}
\displaystyle\beta_0^{-1}\|N_h(x_h)\|_{Y_h^*}\leq {\delta}
\end{equation}
with ${\delta}\ge 0$ from  \eqref{defn_delta}. The 
combination of \eqref{Kant_cond1} and  \eqref{eqproofofKant_cond1}
shows the second inequality  in \eqref{Kanto_Condn}. The smallness assumption 
reads $h<1/2$ and is stated  explicitly in the theorem; hence 
the Newton-Kantorovich Theorem~\ref{kantorovich} applies. 

Let us interrupt the proof for a brief discussion of the extreme but possible  case 
$\delta=0$ with the implications  $\delta_6=\delta_5=\delta_3=0$ and 
$x_h=u$ in {\bf (H6)}. The proof of   \eqref{eqproofofKant_cond1} remains valid 
in this case and then 
$N_h(x_h)=0$ guarantees  that $u=x_h$ is the discrete solution $u_h$.
In this very particular situation, the Newton scheme converges  and leads to
the constant sequence $x_h=x_1=x_2=...$ with the  limit $x_h=u_h$. 
Theorem~\ref{kantorovich} applies with $r_-=0=\epsilon$ and provides 
(i)-(iii).

Therefore, throughout the remainder of this proof suppose that $\delta>0$ and so 
$\rho,\epsilon, r_- >0$ in  Theorem~\ref{kantorovich} show   the existence of a discrete solution  $u_h$ to $N_h(u_h)=0 $ in $\overline{B(x_1,r_-)}$
 and this is the only discrete solution in $\overline{B(x_h,\rho)}$. This and 
 triangle inequalities lead to
\begin{align*}
\|u-u_h\|_{\widehat{X}}\leq \|u-x_h\|_{\widehat{X}}+
\|x_1-x_h\|_{X_h}
+\|x_1-u_h\|_{X_h}\leq \delta_6+\delta+r_-= \epsilon \label{Newton_conv}
\end{align*}
for  the Newton correction $x_1-x_h$ is estimated in the second inequality of 
\eqref{Kanto_Condn}. This proves the existence of a discrete solution 
$u_h$ in  $ X_h\cap \overline{B(u,\epsilon)}$  as asserted in (i). 

Theorem~\ref{kantorovich} implies (ii) and it remains to prove the uniqueness of discrete solutions in 
$\overline{B(u,\epsilon)}$ under the additional assumption that 
 $4 \epsilon \|\widehat{\Gamma}\| \le  \beta_0$, i.e.,  $2m\epsilon\le 1$.
Recall  that the limit $u_h\in \overline{B(x_1,r_-)}$  in (i)-(ii) is the only discrete
solution in $\overline{B(x_h,\rho)}$.
Suppose there exists a second solution 
$\widetilde{u}_h\in X_h\cap \overline{B(u,\epsilon)}$ to  $N_h(\widetilde{u_h})=0$.
The uniqueness in $\overline{B(x_h,\rho)}$ and a triangle inequality imply that 
\[
\rho< \|  x_h- \widetilde{u}_h\|_{\widehat{X}}\le  \|  u- \widetilde{u}_h\|_{\widehat{X}}
+\|  u- x_h\|_{\widehat{X}}\le \epsilon+\delta_6\le 2\epsilon\le 1/m
\]
with the smallness assumption on $\epsilon$ in the end. But this leads to a
contradiction with 
the  definition of $\rho$ in  \eqref{rminus} and so  concludes the proof of (iii).
\qed

\begin{rem}
In the applications, if $h_{\max}$ is chosen sufficiently small, the parameters $\delta_1,\delta_2,\delta_3$, $\delta_5$, and  $\delta_6$ are also small.
 In particular, ${\delta}$ from \eqref{defn_delta} is small and 
 so is $\epsilon$. This ensures 
 $4 {\delta}\|\widehat{\Gamma}\|\le  4 {\epsilon}\|\widehat{\Gamma}\|< \beta_0$
so that  Theorem~\ref{thm3.1} applies. 
\end{rem}

\begin{rem}
The convergence speed in the Newton-Kantorovich theorem is known 
to be $h=\delta m$ and this parameter is uniformly smaller than one in the applications.
Hence the number of iterations in the Newton scheme does not increase 
as the mesh-size decreases. 
\end{rem}

\subsection{Best-approximation}
This subsection discusses the best-approximation result {\bf (C)} for regular solutions of semilinear problems with trilinear nonlinearity \ccnew{under the assumption
{\bf (H1)}-{\bf (H6)} with parameters $\delta_1,\ldots,\delta_6$ and $\widehat{\alpha}$ (resp. $\widehat{\beta}$) from \eqref{dis_Ah_infsup}  (resp.  \eqref{AsmpCondn})}.

The extra term  $\|\hat N (u)\|_ {Y_h^*}$ in the best-approximation result in 
Theorem~\ref{err_apriori_thm_vke}  will be discussed afterwards and leads to 
some best-  and data-approximation term.

\begin{thm}[a priori]\label{err_apriori_thm_vke} 
If $u$ is a regular solution to $N(u)=0$ and 
 $\delta $ and  $ \epsilon:=\delta_6 +\delta + r_-  $ from  \eqref{defn_delta}-\eqref{rminus}
 satisfy $2m\epsilon\le 1$, then there exists $ C_\text{\rm qo}>0 $ 
 such that  the  unique discrete solution 
 $u_h\in X_h\cap\overline{B(u,\epsilon )}$  satisfies the best-approximation property 
\begin{equation*}
\| u-u_h\|_{\widehat{X}}\leq C_\text{\rm qo} \left(\min_{v_h\in X_h}\|u-v_h\|_{\widehat{X}}
+\|\widehat{N}(u)\|_{Y_h^*}\right).
\end{equation*}
\end{thm}

\begin{proof}
Given the best-approximation $u_h^*$ to $u$ in $X_h$ with respect to the norm in 
$\widehat{X}$, set $e_h:=u_h^*-u_h\in X_h$ and apply the discrete $\inf$-$\sup$ 
condition  \eqref{eqdis_inf_sup_defbetah} to the bilinear form  
$D\widehat{N}(u)|_{X_h\times Y_h}$ with  the constant 
$\beta_1:=\widehat{\alpha}\widehat{\beta}-(\delta_1+\delta_2+\delta_3)$
from the proof of Theorem~\ref{thm3.1}. 
This leads to  $y_h\in Y_h$ with $\|y_h\|_{Y_h}\le 1/\beta_1$ and 
\begin{equation}\label{inf_sup_tp1}
\|e_h\|_{X_h}= D\widehat{N}(u;e_h,y_h).
\end{equation}
Since the quadratic Taylor expression of $\widehat{N}$ at $u$ for $N_h(u_h;y_h)=0$ is exact, $e:=u-u_h\in \widehat{X}$ satisfies
\begin{equation}\label{inf_sup_tp2}
0=\widehat{N}(u;y_h)-D\widehat{N}(u;e,y_h) {-} \half D^2\widehat{N}(u;e,e,y_h).
\end{equation}
The sum of \eqref{inf_sup_tp1} and \eqref{inf_sup_tp2},   
$ D^2\widehat{N}(u;e,e,y_h)=2 \Gamma(e,e,y_h)$, and $\|y_h\|_{Y_h}\le 1/\beta_1$
prove
\begin{equation*}
\beta_1 \|e_h\|_{X_h}\leq \|\widehat{N}(u)\|_{Y_h^*}
+\|D\widehat{N}(u)\|\|u-u_h^*\|_{\widehat{X}}
+\|\widehat{\Gamma}\|\|e\|_{\widehat{X}}^2.
\end{equation*}
This,  a triangle inequality, and  
$\min_{x_h\in X_h}\|u-x_h\|_{\widehat{X}}  = \|u-u_h^* \|_{\widehat{X}}$ show
\begin{equation}\label{eqbestapproxalmost}
\left(\beta_1-\|\widehat{\Gamma}\| \|e\|_{\widehat{X}}\right)\|e\|_{\widehat{X}} \leq \|\widehat{N}(u)\|_{Y_h^*}+\left(\beta_1+\| D\widehat{N}(u)\|\right)\min_{x_h\in X_h}\|u-x_h\|_{\widehat{X}}.
\end{equation}
Recall $4\epsilon\|\widehat{\Gamma}\|\le  \beta_0\le \beta_1$ and 
$\|e\|_{\widehat{X}}\le \epsilon$ from Theorem~\ref{thm3.1}, so that
$3\beta_1/4  \le \beta_1-\|\widehat{\Gamma}\| \|e\|_{\widehat{X}}$ leads 
in \eqref{eqbestapproxalmost} to 
$C_{\text{qo}}=3/4\, \max\{  1/ \beta_1,  1 +\| D\widehat{N}(u)\|/\beta_1   \} $
and  ${\rm apx}(\cT):=\|\widehat{N}(u)\|_{Y_h^*}$ 
in the asserted best-approximation. 
This concludes the proof.
\end{proof}

Two examples for the term ${\rm apx}(\cT):=\|\widehat{N}(u)\|_{Y_h^*}$ conclude this subsection. 

\begin{example}
If $Y_h\subset Y$, then ${\rm apx}(\cT)=\|\widehat{N}(u)\|_{Y_h^*}\leq \|N(u)\|_{Y^*}=0$. 
Hence, Theorem~\ref{err_apriori_thm_vke} implies the quasi-optimality of the conforming FEM.
\end{example}

\begin{example}
For the second-order linear non-selfadjoint and indefinite elliptic problems of Subsection~2.2,  
$\|\widehat{\Gamma}\|=0$ and $\beta_0=\beta_1$ etc. is feasible in 
Theorem \ref{err_apriori_thm_vke} and the best-approximation estimate holds. The 
approximation  term ${\rm apx}(\cT)$ is the norm of 
the functional $F-(a_{\text{pw}}+b_{\text{pw}}) (u,\bullet )$ in $V_h^*$. This is exactly 
the extra  term in Corollary~\ref{rembestapproximation} that leads to the additional
two terms in Theorem~\ref{SA_err_est}.
\end{example}

\subsection{A posteriori error control}
The regular solution $u$ to $N(u)=0$ is approximated by some $v_h\in X_h$ sufficiently close to $u$ such that the Theorem~\ref{abs_res_thm} below asserts reliability \eqref{relib_eqn} and efficiency 
\eqref{enrich_bdd}-\eqref{eff_temp2}.

\begin{thm}\label{abs_res_thm}
Any $v_h\in X_h$ with $\|u-Qv_h\|_{X}<\beta/\|\Gamma\|$ satisfies
\begin{align}
& \|u-v_h\|_{\widehat{X}}\leq \frac{\|N(Qv_h)\|_{Y^*}}{\beta-\|\Gamma\| \|u-Qv_h\|_{X}}+\|Qv_h-v_h\|_{\widehat{X}},\label{relib_eqn}\\
& \|Qv_h-v_h\|_{\widehat{X}}\leq \Lambda_4\|u-v_h\|_{\widehat{X}},\label{enrich_bdd}\\
& \|N(Qv_h)\|_{Y^*} \leq (1+\Lambda_4)\left(\|DN(u)\|+
\beta 
\right)\|u-v_h\|_{\widehat{X}}.\label{eff_temp2}
\end{align}
\end{thm}

\begin{proof}
Abbreviate $\xi=Qv_h$ and $e:=u-\xi$. Recall that the bilinear form $a+b$ is associated 
to the derivative $DN(u;\bullet,\bullet)\in L(X;Y^*)$ 
with an inf-sup constant $\beta>0$. Hence for any 
$0<\epsilon<\beta$ there exists some $y\in Y$ with $\|y\|_{Y}=1$ and 
\begin{equation}\label{Der_infsup}
(\beta-\epsilon)\|e\|_{X}\leq DN(u;e,y).
\end{equation}
Since $N(u)=0$ and $N$ is quadratic, the finite Taylor series
\begin{align}\label{quadratic_identity}
N(\xi,y)=- DN(u;e,y)+\half D^2N(u;e,e,y)
\end{align}
is exact. This, $D^2N(u;e,e,y)=2\Gamma(e,e,y)$, and \eqref{Der_infsup} imply
\begin{align*}
(\beta-\epsilon)\|e\|_{X}&\leq -N(\xi,y)+\Gamma(e,e,y)
\leq \|N(\xi)\|_{Y^*}+\|\Gamma\| \|e\|_X^2.
\end{align*}
With $\epsilon\searrow 0$ and $\beta-\|\Gamma\| \|e\|_X>0$, this leads to
\begin{equation*}
\|e\|_X\leq \frac{\|N(\xi)\|_{Y^*}}{\beta-\|\Gamma\| \|e\|_X}.
\end{equation*}
A triangle inequality 
$\displaystyle \|u-v_h\|_{\widehat{X}}\leq \|e\|_X+\|Qv_v-v_h\|_{\widehat{X}}$ 
concludes the proof of \eqref{relib_eqn}.

Recall that {\bf (H4)} implies \eqref{enrich_bdd}. This and a triangle inequality show
\begin{equation*}
\|e\|_X\leq \|u-v_h\|_{\widehat{X}}(1+\Lambda_4).
\end{equation*}
The identity \eqref{quadratic_identity} results in
\begin{align*}
\|N(\xi)\|_{Y^*}\leq \|DN(u;e)\|_{Y^*}+\|\Gamma(e,e,\bullet)\|_{Y^*}
\leq \left(\|DN(u)\|+\|\Gamma\| \|e\|_X\right)\|e\|_X.
\end{align*}
The combination of the previous two displayed estimates proves \eqref{eff_temp2}.
\end{proof}

The discrete function $v_h$ can be estimated in the sense of {\bf (D)} from the introduction.

\begin{cor}[a posteriori]\label{coraposteriori}
In addition to  the assumptions of Theorem~\ref{abs_res_thm} suppose that 
$\|  u- v_h\|_{\widehat{X}}  \le  \epsilon\le \kappa\beta /(\|\Gamma\| (1+\Lambda_4))$ holds
for some positive $\kappa<1$ and  $v_h\in X_h$. Then 
$C_{\text{\rm rel},1}:=1/(\beta(1-\kappa))$ and 
 $C_{\text{\rm rel},2}:= 1+  LC_{\text{\rm rel},1}$  for 
 $L:=\| \widehat{a}\|+2\|\widehat{\Lambda}\|(\| u \|_X+\epsilon(1+\Lambda_4))$
satisfy reliability in the sense that 
\[
\|  u- v_h\|_{\widehat{X}}  \le C_{\text{\rm rel},1} \|\widehat{N}(v_h)\|_{Y^*}
+C_{\text{\rm rel},2} \|Qv_h-v_h\|_{\widehat{X}}
\]
and efficiency with  \eqref{enrich_bdd} and with $C_{\text{\rm eff},1}:= \left((1+\Lambda_4) (\|DN(u)\|+\beta)  + L\Lambda_4\right)$ in 
\[
  \|\widehat{N}(v_h)\|_{Y^*}\le C_{\text{\rm eff},1}\|  u- v_h\|_{\widehat{X}} .
\] 
\end{cor}

\begin{proof}
Recall the abbreviations  $\xi=Qv_h$ and $e:=u-\xi$.  A triangle inequality and {\bf (H4)}  
show that $\|e\|_X \le (1+\Lambda_4) \| u- v_h\|_{\widehat{X}}\le \epsilon(1+\Lambda_4)\le
\kappa \beta/ \|\Gamma\|$.
This and Theorem~\ref{abs_res_thm} imply 
\[
 \|u-v_h\|_{\widehat{X}}\leq \frac{\|N(Qv_h)\|_{Y^*}}{\beta(1-\kappa)}
 +\|Qv_h-v_h\|_{\widehat{X}}.
\]
The derivative $D\widehat{N}$ is globally Lipschitz continuous with a Lipschitz constant
$2\|\widehat{\Lambda}\|$, the function 
$\widehat{N}$ is  Lipschitz continuous 
 in the closed  ball   $\overline{B(u,  \epsilon(1+\Lambda_4))}$  in $\widehat{X}$ 
 with a Lipschitz constant  $L$. 
Since  $  v_h, Qv_h\in \overline{B(u,  \epsilon(1+\Lambda_4))}$,
\[
\|N(Qv_h)\|_{Y^*} \le \|\widehat{N}(v_h)\|_{Y^*} + L\,\|Qv_h-v_h\|_{\widehat{X}}.
\]
The combination of the previous displayed estimates proves the asserted reliability. The
efficiency employs the Lipschitz continuity as well and then utilises 
\eqref{enrich_bdd}-\eqref{eff_temp2}
to verify 
\[
  \|\widehat{N}(v_h)\|_{Y^*} \le \|N(Qv_h)\|_{Y^*}+  L\,\|Qv_h-v_h\|_{\widehat{X}} 
 \leq C_{\text{\rm eff},1}  \|u-v_h\|_{\widehat{X}}.
\]
This  concludes the proof. 
\end{proof}

\section[Incompressible 2D Navier-Stokes problem]{Stream function 
vorticity formulation of the\\ incompressible 2D Navier-Stokes problem}\label{sec:NS}
This section is devoted to the stream function vorticity formulation of 2D Navier-Stokes 
equations with right-hand side $f\in\lt$ in a 
polygonal bounded Lipschitz domain $\Omega\subset\mathbb{R}^2$: 
There exists \cite{Lions} at least one 
distributional solution $u\in V:=\hto$ to
\begin{equation}\label{NS}
\Delta^2 u+\frac{\partial}{\partial x_1}\left((-\Delta u)\frac{\partial u}{\partial x_2}\right)-\frac{\partial}{\partial x_2}\left((-\Delta u)\frac{\partial u}{\partial x_1}\right)=f\text{ in }\Omega.
\end{equation}
The analysis of extreme viscosities lies beyond the scope of this paper and the viscosity 
(the factor in front of the bi-Laplacian in \eqref{NS})
is set  one throughout this paper.

\subsection{Continuous problem}\label{subsectContinuousproblemNS}
The weak formulation to \eqref{NS} seeks $u\in V$ such that 
\begin{equation}\label{NS_weak}
a(u,v)+\Gamma(u,u, v)=F( v)\fl  v\in V.
\end{equation}
The associated bilinear form $a: V\times V\to \bR$ and the trilinear form $\Gamma: V\times V\times V\to \bR$ read
\begin{align*}
& a(\eta, \chi):=\integ \Delta\eta\Delta \chi\dx, \quad
 \Gamma(\eta,\chi,\phi):=\integ \Delta\eta\left(\frac{\partial \chi}{\partial x_2}\frac{\partial \phi}{\partial x_1}-\frac{\partial \chi}{\partial x_1}\frac{\partial \phi}{\partial x_2}\right)\dx,
\end{align*}
and $F\in V^*$ is given by $\displaystyle F( \phi):=\integ f \phi\dx$ for all $\eta,\chi,\phi\in V$. 
The Hilbert space $V\equiv H^2_0(\Omega)$ with the scalar product 
 $a(\bullet,\bullet)$ is endowed with the $H^2$ seminorm 
$\displaystyle\trinl\bullet\trinr:=|\bullet|_{H^2(\Omega)}$
and $\|\bullet\|_{V^*}$ denotes the dual norm.
The bilinear form $a(\bullet,\bullet)$ is equivalent to the  scalar product in $V$
and the trilinear form 
$\Gamma(\bullet,\bullet,\bullet)$ is bounded \ccnew{(owing to the continuous embedding 
$V\subset H^2(\Omega)\hookrightarrow W^{1,4}(\Omega)$)}  with
\[
\langle N(u), v\rangle=N(u; v):=a(u, v)-F(v)+\Gamma(u,u, v)\quad\text{ for all } u,v\in V.
\]
The 2D Navier-Stokes equations in the weak stream function vorticity 
formulation \eqref{NS_weak} seeks $u\in V$ with $N(u)=0$.

The regularity results for the biharmonic operator $\Delta^2$  in \cite{BlumRannacher} ensure that $z\in V$ with $a(z,\bullet)\in H^{-1}(\Omega)\subset V^*$ belongs to 
$ H^{2+s}(\Omega)$ for some elliptic regularity index  {$s \in  (1/2,1]$}  and 
$\|z\|_{H^{2+s}(\Omega)}\leq C \|a(z,\bullet)\|_{H^{-1}(\Omega)}$. 
The regularity results for the 
Navier-Stokes problem \ccnew{in} \cite[Section 6(b)]{BlumRannacher} ensure that 
\ccnew{any weak solution $u\in V$ to $N(u)=0$ satisfies $u\in H^{2+s}(\Omega)$.
This makes the  continuous embeddings  
$H^{2+s}(\Omega) \hookrightarrow W^{1,\infty}(\Omega)$ (for $s>0$) and 
$H^{2+s}(\Omega)\hookrightarrow W^{2,4}(\Omega)$ (for $s>1/2$) available
throughout this (and the subsequent) section.} The embeddings 
and \Holder inequalities  imply 
for $u\in H^{2+s}(\Omega)$ and for $\theta \in V$, {$\phi \in H^1(\Omega)$} that
\begin{align*}
\Gamma(u,\theta,\phi)
& \lesssim \|u\|_{H^{2+s}(\Omega)} \|\theta\|_{H^2(\Omega)}\|\phi\|_{H^1(\Omega)}.
\end{align*}
Consequently,  the derivative $b(\bullet,\bullet) := DN(u;\bullet,\bullet):= 
\Gamma(u,\bullet,\bullet)+\Gamma(\bullet,u,\bullet)$ at the solution $u$ 
is a bounded bilinear form in $H^2(\Omega)\times H^1(\Omega)$ and will be key in the subsequent analysis.

\subsection{Conforming FEM}\label{sec:CFEM_NS}
Let $V_C$ be a conforming finite element space contained in
$ C^1(\overline\Omega)\cap V$; for example, the spaces associated with Bogner-Fox-Schmit, HCT, or Argyris elements \cite{Ciarlet} and a regular triangulation $\cT$
of $\Omega$ into triangles. The conforming finite element formulation  seeks $u_C\in V_C$ with
\begin{equation}\label{NS_dis}
N_h(u_C;v_C):=N(u_C;v_C):=a(u_C, v_C)-F( v_C)+\Gamma(u_C,u_C, v_C)=0\fl  v_C\in V_C.
\end{equation}

\begin{thm}[a priori]\label{apriori_NS_est_C}
If $u$ is a regular solution to $N(u)=0$, then there exist positive 
$\epsilon$, $\delta$, and $\rho$ such that 
 {\bf (A)}-{\bf (C)} hold with ${\rm apx}(\cT)\equiv 0$ 
for all $\cT \in \bT(\delta)$. 
\end{thm}

\begin{proof}
Set $X=Y=V$,  $X_h=Y_h=V_C$,  
$\widehat{a}(\bullet,\bullet):=a(\bullet,\bullet)$, and 
$ \widehat{b}(\bullet,\bullet):=b(\bullet,\bullet)
:=\Gamma(u,\bullet,\bullet)+\Gamma(\bullet,u,\bullet)$. 
For $\cC$ and $Q$ chosen as identity, 
the parameters in the hypotheses {\bf (H1)} and {\bf (H3)}-{\bf (H5)} are  
$\delta_1=\delta_3=\Lambda_4=\delta_5=0$.

For the proof of {\bf (H2)}, suppose 
 $\theta_h\equiv x_h\in V_C\subset V$ with $ \trinl\theta_h \trinr=1$ and recall
 from the end of the previous subsection that 
 $\widehat{b}(\theta_h, \bullet)\in H^{-1}(\Omega)$.  Hence  the solution 
$z\in V$ to the biharmonic problem
\begin{equation*}
a(z,\phi)=\widehat{b}(\theta_h,\phi)\fl \phi\in X
\end{equation*}
satisfies $z\in H^{2+s}(\Omega)$ and 
$\|z\|_{H^{2+s}(\Omega)}\leq C \trinl\theta_h\trinr= C$. 
(Note that $z$ is called $A^{-1}(\widehat{b}(x_h,\bullet)|_{Y})$ in
Subsection~\ref{sec:abs_result}).
This regularity and the Galerkin projection ${P}$ with the 
Galerkin orthogonality and the approximation 
property \ccnew{$\trinl z-{P}z\trinr\lesssim h_{\rm {max}}^s$ \cite{Brenner} lead 
for any $y\in Y\equiv V$ with $ \trinl y \trinr=1$ to 
\begin{equation*}
a(x_h + z, {y}-{P}{{y}} ) =a(z,{y}-{P}{{y}} )
= a(z-Pz,y) \lesssim h_{\max}^{s}.
\end{equation*}
This proves   {\bf (H2)} with  $\delta_2\lesssim h_{\rm max}^s$.
The choice $x_h={P}u$ implies {\bf (H6)} with $\delta_6\lesssim h_{\rm max}^{s}$ 
(from the higher regularity of $u$ and 
$\trinl u-Pu\trinr\lesssim h_{\max}^s$). Consequently,   
for sufficiently small maximal mesh-size $h_{\max}$, 
Theorem~\ref{dis_inf_sup_thm} provides 
the discrete inf-sup condition $1\lesssim \beta_h$ and 
Theorem~\ref{thm3.1} applies.}
Since $V_C$ is a conforming finite element space, 
Theorem~\ref{err_apriori_thm_vke} holds with 
${\rm apx}(\cT):= \|\widehat{N}(u)\|_{Y_h^*} \equiv 0$. 
This concludes the proof.
\end{proof}

The explicit residual-based $a~posteriori$ error estimator for the stream function vorticity formulation of 2D Navier-Stokes equations requires some notation for the differential operators: For any scalar function $v$, vector field $\Phi=(\phi_1,\phi_2)^T$, 
and tensor $ {\boldsymbol \sigma}$ with the 4 entries $\sigma_{11}$, $\sigma_{12}$,
$\sigma_{21}$, and $\sigma_{22}$ in form of a $2\times 2$ matrix,  
$$\displaystyle \nabla v=\begin{pmatrix}
\frac{\partial v}{\partial x_1}\\
\frac{\partial v}{\partial x_2}
\end{pmatrix},\;  {\rm Curl}\, v=\begin{pmatrix}
-\frac{\partial v}{\partial x_2}\\
\frac{\partial v}{\partial x_1}
\end{pmatrix},\;
{\rm curl} \begin{pmatrix}
\phi_1\\
\phi_2
\end{pmatrix}
=\frac{\partial \phi_2}{\partial x_1}-\frac{\partial \phi_1}{\partial x_2},
\quad 
D\Phi=
\begin{pmatrix}
\frac{\partial \phi_1}{\partial x_1} &\frac{\partial \phi_1}{\partial x_2}\\
\frac{\partial \phi_2}{\partial x_1} &\frac{\partial \phi_2}{\partial x_2}
\end{pmatrix}, 
$$ 

\noindent $$ 
\divc \begin{pmatrix}
\phi_1\\
\phi_2
\end{pmatrix}=\frac{\partial \phi_1}{\partial x_1}+\frac{\partial \phi_2}{\partial x_2}
,\;
{\rm Curl} \begin{pmatrix}
\phi_1\\
\phi_2
\end{pmatrix}
=\begin{pmatrix}
-\frac{\partial \phi_1}{\partial x_2} &\frac{\partial \phi_1}{\partial x_1}\\
-\frac{\partial \phi_2}{\partial x_2} &\frac{\partial \phi_2}{\partial x_1}
\end{pmatrix}
, \;
{\text {and} \;}
\divc{\boldsymbol \sigma}
=\begin{pmatrix}
\frac{\partial \sigma_{11}}{\partial x_1} +\frac{\partial \sigma_{12}}{\partial x_2}\\
\frac{\partial \sigma_{21}}{\partial x_1} +\frac{\partial \sigma_{22}}{\partial x_2}
\end{pmatrix}.
$$
For any $K\in\cT$ and $E\in\cE(\Omega)$, define the volume and edge  error estimators by 
\begin{align*} 
\eta_K^2&:=h_K^4\big{\|}{} \Delta^2 u_C
-{\rm curl}(\Delta u_C\nabla u_C) -f\big{\|}^2_{L^2(K)}, \\
\eta_E^2&:=h_E^{3}\left\|
\left [\divc (D^2   u_C)\right]_E \cdot\nu_E\right\|^2_{L^2(E)}  
+h_E\left\|\left[\Delta u_C \right]_E\right\|^2_{L^2(E)}\\
&\quad+h_E^{3}\left\|[\Delta u_C\nabla u_C]_E\cdot\tau_E\right\|^2_{L^2(E)}
\end{align*}
with   the unit tangential (resp. normal) vector  $\tau_E$ 
(resp. $\nu_E$) along  the edge $E\in\cE$. 
Recall ${\rm osc}_{m}(\bullet,\cT):=\|h_{\cT}^2(I-\Pi_m)\bullet\|_{L^2(\Omega)}$ for 
$m\in\mathbb{N}_0$ in all
fourth-order applications.

\begin{thm}[a posteriori] \label{reliability_C_NS} 
If  $u\in V$ is a regular solution to $N(u)=0$ and $m\in\bN_0$, then there exist positive 
$\epsilon,\delta, C_{\rm rel}$, and $C_{\rm eff}$ such that, for any $\cT \in \bT(\delta)$,
the unique discrete solution  $u_C\in  V_C$ to \eqref{NS_dis} with 
$\trinl u-u_C\trinr<\epsilon$ satisfies 
\begin{align}
C_{\rm rel}^{-2}\trinl u-u_C\trinr^2&\leq\sum_{K\in\cT}\eta_K^2
+\sum_{E\in\cE (\Omega)}\eta_E^2\leq C_{\rm eff}^2(\trinl u-u_C\trinr^2
+{\rm osc}_{m}^2(f)). 
\label{reliability_est_C_NS}
\end{align}
\end{thm}

The proof utilizes a quasiinterpolation operator.

\begin{lem}[quasiinterpolation] \label{interpolation_BFS}
For any $\cT\in\bT$ there exists an interpolation operator $\Pi_h: H^2_0(\Omega)\to V_C$ such that,
 for $0\le k\le m\le 2 $ and  $\varphi\in H^2_0(\Omega)$,  
\begin{equation*}
\|\varphi-\Pi_h\varphi\|_{H^k(K)}\lesssim
h_K^{m-k}|\varphi|_{H^q(\omega_K)} 
\end{equation*}
holds for any in the triangle $K\in\cT$ 
and  the  interior $\omega_K$ of the union  $\overline{\omega_K}$    of the triangles in $\cT$ 
sharing a vertex with $K$. 
\end{lem}

\begin{proof}
This follows from \cite{Clement75} once the required scaling properties of the degrees
of freedom are clarified. The Argyris or the HCT finite element schemes involve some
normal derivative and  do {\em not} form an affine finite element family, but 
an {\em almost affine} finite element element family \cite{Ciarlet}. It is by now understood that
this guarantees the appropriate scaling properties. This is  explicitly  calculated in  \cite{Ciarlet} for the
HCT finite elements and also follows  for the Argyris finite elements, as employed 
e.g. in \cite[p. 995]{brennermathcomp}. Since the result is frequently accepted
\cite{Verfurth},  further details are omitted. 
\end{proof}

\noindent{\it Proof of Theorem~\ref{reliability_C_NS}.}
Continue the notation of the proof of Theorem~\ref{apriori_NS_est_C}  with
$X=Y=V$, $X_h=Y_h=V_C$, $Q=1$,  etc. and recall that, for sufficiently small $\delta$,
Theorem~\ref{apriori_NS_est_C}  guarantees $\trinl u-u_C\trinr < \beta/\|\Gamma\|$.
Hence Corollary~\ref{coraposteriori} implies (for $v_h\equiv u_C$) 
\begin{align}\label{F1_NS}
\trinl u-u_C\trinr\le  C_{\text{\rm rel},1} \|N(u_C)\|_{V^*}.
\end{align}
With $\Pi_h$ from Lemma~\ref{interpolation_BFS},  some 
appropriate  $\phi\in V$ with $\trinl \phi\trinr=1$  satisfies 
\begin{equation}\label{CFEM_NS_res}
\|N(u_C)\|_{V^*}= N(u_C;\phi)= N(u_C;\phi-\Pi_h\phi).
\end{equation} 
Two successive integrations by parts result in
\begin{align}
&a(u_C,\phi-\Pi_h\phi)=\sik (\Delta^2 u_C) (\phi-\Pi_h\phi)  \dx+\sie\left[\Delta u_C\right]_E \nabla(\phi-\Pi_h\phi)\cdot \nu_E \ds \notag \\
&\qquad\qquad\qquad\qquad- 
\sie (\phi-\Pi_h\phi)\left[\divc(D^2 u_C)\right]_E \cdot\nu_E\ds. 	\label{1}
\end{align}
An integration by parts in the nonlinear term 
$\Gamma(u_C,u_C,\phi-\Pi_h\phi)$ leads to
\begin{align}
&\Gamma(u_C,u_C,\phi-\Pi_h\phi)=\sum_{K\in\cT}
\int_K \Delta u_C\nabla u_C\cdot{\rm Curl}(\phi-\Pi_h\phi)\dx \label{2}\\
&=\sum_{K\in\cT}\int_K (\phi-\Pi_h\phi){\rm curl}(-\Delta u_C\nabla u_C)\dx
+\sum_{E\in\cE}\int_{E}(\phi-\Pi_h\phi) [\Delta u_C\nabla u_C]_E\cdot\tau_E\ds.\notag
\end{align}
Those identities show that  \eqref{CFEM_NS_res} is equal to a sum 
over edges of jump contributions plus a sum
over triangle of volume contributions; the latter is  
\[
\sik\left( \Delta^2 u_C-{\rm curl} (\Delta u_C\nabla u_C)-f   \right) (\phi-\Pi_h\phi)  \dx
\lesssim   \sum_{K\in\cT} \eta_K h_K^{-2} ||   \phi- \Pi_h \phi \|_{L^2(K) }  
\]
and controlled  with standard manipulations 
based on Lemma~\ref{interpolation_BFS} (with $k=0$ and $m=2$) and the finite
overlap of the patches $(\omega_K:K\in\cT)$. The jump contributions include 
some trace inequality as well and are otherwise standard as in linear problems
that involve the bi-Laplacian. For instance, the nonlinear jump contribution 
for each edge $E$ reads
\[
\int_{E}(\phi-\Pi_h\phi) [\Delta u_C\nabla u_C]_E\cdot\tau_E\ds
=\int_{E}(\phi-\Pi_h\phi) [\Delta u_C]_E \, \nabla u_C\cdot\tau_E\ds
\]
in case of  an interior edge $E$ shared by the two triangles $T_+$ and $T_-$ that form the patch
$\omega_E$ and  vanishes in case of a boundary edge $E\subset\partial\Omega$
(with $\phi=\Pi_h\phi=0$ on $\partial\Omega$). 
The continuity of $\nabla u_C$ leads to the previous equality. This term is
controlled by the residual  
$h_E^{3/2} \|  [\Delta u_C]_E \, \nabla u_C\cdot\tau_E\|_{L^2(E)}$ times 
\[
h_E^{-3/2} \|\phi-\Pi_h\phi\|_{L^2(E)}
\lesssim   h_E^{-2} \|  \phi-\Pi_h\phi\|_{L^2(T_\pm)} 
+h_E^{-1}  \|  \phi-\Pi_h\phi \|_{H^1(T_\pm)}
\lesssim | \phi |_{H^2(\omega_{T_\pm})}
\]
with a trace inequality on one of the two triangles $T_\pm$ in the first and 
 Lemma~\ref{interpolation_BFS} (for $k=0,1$)  in the second estimate. 
 The remaining terms are
controlled in a similar way.

Some words are in order about the term 
$
h_E^{3/2} \|  [\Delta u_C]_E \, \nabla u_C\cdot\tau_E\|_{L^2(E)} 
$, in which  an inverse inequality along the interior edge $E=\partial T_+\cap\partial T_-$ 
(shared by $T_\pm\in\cT$) of the polynomial
$ \nabla u_C\cdot\tau_E$ (unique as a trace from $T_\pm$) shows 
$\| \nabla u_C\cdot\tau_E \|_{L^\infty(E)}\lesssim h_E^{-1} \| u_C \|_{L^\infty(E)}$.
This and the global continuous embedding  
$H^2(\Omega)\hookrightarrow L^{\infty}(\Omega)$
leads to
\[
h_E^{3/2} \|  [\Delta u_C]_E \, \nabla u_C\cdot\tau_E\|_{L^2(E)} 
\lesssim h_E^{1/2} \|  [\Delta u_C]_E \|_{L^2(E)}   \trinl u_C \trinr.
\]
Since $ \trinl u_C \trinr\lesssim 1$, the nonlinear edge contribution 
 is controlled  by another contribution
$h_E^{1/2} \|  [\Delta u_C]_E \|_{L^2(E)} $ to $\eta_E$;
in other words,
this nonlinear edge contribution can be omitted. 

The overall strategy in the efficiency proof 
follows the  bubble-function technique due to Verf\"urth 
 \cite{Verfurth}. The emphasis in this paper is on the nonlinear contributions 
and on the interaction of the various nonlinear terms with the volume estimator. 
We will give two examples only to illustrate some details and start with 
the cubic bubble-function  $b_K \in W^{1,\infty}_0(K)$ 
(the product of all three barycentric coordinates times $27$)
of the triangle $K\in \cT$ with $0\le b_K\le\max b_K=1$. 
Let $f_K:=\Pi_m f\in P_m(K)$ be the
 $L^2(K)$  orthogonal polynomial projection of $f\in L^2(K)$ for  degree 
 $m\in \mathbb{N}_0$ 
so that $ \|f-f_K\|_{L^2(K)}=h_K^{-2}{\rm osc}_m(f,K)$.

Since $ g:=  \Delta^2 u_C-{\rm curl} (\Delta u_C\nabla u_C)-f_K$ is a polynomial
of degree at most $\max\{k-4,(k-2)(k-1)-1,m\} $ (recall that $k$ is the 
degree of the  finite element functions), an inverse estimate 
reads    $\|  g\|_K^2\lesssim \int_K\rho_K g\dx$  for the test function
$\rho_K:=b_K^2g\in H^2_0(K)\subset V$. 
The above integrations by parts \eqref{1}-\eqref{2}  with the test function 
$ \phi-\Pi_h\phi $ replaced by $\rho_K$ are restricted to $K$ for 
the support of $b_K$ and $\nabla b_K$  is $K$. 
This leads to the first equality in 
\begin{align*}
\int_K g \rho_K\dx&=a(u_C, \rho_K)+  \Gamma(u_C,u_C,\rho_K) -\int_K\rho_Kf_K\dx\\
&=a(u_C-u, \rho_K)+  \Gamma(u_C,u_C,\rho_K) - \Gamma(u,u,\rho_K)
+\int_K\rho_K(f-f_K)\dx
\end{align*}
and  \eqref{NS_weak}  leads to the second.  
Except for the last term (that leads to oscillations in the end), 
elementary algebra,  $ \Gamma(u,u,\rho_K)- \Gamma(u_C,u_C,\rho_K) 
=\Gamma(u-u_C,u,\rho_K)+\Gamma(u_C,u-u_C,\rho_K)$,
Cauchy, and H\"older inequalities bound the 
above terms upto a constant  by
\begin{equation} \label{NS_trilin}
 \| u-u_C\|_{H^2(K)}\left( (1 +  | u |_{W^{1,\infty}(\Omega)}) \| \rho_K\|_{H^2(K)} 
+| u_C|_{H^2(\Omega)}   | \rho_K|_{W^{1,\infty}(K)}\right).
\end{equation}
The  inverse estimates  
$ \| \rho_K\|_{H^2(K)} + | \rho_K|_{W^{1,\infty}(K)} 
\lesssim h_K^{-2} \| \rho_K\|_{L^2(K)} 
\le h_K^{-2}  \| g \|_{L^2(K)}$
lead in the preceding estimates (after division by  $h_K^{-2}  \| g \|_{L^2(K)}$)
to
\[
h_K^2  \|  g\|_{L^2(K)} \lesssim \| u-u_C\|_{H^2(K)}+{\rm osc}_m(f,K).
\]
This and a  triangle inequality  prove efficiency 
$\eta_K\lesssim  \| u-u_C\|_{H^2(K)}+{\rm osc}_m (f,K)$ of the volume contribution.

The patch  $\omega_E$ of an  interior edge  $E\in \cE$ is the interior of the union of the two neighbouring 
triangles in $\cT$ sharing the edge $E$ and may be a non-convex quadrilateral.  Observe that the
shape-regularity in $\cT$ implies the  shape-regularity of 
the largest rhombus  $R$   contained in the patch $\omega_E$ that has $E$ as one diagonal. 
Let $b_R\in H^1_0(R)\subset H^1_0(\omega_E)$ be the (piecewise quadratic) 
edge-bubble function of $E$ in $R$ (with $0\le b_R\le \max b_R=1$) 
and let $\Phi_E\in P_1(R)$ be the affine function that vanishes along $E$
and satisfies $\nabla \Phi_E =  h_E^{-1}\nu_E$.
Then $b_E:= \Phi_E b_R^3\in H^2_0(R)\subset H^2_0(\omega_E)$  satisfies 
$\nabla b_E \cdot\nu_E= h_E^{-1} b_{R}^3$ along $E$ and
$ | b_E|_{L^{\infty}(\omega_E)}\lesssim 1$ as in  \cite{Georgoulis2011}. 
Extend $[\Delta u_{C}]_E$ constantly in the normal direction to $E$ and set 
$\varrho_E:=h_E^2 [\Delta u_{C}]_E\, b_E$ $\in H^2_0(R)\subset H^2_0(\omega_E)$.  An inverse 
estimate in the beginning,
$\nabla \varrho_E\cdot \nu_E=h_E b_R^3[\Delta u_C]_E$ on $E$, 
and piecewise integrations by parts lead to 
\begin{align*}
&h_E \|[\Delta u_{C}]_E\|_{L^2(E)}^2  \lesssim  h_E\|b_{R}^{{3/2}}[\Delta u_{C}]_E\|_{L^2(E)}^2
=\int_E \nabla\varrho_E\cdot\nu_E [\Delta u_{C}]_E\ds \\
&\qquad =\int_{\omega_E}(\Delta u_{C}\Delta \varrho_E -\varrho_E\Delta^2_{\text{pw}} u_{C})\dx.
\end{align*}
The test-function  $\varrho_E$  in \eqref{NS_weak} shows  that,
the right-hand side reads
\[
 a(u_C-u,\varrho_E)+\Gamma(u_C,u_C,\varrho_E)-\Gamma(u,u,\varrho_E)
 + \int_{\omega_E} \hspace{-2mm}(f-\Delta_{\text{pw}} ^2 u_{C}+{\rm curl}_{\text{pw}} (\Delta u_C\nabla u_C))\varrho_E\dx.
\]
A Cauchy inequality in the first, the arguments for \eqref{NS_trilin} in the second term, and the  bound  
$(\eta_{T_+}+\eta_{T_-})h_E^{-2} \|  \varrho_E \|_{L^2(\omega_E)}$ for the third term lead to 
\[
h_E \|[\Delta u_{C}]_E\|_{L^2(E)}^2  \lesssim
 \left( \| u-u_C\|_{H^2(\omega_E)}  +\eta_{T_+}+\eta_{T_-}\right)  
 \left(h_E^{-2}  \|  \varrho_E \|_{L^2(\omega_E)}+ |\varrho_E |_{H^2(\omega_E)} \right).
\]
The function $\varrho_E$ is polynomial in each of the two open triangles in $R\setminus (E\cup\partial R)$
and allows for inverse estimates. Since  $|b_E|\lesssim 1$ a.e., this proves that the last factor is
controlled by
\[
 h_E^{-2} \|\varrho_E\|_{L^2(\omega_E)}\lesssim \|[\Delta u_{C}]_E\|_{L^2(\omega_E)}
\lesssim  h_E^{1/2}\|[\Delta u_{C}]_E\|_{L^2(E)}
\]
for the constant extension of  $[\Delta u_{C}]_E$ in the direction of $\nu_E$ in the last step.
The combination of the previous two displayed inequalities with the above efficiency of the 
volume contribution concludes the proof of 
\[
h_E^{1/2} \|[\Delta u_{C}]_E\|_{L^2(E)}\lesssim \|u-u_{C}\|_{H^2(\omega_E)}+\eta_{T_+}+\eta_{T_-}
\lesssim \|u-u_C\|_{H^2(\omega_E)}+{\rm osc}_m (f,\{T_+,T_-\}).
\]
The efficiency of $h_E^{3}\left\| \left[\divc (D^2   u_C)\right]_E \cdot\nu_E\right\|^2_{L^2(E)}$ is also established through an adoption of the corresponding arguments in  \cite{Georgoulis2011}. Hence the
straightforward details are omitted. 
\qed

\subsection{Morley FEM}
The nonconforming {\it Morley element space} $V_h:=\cM(\cT)$ 
associated with the  triangulation $\cT$ of the polygonal domain 
$\Omega\subset\mathbb{R}^2$ into triangles reads
\[
\cM(\cT):=\left\{ v_M\in P_2(\cT) \; \vrule\;\;
\begin{aligned}
& v_M \text{ is continuous at } \cN(\Omega) \text{ and vanishes at }  
\cN(\partial \Omega), 
\\&  \int_{E}\left[\frac{\partial v_M}{\partial \nu}\right]_E\ds=0
 \text{ for all }  E \in \cE (\Omega), 
\\ &
 \int_{E}\frac{\partial v_M}{\partial \nu}\ds=0
 \text{ for all }E\in \cE (\partial\Omega)
\end{aligned}
\right\}.
\]
The discrete formulation 
seeks $ u_M\in \cM(\cT)$ such that 
\begin{equation}\label{NS_dis_NC}
N_h(u_M;v_M):=a_{\text{pw}}( u_{M}, v_{M})-F( v_{M})
+\Gamma_{\text{pw}}( u_{M}, u_{M}, v_{M})=0\fl  v_{M}\in \cM(\cT).
\end{equation}
Here and throughout this section, $ \widehat{V}:=V+ \cM(\cT)$ is endowed with the mesh-dependent 
norm 
$\displaystyle \trinl\widehat{\varphi}\trinr_{\text{pw}}
:=\sqrt{a_{\text{pw}}(\widehat{\varphi},\widehat{\varphi})}$
for $\widehat{\varphi} \in \widehat{V}$ and, for all $\eta,\chi,\phi\in \cM(\cT)$, 
\begin{align}\label{eqccnerwandlast1234a}
a_{\text{pw}}(\eta,\chi)
&:=\sum_{K\in\cT}\int_K D^2 \eta:D^2\chi\dx,\\ 
\label{eqccnerwandlast1234b}
\Gamma_{\text{pw}}(\eta,\chi,\phi)
&:= \hspace{-0.2cm}\sum_{T\in\cT}\int_T \Delta \eta\left(\frac{\partial \chi}{\partial x_2}\frac{\partial \phi}{\partial x_1}-\frac{\partial \chi}{\partial x_1}\frac{\partial \phi}{\partial x_2}\right)\dx.
\end{align}

The a~priori error estimate  means  best-approximation upto first-order terms and so refines 
\cite{CN86,CN89} for the Morley FEM and generalises it for any regular solution.  

\begin{thm}[a priori]\label{apriori_NS_est_NCFEM}
If $u\in H^2_0(\Omega)$ is a regular solution to $N(u)=0$, then there exist positive 
$\epsilon$, $\delta$, and $\rho$ such that  {\bf (A)}-{\bf (C)} hold for all $\cT \in \bT(\delta)$ with
\begin{align}\label{eqnapriori_NS_est_NCFEM}
{\rm apx}(\cT) \lesssim \ensuremath{| \!| \! |}u-I_{ M} u\ensuremath{| \!| \! |}_{\text{\rm pw}}
+ \|  h_\cT  \Delta u \nabla u \|
+ \text{\rm osc}_0(f,\cT)\lesssim h_{\rm max}^{s}.
\end{align}
\end{thm}

The proof requires the following four lemmas. 

\begin{lem}[Morley interpolation \cite{HuShi_Morley_Apost,CCDGJH14}] \label{Morley_Interpolation}
For any $v\in V+ \cM(\cT)$, the Morley interpolation 
$I_M(v)\in \cM(\cT)$  defined by
$$
(I_M v)(z)=v(z) \text{ for any } z\in \cN(\Omega) \text{ and } 
\int_E\frac{\partial I_M v}{\partial \nu_E}\ds=\int_E\frac{\partial v}{\partial \nu_E}\ds \text{ for any } E\in \cE
$$
satisfies (a) $D^2_{\text{\rm pw}} I_M =\Pi_0 D^2$ and (b)
\begin{equation*}
\|h_K^{-2}(1-I_M)v \|_{L^2(K)}+\|h_K^{-1}\nabla(1-I_M) v\|_{L^2(K)}
+{\|D^2 I_Mv\|_{L^2(K)}}\lesssim
\|D^2v\|_{L^2(K)}.\qquad\qed
\end{equation*}
\end{lem}

Let $H^2(\cT(\omega_K))$ denote the piecewise $H^2$ functions on 
the neighbourhood  $\omega_K$, piecewise with respect to  the triangulation 
$\cT(\omega_K)$ of all triangles $T$ with zero distance to $K\in\cT$. Let
$|\bullet|_{H^2(\cT(\omega_K))}$ be the corresponding seminorm
as the local contributions of $\trinl \bullet \trinr_{\text{pw}}$ associated with $\omega_K$.

\begin{lem}[enrichment \cite{BSZ2013,DG_Morley_Eigen}] 
\label{hctenrich} 
There exists an enrichment operator $E_M:\cM(\cT)\to V$ such that $\varphi_M\in \cM(\cT)$ satisfies
\begin{align*}
&(a)\quad \sum_{m=0}^2 h_K^{2m}|\varphi_M-E_M\varphi_M|_{H^m(K)}^2 \lesssim \: h_{K}^4|\varphi_M|_{H^2(\cT(\omega_K))}^2\fl K\in\cT; \\
&(b)\quad \| h_{\cT}^{-2}(\varphi_M-E_M\varphi_M)\|_{\lt}^2\lesssim \sum_{E\in\cE} h_E \|[D^2\varphi_M]_E\tau_E\|_{L^2(E)}^2\\
&\hspace*{4cm}\lesssim \trinl \varphi_M-E_M\varphi_M\trinr_{\text{\rm pw}}^2\leq\Lambda \min_{\varphi\in V}\| D_{h}^2(\varphi_M-\varphi)\|_{\lt}^2; \\
&(c) \quad I_ME_M\varphi_M=\varphi_M,\quad\text{and}\quad
\varphi_M-E_M\varphi_M\perp P_0(\cT)\text{ in }L^2(\Omega).\qquad\qquad\qed
\end{align*}
\end{lem}

The Sobolev embeddings for conforming functions depend on the domain 
$\Omega$, while their discrete counterparts for nonconforming functions require
particular attention.

\begin{lem}[discrete embeddings]\label{lemmadiscreteembeddings}
For any $1\le p<\infty$, there exists a constant $C=C(\Omega,p,\sphericalangle \cT)$ 
(which depends on $p$, $\Omega$, and the shape regularity of $\cT$) with
\[
\| \widehat{v} \|_{L^{\infty}(\Omega)}+\| \widehat{v} \|_{W^{1,p}(\cT)}\le C 
\trinl \widehat{v} \trinr_{\text{\rm pw}}\quad\text{for all }\widehat{v} \in H^2_0(\Omega)+  \cM(\cT).
\] 
\end{lem}

\begin{proof}The main observation is that the enrichment operator $E_M$ from 
Lemma~\ref{hctenrich} maps into the HCT finite element space  plus squared 
bubble-functions  \cite{DG_Morley_Eigen}; 
so  $v_M-E_Mv_M$ is a piecewise polynomial of degree at most $6$ for
any $v_M\in \cM(\cT)$
(with respect to some refinement of $\cT$, where each triangle $T$ is divided into three 
sub-triangles by connecting each vertex with its center of inertia). This leads to 
inverse estimates such as 
\[
| v_M-E_Mv_M |_{W^{1,\infty}(T)}\lesssim h_T^{-1} | v_M-E_Mv_M |_{H^{1}(T)}
\lesssim h_T^{-2} \| v_M-E_Mv_M \|_{L^{2}(T)}.
\]
Lemma~\ref{hctenrich}.b shows for $v\in H^1_0(\Omega)$ that the right-hand side is controlled by
 \[
 \|   h_\cT^{-2} (v_M-E_Mv_M) \|_{L^{2}(\Omega)}\lesssim 
 \trinl v_M-E_Mv_M \trinr_{\text{pw}}\le
 \Lambda 
 \min_{\varphi\in V}\trinl v_M-\varphi \trinr_{\text{pw}}\le  \Lambda  \trinl v+v_M \trinr_{\text{pw}}.
 \]
Since $T\in \cT$ is arbitrary, this proves
\begin{equation}\label{eqnewccstra1}
| v_M-E_Mv_M |_{W^{1,\infty}(\Omega,\cT)}\lesssim\trinl v+v_M \trinr_{\text{pw}}.
\end{equation}
Since  $v_M-E_Mv_M$ is Lipschitz continuous with  Lipschitz constant 
$| v_M-E_Mv_M |_{W^{1,\infty}(\cT)} $ and
 vanishes at the vertices of $T\in\cT$,  
\[
|| v_M-E_Mv_M ||_{L^{\infty}(\Omega)}
\le h_{\max} | v_M-E_Mv_M |_{W^{1,\infty}(\Omega,\cT)}
\] 
holds for the maximal mesh-size 
$h_{\max} \le \text{diam}(\Omega)$.
This and \eqref{eqnewccstra1}    imply  (with $C_1\approx 1$)
\[
\| v_M-E_Mv_M \|_{L^{\infty}(\Omega)}\le  C_1 \trinl v+v_M\trinr_{\text{pw}}.
\] 
The boundedness of  the  continuous 2D Sobolev embedding 
$ H^2(\Omega)\hookrightarrow  L^{\infty} (\Omega) $ leads to  $\|\bullet \|_{L^{\infty}(\Omega)}
\le C_2\,  \trinl \bullet \trinr$ in $H^2_0(\Omega)$.
Consequently, with  a triangle inequality in the beginning, 
\begin{align*}
\|v+v_M \|_{L^{\infty}(\Omega)}  &\le
 \|v_M-E_M v_M \|_{L^{\infty}(\Omega)}+
 \| v+ E_Mv_M \|_{L^{\infty}(\Omega)}\\
& \le  C_1\trinl v+v_M\trinr_{\text{pw}}+  C_2  \trinl v+ E_M v_M \trinr.
\end{align*}
The  triangle inequality and Lemma~\ref{hctenrich}.b (again with $\varphi=-v$) show
\begin{equation}\label{eqnewccstra3}
\trinl v+E_M v_M\trinr \le  \trinl v+v_M \trinr_{\text{pw}}+ \trinl v_M -E_Mv_M \trinr_{\text{pw}}
\lesssim  \trinl v+v_M \trinr_{\text{pw}}. 
\end{equation}
The combination of \eqref{eqnewccstra3} with the previously displayed estimate shows 
the first assertion 
$\|v+v_M \|_{L^{\infty}(\Omega)} \lesssim  \trinl v+v_M \trinr_{\text{pw}}$.
The proof of the second assertion is similar with 
\eqref{eqnewccstra1}-\eqref{eqnewccstra3}.  
The boundedness of  the  continuous 2D Sobolev embedding 
$ H^2(\Omega)\hookrightarrow W^{1,p}(\Omega) $ leads to 
$   | \bullet  |_{W^{1,p}(\Omega)}\le C(p,\Omega)\,  \trinl \bullet \trinr$ in $H^2_0(\Omega)$.
Consequently,
\begin{align*}
| v+v_M |_{W^{1,p}(\Omega,\cT)} &\le  | v+E_M v_M  |_{W^{1,p}(\Omega)}
+| v_M -E_Mv_M|_{W^{1,p}(\Omega,\cT)}\\
&\le C(p,\Omega)   \trinl v+E_M v_M \trinr+ |\Omega|^{1/p}
| v_M -E_Mv_M|_{W^{1,\infty}(\Omega,\cT)}
\end{align*}
with the area $|\Omega|\approx 1\approx C(p,\Omega)$. 
Recall  \eqref{eqnewccstra1}  and \eqref{eqnewccstra3} in the end to control the 
previous upper bound in terms of
$ \trinl v+E_M v_M \trinr+ \trinl  v+ v_M \trinr_{\text{pw}}\lesssim 
 \trinl  v+ v_M \trinr_{\text{pw}}$. This concludes the proof of the second assertion
 $| v+v_M |_{W^{1,p}(\Omega,\cT)}\lesssim  \trinl  v+ v_M \trinr_{\text{pw}}$.
\end{proof}

\begin{rem}[boundedness]\label{remarkccnew12324boundedness}
The bound for  $a_{\text{pw}}$ is immediate from 
\eqref{eqccnerwandlast1234a} for the norm  $\trinl\bullet\trinr_{\text{pw}}$
in $ \widehat{V}\equiv V+ \cM(\cT)$.  The bound   
$\| \Gamma_{\text{pw}}  \| =\sqrt{2} C(\Omega,4,\sphericalangle \cT)^2 $
in
\[
|  \Gamma_{\text{pw}}(\widehat \eta,\widehat{ \chi},\widehat{\phi})|\le 
 \| \Gamma_{\text{pw}}  \|
  \trinl\widehat \eta \trinr_{\text{pw}}
 \trinl\widehat{ \chi} \trinr_{\text{pw}}
 \trinl\widehat\phi \trinr_{\text{pw}}
 \quad\text{for all } \widehat \eta , \widehat{ \chi}, \widehat{\phi}\in 
  \widehat{V}\equiv V+ \cM(\cT)
\]
follows  from \eqref{eqccnerwandlast1234b} 
with H\"older inequalities and Lemma~\ref{lemmadiscreteembeddings}.
\end{rem}

\begin{lem}[\cite{BSZ2013}]\label{EnrichSmooth} For $1/2<s\le 1$ there 
exists a positive constant $C$ such that any $\eta\in H^{2+s}(\Omega)$
and  $\varphi_M\in \cM(\cT)$ satisfy 
$\displaystyle
a_{\text{pw}}(\eta,\varphi_M-E_M\varphi_M)\leq Ch_{\max}^{s}\|\eta\|_{H^{2+s}(\Omega)}\trinl \varphi_M\trinr_{\text{\rm pw}}.
$ \qed
\end{lem}

{\it Proof of Theorem~\ref{apriori_NS_est_NCFEM}}. 
Set $X=Y=V$, $X_h=Y_h=V_h$, $\widehat{X}=V+V_h$, $\widehat{a}(\bullet,\bullet):=a_{\text{pw}}(\bullet, \bullet)$, $ \widehat{b}(\bullet,\bullet):=\Gamma_{\text{pw}}( u,\bullet,\bullet)+\Gamma_{\text{pw}}(\bullet,u,\bullet)$ and $P=I_M$, $Q={\mathcal C}= E_M$. The regularity $u\in H^{2+s}(\Omega)$ of Subsection~\ref{subsectContinuousproblemNS} with $s>1/2$ allows for  the 
bounded global Sobolev embeddings   
$H^{2+s}(\Omega)$ $\hookrightarrow W^{2,4}(\Omega)\hookrightarrow W^{1,\infty}(\Omega)$.
This and   Lemma~\ref{lemmadiscreteembeddings}
lead for $\widehat{\theta}\in \widehat{X}$ and $\phi\in H^1(\Omega)$ to
\begin{align}
|\Gamma_{\text{pw}}(u,\widehat{\theta},\phi)|+|\Gamma_{\text{pw}}(\widehat{\theta},u,\phi)|
& \lesssim \left(  \|u\|_{W^{2,4}(\Omega)}\|\widehat{\theta}\|_{W^{1,4}(\Omega,\cT)}
+ \| u\|_{W^{1,\infty}(\Omega)}\ensuremath{| \!| \! |}\widehat{\theta}\ensuremath{| \!| \! |}_{\text{pw}}
\right)\|\phi\|_{H^1(\Omega)}  \notag\\
& \lesssim\|u\|_{H^{2+s}(\Omega) }\ensuremath{| \!| \! |}\widehat{\theta}\ensuremath{| \!| \! |}_{\text{pw}}\|\phi\|_{H^1(\Omega)}. \label{Gamma2bdd}
\end{align}
For $\theta_M\in \cM(\cT)$ with $\trinl\theta_M\trinr_{\text{pw}}=1$, the aforementioned estimates imply that $\widehat{b}(\theta_M,\bullet)\in H^{-1}(\Omega)$ 
and so the solution $z\in V$ to the biharmonic problem
\begin{equation*}
a(z,\phi)=\widehat{b}(\theta_M,\phi)\fl \phi\in V
\end{equation*}
satisfies $z\in H^{2+s}(\Omega)$ and $\| z \|_{H^{2+s}(\Omega)}\lesssim 1$
\cite{BlumRannacher}.
The regularity  $z\in H^{2+s}(\Omega)$ and Lemma~\ref{EnrichSmooth} 
(resp. Lemma~\ref{Morley_Interpolation}) 
imply  {\bf (H1)}  (resp.  {\bf (H2)}) 
with $\delta_1\lesssim h_{\rm max}^s$ (resp. $\delta_2\lesssim h_{\rm max}^s$).
The estimate \eqref{Gamma2bdd} and Lemma~\ref{hctenrich} verify {\bf (H3)} with 
$\delta_3\lesssim h_{\rm max}$.  Lemma~\ref{hctenrich}.b leads to {\bf (H4)}
with $\Lambda_4=\Lambda$. For any $y_M\in M(\cT)$ with 
$\ensuremath{| \!| \! |}y_M\ensuremath{| \!| \! |}_{\text{pw}}=1$, 
Lemma~\ref{EnrichSmooth} guarantees 
\[
a_{\text{pw}}(u, y_M-E_My_M)\lesssim h_{\max}^s \| u\|_{H^{2+s}(\Omega)}\approx h_{\max}^s,
\]
while  Lemma~\ref{hctenrich} shows 
\[
F(y_M-E_My_M)\lesssim h_{\max}^2\| f\|\lesssim  h_{\max}^s.
\] 
This implies  {\bf (H5)} with $\delta_5\lesssim  h_{\max}^s$. 
Choose $x_h=I_M u$ so that {\bf (H6)} holds with $\delta_6\lesssim h_{\rm max}^{s}$.

In conclusion,  for sufficiently small mesh-size $h_{\max}$, 
the discrete inf-sup inequality of Theorem~\ref{dis_inf_sup_thm} holds  with $\beta_h\ge \beta_0>0$. 
Moreover, Theorems~\ref{thm3.1} and \ref{err_apriori_thm_vke} apply and prove
{\bf (A)}-{\bf (C)}. To compute ${\rm apx}(\cT)=\| \widehat{N}(u)\|_{M(\cT)^*} $,  let 
$\phi_M \in M(\cT)$ satisfy  $\ensuremath{| \!| \! |}\phi_M\ensuremath{| \!| \! |}_{\text{pw}}=1$
and ${\rm apx}(\cT) = \widehat{N}(u; \phi_M) $. Since $N(u, E_M\phi_M)=0$, 
the difference 
$\widehat{\psi}
:=\phi_M-E_M \phi_M\in\widehat{V}$ satisfies 
\begin{align*}
{\rm apx}(\cT)  &= \widehat{N}(u;\widehat{\psi} ) 
= {a}_{\text{\rm pw}} (u- I_M u, \widehat{\psi})-F((1-\Pi_0)\widehat{\psi})+
\Gamma_{\text{\rm pw}}(u,u,\widehat{\psi}) 
\end{align*}
with  Lemma~\ref{hctenrich}.c for ${a}_{\text{\rm pw}} (I_M u, \widehat{\psi})=0$ and
$\Pi_0 \widehat{\psi}=0$ a.e. in the last step.
This, the finite overlap of $(\omega_K:K\in\cT)$ in Lemma~\ref{hctenrich}.a and \ref{hctenrich}.b 
 for $\ensuremath{| \!| \! |}\widehat{\psi}\ensuremath{| \!| \! |}_{\text{pw}}\lesssim 1$ 
lead to \eqref{eqnapriori_NS_est_NCFEM}.
 \qed

\subsection{A posteriori error estimate}
For any $K\in\cT$ and $E\in\cE$, define the volume and edge error estimators  by 
\begin{align*}
&\eta_K^2:=h_K^4\|{\rm curl}(-\Delta u_M\nabla u_M)-f\|_{L^2(K)}^2\text{ and }\\
&\eta_E^2 :=h_E\|\left[D^2 u_M\right]_E\tau_E\|_{L^2(E)}^2
+h_E^{3}\|[\Delta u_M\nabla u_M]_E\cdot\tau_E\|_{L^2(E)}^2
+h_E^{3} \|\{\Delta u_M\nabla u_M\}_E\cdot\tau_E\|_{L^2(E)}^2.
\end{align*}
Here and throughout this section, the average of $\widehat{\phi}\in\widehat{X}$ across the interior edge $E=\partial K_+\cap \partial K_-\in \cE(\Omega)$ shared by two triangles $K_\pm\in\cT$  reads
$\{\widehat{\phi}\}_E:=(\widehat{\phi}|_{K_+}+\widehat{\phi}|_{K_-})/2$,
while $\{\widehat{\phi}\}_E:=\widehat{\phi}|_E$ along any boundary edge $E\in \cE(\partial\Omega)$.

\begin{thm}[a posteriori]\label{reliability_NS_NC}
If $u\in V$ is a regular solution to \eqref{NS_weak}, then there exist positive $\delta,\epsilon$, and   $C_{\rm rel}$ 
such that, for any $\cT\in\bT(\delta)$, the discrete solution $u_{M}\in \cM(\cT)$ to \eqref{NS_dis_NC} 
with $\trinl  u-u_M\trinr_{\text{\rm pw}}\le \epsilon$  satisfies
\begin{align*}
C_{\rm rel}^{-2}\trinl u-u_{M}\trinr_{\text{\rm pw}}^2&\leq \sum_{K\in\cT}\eta_K^2+\sum_{E\in\cE}\eta_E^2.
\end{align*}
\end{thm}

\begin{proof} 
Let $u_M$ be the solution to \eqref{NS_dis_NC} close to $u$ and  
apply Theorem~\ref{abs_res_thm} with  
$X=Y=V,$ $X_h=Y_h=V_h$,  $v_h= u_M$, and $Q:=E_M$ from Lemma~\ref{hctenrich}. 
Suppose that  $\epsilon,\delta$ satisfy Theorem~\ref{apriori_NS_est_NCFEM} and,
if necessary, are chosen smaller
such that, for any $\cT\in\bT(\delta)$, exactly one discrete solution
$u_{\rm M}\in X_M$  to \eqref{NS_dis_NC} satisfies  $\trinl  u-u_{ M}  \trinr_{\text{\rm pw}} \le \epsilon
 \le \beta/( 2(1+\Lambda )\|\Gamma
 \|)$.  Lemma \ref{hctenrich}.b implies
$
\trinl   u_M-E_M u_M  \trinr_{\text{\rm pw}}  
\le \Lambda \trinl   u -u_M \trinr_{\text{\rm pw}}
\le  \Lambda \epsilon 
$.
This and  triangle inequalities show
\begin{align*}
\trinl E_M u _M\trinr + \trinl u_M\trinr_{\text{pw}} &\le 
 \trinl u_M-E_M u_M \trinr_{\text{pw}} + 2  \trinl u_M\trinr_{\text{pw}}
\le  2 \trinl u \trinr +(2+ \Lambda)\epsilon=:M;
\\
\trinl u- E_Mu_M\trinr &\le  \trinl  u- u_M\trinr_{\text{pw}}+ \trinl  u_M- E_M u_M\trinr_{\text{pw}}
\le (1+\Lambda)\epsilon
\le  \beta/( 2\|\Gamma
\|).
\end{align*}
Consequently,  the abstract residual \eqref{relib_eqn} in Theorem~\ref{abs_res_thm} implies
\begin{align}\label{eqccforold4.14}
\trinl u- u_M\trinr_{\text{pw}}\le 2\beta^{-1} \|N(E_M u_M)\|_{ V^*}
+\trinl  u_M- E_M u_M\trinr_{\text{pw}}.
\end{align}
There exists some  $\phi\in V$ with $\trinl \phi\trinr=1$ and
\begin{align*}
&\|N(E_Mu_M)\|_{V^*}= N(E_M u_M;\phi)= a(E_M u_M,\phi)-F(\phi)
+\Gamma(E_M u_M,E_M u_M,\phi)\notag\\
&=\widehat{N}( u_M;\phi)+ a_{\text{pw}}(E_M u_M- u_M,\phi)+\Gamma(E_M u_M,E_M u_M,\phi)
-\Gamma_{\text{pw}}( u_M, u_M,\phi)
\end{align*}
with the definition of $N$ and of   $\widehat{N}$. This,
the  bound  of  $a_{\text{pw}}$,    elementary arguments  with the trilinear form
and its bound $\|\Gamma_{\rm pw}\|$ from Remark~\ref{remarkccnew12324boundedness},   
and $M$ prove
\begin{align} \label{eqccforold4.14b}
\|N(E_M u_M)\|_{ V^*}\le \widehat{N}( u_M;\phi)+(1+ M \|\Gamma_{\text{\rm pw}}\|)
\trinl u_M- E_M  u_M\trinr_{\text{\rm pw}}.
\end{align} 
Since  $u_M$  solves  \eqref{NS_dis_NC},  
$N_h( u_M;\phi)=N_h( u_M;\chi)$ holds for $\chi:=\phi-I_M \phi$ with 
the  Morley interpolation $ I_M \phi$ of $\phi$. 
Since Lemma~\ref{Morley_Interpolation}.a  implies $a_{\text{pw}}( u_{M},\phi-I_M\phi)=0$, 
an integration by parts in the nonlinear term $\Gamma_{\text{pw}}(\bullet,\bullet,\bullet)$ leads to
\begin{align*}
\widehat{N}( u_M;\phi)&=\sum_{K\in\cT}\int_K \Delta u_M\nabla u_M\cdot{\rm curl}\chi\dx
-F(\chi)\notag\\
&=\sum_{K\in\cT}\int_K ({\rm curl}(-\Delta u_M\nabla u_M)-f)\chi\dx
+\sum_{E\in\cE}\int_{E} [\Delta u_M\nabla u_M\cdot\tau_E]_E\{\chi\}_E\ds\notag\\
&\quad+\sum_{E\in\cE}\int_{E} \{\Delta u_M\nabla u_M\}_E\cdot\tau_E
\; [\chi]_E\ds. 
\end{align*}
This and standard arguments with Cauchy and trace inequalities
plus  Lemma~\ref{Morley_Interpolation}.b with $\trinl \phi\trinr=1$ eventually lead to
some constant $C_A\approx 1$ with
\begin{align}\label{eqccverlateinMayno33extra}
C_A ^{-2} \widehat{N}( u_M;\phi)^2 \le  \sum_{K\in\cT}\eta_K^2+\sum_{E\in\cE}\eta_E^2.
\end{align}
Piecewise inverse estimates  
$\trinl u_M- E_M  u_M\trinr_{\text{pw}}\lesssim \|  h_\cT^{-2} ( u_M- E_M  u_M )\|_{L^2(\Omega)}$
and  Lemma~\ref{hctenrich}.b with the tangential jump residuals 
lead to some constant $C_B\approx 1$ with 
\begin{align}\label{eqccverlateinMayno33}
 C_B ^{-2} \trinl u_M- E_M  u_M\trinr_{\text{pw}}^2 \le
\sum_{E\in\cE} h_E \|[D^2 u_M]_E\tau_E\|_{L^2(E)}^2.
\end{align}
This is bounded by $\sum_{E\in\cE}\eta_E^2$.
The combination of  \eqref{eqccforold4.14}-\eqref{eqccverlateinMayno33}
concludes the proof with 
 $C_{\rm rel}=   2 \beta^{-1} C_A+  (1+ 2\beta^{-1}(1+M \|\Gamma _{\text{pw}}  \|)) C_B $.
\end{proof}

\begin{rem}[residuals develop correct convergence rate]
The efficiency of the estimator remains  as an open question  owing to the average term 
$\left\|\{\Delta u_M\nabla u_M\cdot\tau_E\}_E\right\|_{L^2(E)}$ in $\eta_E$ 
(for the remaining contributions
are efficient). 
The sum of all those contributions associated to those terms,  however, converge
 (at least) with linear rate in  that
 \begin{equation}\label{extrahigherorderremarkapostccnew1}
 S:= (\sum_{E\in\cE}h_E^{3}\left\|\{\Delta u_M\nabla u_M\cdot\tau_E\}_E\right\|_{L^2(E)}^2
 )^{1/2}  \lesssim  h_{\max}  \| u \|_{H^{2+s}(\Omega)} \trinl u_M\trinr_{\text{pw}}
 =O(h_{\max}).
 \end{equation}
 Before a sketch of the proof concludes this remark, it should be stressed that 
 \eqref{extrahigherorderremarkapostccnew1}  can be 
 a higher-order term:  Consider a uniform mesh in a singular situation with 
 re-entering corners (with  an exact solution of reduced regularity  
 $u\notin H^3(\Omega) $ \cite{BlumRannacher}) with a  suboptimal convergence rate
$s<1$. Then   $S$ in  \eqref{extrahigherorderremarkapostccnew1} is of higher-order. 

The proof of \eqref{extrahigherorderremarkapostccnew1} starts with a  triangle inequality 
\[
S^2 
 \le  \sum_{T\in \cT}\sum_{E\in \cE(T)}
 h_E^{3}\left\| (\Delta u_M\nabla u_M) |_T\right\|_{L^2(E)}^2.
\]
The discrete trace inequality  (i.e. a trace inequality followed by an inverse inequality) 
for each summand shows
 \[
 h_E^{3}\left\| (\Delta u_M\nabla u_M) |_T\right\|_{L^2(E)}^2
 \lesssim h_E^{2}\left\| \Delta u_M\nabla u_M\right\|_{L^2(T)}^2.
 \] 
Recall  the piecewise constant mesh-size  $h_{\cT}\in P_0(\cT)$,
$h_{\cT}|_T:=h_T:=\text{diam}(T)   $ for $T\in\cT$, with maximum
$h_{\max}:=\max h_\cT  \le\delta$.
The shape regularity of $\cT$ shows
\[
S\lesssim \left\|  h_\cT  \Delta_{\text{pw}} u_M\nabla_{\text{pw}} u_M\right\|_{L^2(\Omega)}
\le  \left\|  h_\cT \Delta u \nabla_{\text{pw}} u_M\right\|_{L^2(\Omega)} +
 \left\|  h_\cT  \Delta_{\text{pw}} (u-u_M)\nabla_{\text{pw}} u_M\right\|_{L^2(\Omega)}
\]
with a triangle inequality in the last step.
Recall that $u\in H^{2+s}(\Omega)$ for $s>1/2$ enables 
the bounded embedding $H^s(\Omega)\hookrightarrow L^{2p}(\Omega)$
for any $p$ with $1<p<1/(1-s)$.
This and a H\"older inequality with $1/p+1/p'=1$ leads to
\[
  \left\| \Delta u \nabla_{\text{pw}} u_M\right\|_{L^2(\Omega)} \le 
    \| \Delta u \|_{L^{2p}(\Omega)}
      \left\| \nabla_{\text{pw}} u_M\right\|_{L^{2p'}(\Omega)}.
\]
Lemma~\ref{lemmadiscreteembeddings} shows that the last term is controlled 
by  $\trinl u_M\trinr_{\text{pw}}$. Consequently,
\[
 \left\| h_\cT \Delta u \nabla_{\text{pw}} u_M\right\|_{L^2(\Omega)}
 \lesssim h_{\max} 
  \| u \|_{H^{2+s}(\Omega)}\trinl u_M\trinr_{\text{pw}}. 
\]
The analysis of the second term  starts with $0<r\le s\le 1$ and 
the elementary observation 
\[
\left\|  h_\cT  \Delta_{\text{pw}} (u-u_M)\nabla_{\text{pw}} u_M\right\|_{L^2(\Omega)}
\le  \sqrt{2} h_{\max}^{1-r}   \trinl u-u_M\trinr_{\text{pw}} 
| h_\cT^r \,  u_M |_{W^{1,\infty}(\Omega,\cT)}.
\]
The  asserted convergence rate follows  with
 $\trinl u-u_M\trinr_{\text{pw}}\lesssim  h_{\max}^s \| u \|_{H^{2+s}(\Omega)}$. The maximum of the remaining term 
$ | h_\cT^r \,  u_M |_{W^{1,\infty}(\Omega,\cT)}=h_T^r |   u_M |_{W^{1,\infty}(T)}$
is attained for (at least) one $T\in\cT$. An inverse inequality and 
Lemma~\ref{lemmadiscreteembeddings}  in the end show
\[
h_T^r |   u_M |_{W^{1,\infty}(T)}\lesssim | u_M |_{W^{1,2/r}(T)} \le 
| u_M |_{W^{1,2/r}(\Omega,\cT)}\lesssim  \trinl u_M\trinr_{\text{pw}}. 
\]
Consequently, 
$\left\|  h_\cT  \Delta_{\text{pw}} (u-u_M)\nabla_{\text{pw}} u_M\right\|_{L^2(\Omega)}
\lesssim  h_{\max}^{1+s-r} \trinl u_M\trinr_{\text{pw}}$.
The combination of the previous estimates proves  \eqref{extrahigherorderremarkapostccnew1}. 
\qed 
\end{rem}

\begin{rem}[no efficiency analysis]
The lack of local efficiency  is part of a  more general structural 
difficulty. Whenever volume terms require a piecewise integration by parts with Morley  
finite element test functions, there arise average 
terms like $\{\widehat{\phi}\}_E$ in Theorem~\ref{reliability_NS_NC}, which are not residuals.
This prevents an efficiency analysis in this section as well as in \cite[Adini FEM]{CarstensenGallistlHu2013} or 
\cite[Subsect 7.8]{Gallistl2014Adaptive}. It is left as an open problem for future research and may cause 
a modification of the discrete scheme. In the vibration of a biharmonic plate or in the von K\'{a}rm\'{a}n
equations of the subsequent section, this difficulty does not arise.
\end{rem}

\section{Von K\'{a}rm\'{a}n equations}\label{Preli}
Given a load function $f\in L^{2}(\Omega)$, the von K\'{a}rm\'{a}n equations model the deflection of a very thin elastic plate with vertical displacement $u\in\hto$ and the Airy stress function $v\in\hto$ such that
\begin{equation}\label{vke}
\Delta^2 u =[u,v]+f \text{ and }\Delta^2 v =-\half[u,u] \text{ in } \Omega.
\end{equation}
With  the co-factor matrix  $\cof(D^2 v )$  of $D^2 v $, the von K\'{a}rm\'{a}n brackets read
\begin{equation*}
[u,v]:=\frac{\partial^2 u}{\partial x_1^2}\frac{\partial^2 v}{\partial x_2^2}
+\frac{\partial^2 u}{\partial x_2^2}\frac{\partial^2 v }{\partial x_1^2}
-2\frac{\partial^2 u}{\partial x_1\partial x_2}\frac{\partial^2 v}{\partial x_1\partial x_2}=\cof(D^2 u):D^2 v.
\end{equation*}

\subsection{Continuous problem}
The  weak formulation  of  the von K\'{a}rm\'{a}n equations \eqref{vke}
seeks  $u,v\in V:=H^2_0(\Omega)$ with
\begin{subequations}\label{wform}
\begin{align}
a(u,\varphi_1)+ \gamma(u,v,\varphi_1)+\gamma(v,u,\varphi_1)&=
(f,\varphi_1)_{L^2(\Omega)}   \fl\varphi_1\in V\label{wforma}\\
a(v,\varphi_2)-\gamma(u,u,\varphi_2)   &=0            \fl\varphi_2 \in V.
\label{wformb}
\end{align}
\end{subequations} 
Here and throughout this section  abbreviate, for all $ \eta,\chi,\varphi\in V$, 
\[
a(\eta,\chi):=\integ D^2 \eta:D^2\chi\dx  \quad\text{and} \quad 
 \gamma(\eta,\chi,\varphi):=-\half\integ [\eta,\chi]\varphi\dx.
\]
The abstract theory of Sections~\ref{infsup}-\ref{error} applies for the real
Hilbert space  $X:=V\times V$  with its dual $X^*$ to  the operator $ N:  X\to  X^*$  defined by
\begin{align}
N({\boldmath\Psi};{\boldmath\Phi}):=\langle  N({\boldmath\Psi}), \Phi\rangle:= 
A({\boldmath\Psi},{\boldmath\Phi})-F({\boldmath\Phi})+\Gamma({\boldmath\Psi},{\boldmath\Psi},{\boldmath\Phi})
\end{align}
for all ${\boldmath\Xi}=(\xi_{1},\xi_{2})$, ${\boldmath\Theta}=(\theta_{1},\theta_{2})$, 
${\boldmath \Phi}=(\varphi_{1},\varphi_{2})\in  X$ 
and  the abbreviations 
\begin{align*}
A(\Theta,{\boldmath \Phi})&:=a(\theta_1,\varphi_1)+a(\theta_2,\varphi_2),\\
F({\boldmath\Phi})&:=(f,\varphi_1)_{\lt},\\
\Gamma({\boldmath\Xi},\Theta,{\boldmath\Phi})&:=\gamma(\xi_{1},\theta_{2},\varphi_{1})+\gamma(\xi_{2},\theta_{1},\varphi_{1})-\gamma(\xi_{1},\theta_{1},\varphi_{2}).
\end{align*}
Note that   $A(\bullet,\bullet)$ is a scalar product in $X$  and the trilinear 
form $\Gamma(\bullet,\bullet,\bullet)$ is bounded \cite{GMNN_CFEM}.

It is known \cite{CiarletPlates,Brezzi} that there exist
a solution $\Psi \in X$ with $N(\Psi)=0$. Any solution has the regularity 
$\Psi\in {\bf H}^{2+s}(\Omega):=(H^{2+\alpha}(\Omega))^2$ for  $1/2<s\le 1$ depending on the polygonal
bounded Lipschitz domain $\Omega$ \cite{BlumRannacher}. This allows for the
boundedness 
\[
 \Gamma({\boldmath\Psi},{\boldmath\Theta},{\boldmath\Phi})\leq C \|{\boldmath\Psi}\|_{H^{2+s}(\Omega)}
\trinl\Theta\trinr  \|{\boldmath\Phi}\|_{H^{1}(\Omega)}
\quad\text{for any }{\boldmath\Theta}\in X\text{ and }{\boldmath\Phi} \in H^1_0(\Omega;\bR^2).
\]

\subsection{Conforming FEM}\label{Sec:CFEM_vke}
With the notation of Section~\ref{sec:CFEM_NS} on $V_C\subset H^2_0(\Omega)$, the conforming 
finite element formulation seeks ${\boldmath\Psi}_C=(u_C,v_C)\in X_{h}:=V_C\times V_C$ such that
\begin{align} \label{vformdC}
\displaystyle 
N({\boldmath\Psi}_C;{\boldmath\Phi}_C)=0 \quad\text{for all} \quad {\boldmath\Phi}_C \in  X_h.
\end{align}

\begin{thm}[a priori]\label{apriori_NS_est_NC}
If ${\boldmath\Psi}\in X$ is a regular solution to $N({\boldmath\Psi})=0$, then  
there exist positive  $\epsilon$, $\delta$, and $\rho$ such that  {\bf (A)}-{\bf (C)} hold with
${\rm apx}(\cT)\equiv 0$  for all $\cT \in \bT(\delta)$. 
\end{thm}

\begin{proof}
The proof  is analogous to that of Theorem~\ref{apriori_NS_est_C} 
and hence omitted. The a~priori error analysis is derived
in  \cite{GMNN_CFEM} with a fixed point iteration (of linear convergence).
\end{proof}

For any $K\in\cT$ and $E\in\cE$, define  the volume and edge error
estimators by 
\begin{align*} 
\eta_K^2&:=h_K^4\big{\|}\Delta^2 u_C-[u_C,v_C]-f\big{\|}^2_{L^2(K)}
+h_K^4 \big{\|}\Delta^2 v_C+1/2[u_C,u_C]\big{\|}^2_{L^2(K)},\\
\eta_E^2&:=h_E^{3}\left\|\left[\divc (D^2   u_C)\right]_E\cdot\nu_E\right\|^2_{L^2(E)}
+h_E^{3} \left\|\left[\divc (D^2   v_C)\right]_E\cdot\nu_E\right\|^2_{L^2(E)}\notag\\ 
&\qquad  +h_E\left\|\left[D^2 u_C\nu_E\right]_E\cdot\nu_E\right\|^2_{L^2(E)}
+h_E\left\|\left[D^2 v_C\nu_E\right]_E\cdot\nu_E\right\|^2_{L^2(E)} . 
\end{align*}

\begin{thm}[a posteriori]\label{reliability_C}
If ${\boldmath\Psi}\in X$  is a regular solution  to $N({\boldmath\Psi})=0$,  then there exist positive 
$\delta, \epsilon$, $C_{\rm rel}$, and $C_{\rm eff}$ 
such that, for all $\cT \in \bT(\delta)$, the unique discrete solution  
${\boldmath\Psi}_{C}=(u_C,v_C)\in  X_h$ to \eqref{vformdC} with $  \trinl{\boldmath\Psi}-{\boldmath\Psi}_C\trinr   <\epsilon  $  satisfies 
\begin{align}
C_{\rm rel}^{-2}\trinl{\boldmath\Psi}-{\boldmath\Psi}_C\trinr^2&\leq\sum_{K\in\cT}\eta_K^2
+\sum_{E\in\cE }\eta_E^2\leq C_{\rm eff}^2(\trinl{\boldmath\Psi}-{\boldmath\Psi}_{C}\trinr^2+{\rm osc}_0^2(f)).
\label{reliability_est_C}
\end{align}
\end{thm}
 
\begin{proof}
For $Y=X$,  $Y_h=X_h$, we proceed as in the proof of Theorem \ref{reliability_C_NS}
and, for sufficiently small $\delta$, derive  {\bf (H1)}- {\bf (H6)} and  
$\trinl u-u_C\trinr < \beta/\|\Gamma\|$ from 
Theorem~\ref{apriori_NS_est_C}.  
Hence Corollary~\ref{coraposteriori} implies for $v_h\equiv {\boldmath\Psi}_C= (u_C, v_C)$ that
\begin{align*}
\trinl{\boldmath\Psi}-{\boldmath\Psi}_C\trinr\lesssim \|N({\boldmath\Psi}_C)\|_{X^*}=N({\boldmath\Psi}_C;{\boldmath\Phi})
\end{align*}
for some  ${\boldmath\Phi}\in X$ with $\trinl{\boldmath\Phi}\trinr=1$ and its approximation 
$\Pi_h{\boldmath\Phi}\in X_h$
($\Pi_h$ from Lemma~\ref{interpolation_BFS} applies componentwise). 
Abbreviate $(\chi_1,\chi_2):={ \boldmath \chi} :={\boldmath\Phi}-\Pi_h{\boldmath\Phi}$ and deduce from  
 \eqref{vformdC} that 
$\|N({\boldmath\Psi}_C)\|_{X^*}=N({\boldmath\Psi}_C;(\chi_1,\chi_2))$.
Successive integrations by parts show
\begin{align*}
&A({\boldmath\Psi}_C,{ \boldmath \chi} )=\sik (\Delta^2 u_C) \chi_1  \dx + \sik(\Delta^2 v_C) \chi_2  \dx \notag \\ 
&\quad +\sie\left[D^2 u_C\right]_E\nu_E\cdot\nabla\chi_1\ds
+ \sie\left[D^2 v_C\right]_E\nu_E\cdot\nabla\chi_2 \ds \notag \\
&\quad- \sie \chi_1 \left[\divc(D^2 u_C)\right]_E\cdot\nu_E\ds
- \sie  \chi_2 \left[\divc(D^2 v_C)\right]_E\cdot\nu_E \ds.
\end{align*}
This and the definition of $\Gamma(\bullet,\bullet,\bullet)$ lead to the residual
\begin{align}
&A({\boldmath\Psi}_C,{\boldmath\Phi}-\Pi_h{\boldmath\Phi})-F({\boldmath\Phi}-\Pi_h{\boldmath\Phi})+\Gamma({\boldmath\Psi}_C,{\boldmath\Psi}_C,{\boldmath\Phi}-\Pi_h{\boldmath\Phi})\notag\\
&\quad= \sik\left(\Delta^2 u_C-[u_C,v_C]-f\right)\chi_1\dx+\sik \big(\Delta^2 v_C+\half[u_C,u_C]\big)\chi_2\dx\notag\\
&\quad\quad-\sie\left(\left[\divc(D^2 u_C)\right]_E\cdot\nu_E\right)\chi_1\ds
+\sie\left[D^2 u_C\right]_E\nu_E\cdot\nabla\chi_1\ds\notag\\
&\quad\quad -\sie\left(\left[\divc(D^2 v_C)\right]_E\cdot\nu_E\right)\chi_2\ds
+ \sie\left[D^2 v_C\right]_E\nu_E\cdot\nabla\chi_2\ds. \label{int_estimator}
\end{align}
The two edge terms in the above expression that involve $\nabla \chi_j $ for $j=1,2$  can be rewritten as    
\begin{align} 
&\sie\left[D^2 u_C\right]_E\nu_E\cdot\nabla\chi_1\ds+\sie\left[D^2 v_C\right]_E\nu_E\cdot\nabla\chi_2\ds\notag\\
&\quad=\sie\left[D^2 u_C\nu_E\right]_E\cdot\nu_E \frac{\partial\chi_1}{\partial \nu}\ds+\sie\left[D^2 v_C\nu_E\right]_E\cdot\nu_E \frac{\partial\chi_2}{\partial \nu}\ds \nonumber \\
&\quad\quad +\sie\left[D^2 u_C\nu_E\right]_E\cdot\tau_E \frac{\partial\chi_1}{\partial \tau}\ds+\sie\left[D^2 v_C\nu_E\right]_E\cdot\tau_E \frac{\partial\chi_2}{\partial \tau}\ds. \nonumber
\end{align}
The last two terms involve tangential derivatives and so vanish for $u_C$ and $v_C$ belong to $\hto$.
Standard arguments analogous to \cite[(5.12)-(5.14)]{CCGMNN18} with a Cauchy  inequality, an inverse inequality, and Lemma \ref{interpolation_BFS}  conclude the proof of the reliability. 

The proof of the efficiency of the volume term $\eta_K$ is immediately adopted from  that of 
\cite[Lemma 5.3]{CCGMNN18}. The arguments in the proof of efficiency for the edge terms 
$h_E\left\|\left[D^2 u_C\nu_E\right]_E\cdot\nu_E\right\|_{L^2(E)}$  
and $h_E\left\|\left[D^2 v_C\nu_E\right]_E\cdot\nu_E\right\|_{L^2(E)}$ are the same 
as for the (linear) biharmonic  equation and can be adopted from 
\cite[Theorem 4.4 ]{Georgoulis2011} or \cite[Theorem 6.2]{CCGMNN18}. 
Further  details are omitted. 
\end{proof}

\subsection{Morley FEM}\label{Sec:NCFEM_vke}
The Morley FEM
seeks ${\boldmath\Psi}_{M}\in  X_M:={\cM(\cT) \times \cM(\cT)}\subset \widehat{X}:=X+X_M$ 
(endowed with the norm $\trinl\bullet \trinr_{\text{pw}}$) such that
\begin{equation}\label{vformdNC}
N_h({\boldmath\Psi}_{M};{\boldmath\Phi}_M):=A_{\text{pw}}({\boldmath\Psi}_M,{\boldmath\Phi}_M)
+\Gamma_{\text{pw}}({\boldmath\Psi}_M,{\boldmath\Psi}_M,{\boldmath\Phi}_M)-F({\boldmath\Phi}_M)=0 \fl {\boldmath\Phi}_M \in  X_M.
\end{equation}
Here and throughout this subsection,  for all $ {\boldmath\Xi}=(\xi_{1},\xi_{2})$, 
$\Theta=(\theta_{1},\theta_{2})$, ${\boldmath\Phi}=(\varphi_{1},\varphi_{2})\in   \widehat{X}$,
\begin{align*}
& A_{\text{pw}}(\Theta,{\boldmath\Phi}):=a_{\text{pw}}(\theta_1,\varphi_1)+a_{\text{pw}}(\theta_2,\varphi_2), 
 \; F({\boldmath\Phi}):=\sit f\varphi_1\dx, \\
 &\Gamma_{\text{pw}}({\boldmath\Xi},\Theta,{\boldmath\Phi}):=b_{\text{pw}}(\xi_{1},\theta_{2},\varphi_{1})+b_{\text{pw}}(\xi_{2},\theta_{1},\varphi_{1})-b_{\text{pw}}(\xi_{1},\theta_{1},\varphi_{2}),
\end{align*} 
and, for all $ \eta,\chi,  \varphi  \in \widehat{V}:=H^2_0(\Omega)+ \cM(\cT)$,
\begin{equation*}
 a_{\text{pw}}(\eta,\chi):=\sit D^2 \eta:D^2\chi \dx\text{ and } 
 b_{\text{pw}}(\eta,\chi,\varphi):=-\half\sit [\eta,\chi]\varphi \dx.
\end{equation*} 
(The boundedness of  $a_{\text{pw}}$ is immediate and  that of $\Gamma_{\text{pw}} $ 
follows from Lemma~\ref{lemmadiscreteembeddings}.) 

\begin{thm}[a priori]\label{apriori_VKE_est_NC}
If ${\boldmath\Psi}\in X$ is a regular solution to $N({\boldmath\Psi})=0$, then  
there exist positive  $\epsilon$, $\delta$, and $\rho$ such that  {\bf (A)}-{\bf (C)} hold
for any $\cT \in \bT(\delta)$ with 
\[
{\rm apx(\cT)}\lesssim  \trinl   {\boldmath\Psi}-I_{M} {\boldmath\Psi} \trinr_{\rm pw} + \text{\rm osc}_0(f+ [u,v],\cT)
+ \text{\rm osc}_0([u,u],\cT)\lesssim  h_{\rm max}^{s}.
\]
\end{thm}

\begin{proof} 
Set $Y=X$, $Y_h=X_M$, $\widehat{X}= X+X_M$, $\widehat{a}(\bullet,\bullet):=A_{\text{pw}}(\bullet,\bullet),$ $\widehat{b}(\bullet,\bullet):=2\Gamma_{\text{pw}}( {\boldmath\Psi},\bullet,\bullet)$ and $P=I_M$,  $Q={\mathcal C}=E_M$. 
Given  ${\boldmath\Psi}\in {\bf H}^{2+s}(\Omega)$,  
$\widehat{\Theta},\: \widehat{{\boldmath\Phi}}\in \widehat{X}$, piecewise \Holder inequalities and 
the bounded global  Sobolev imbedding  $  H^{2+s}(\Omega)\hookrightarrow W^{2,4}(\Omega)$ (for $s>1/2$)
show
\begin{align} \label{Gamma3bdd}
\Gamma_{\text{pw}}({\boldmath\Psi},\widehat{\theta},\widehat{{\boldmath\Phi}})
\lesssim \|{\boldmath\Psi}\|_{H^{2+s}(\Omega)}\ensuremath{| \!| \! |}\widehat{\Theta}\ensuremath{| \!| \! |}_{\text{pw}}\|\widehat{{\boldmath\Phi}}\|_{L^4(\Omega)}. 
\end{align}
For $\Theta_M \in X_M$ with $\trinl\Theta_M\trinr_{\text{pw}}=1$, 
the linear functional $\Gamma({\boldmath\Psi},\Theta_M,\bullet)\in {\bf H}^{-1}(\Omega)$
leads to a  unique solution ${\boldmath Z} \in  X$ to the biharmonic problem 
$A({\boldmath Z},{\boldmath\Phi})=\Gamma({\boldmath\Psi},\Theta_M,{\boldmath\Phi})$ for all ${\boldmath\Phi}\in X$ 
with ${\boldmath Z}\in {\bf H}^{2+s}(\Omega)$ \cite{BlumRannacher}. For  $\varphi_M\in \cM(\cT)$, the inverse estimate
\begin{equation*}
\|\varphi_M-E_M\varphi_M\|_{L^4(K)}\leq C h_K^{-1/2}\|\varphi_M-E_M\varphi_M\|_{L^2(K)} \; \text {for all} \;  K\in\cT,
\end{equation*}
the bound for $\Gamma_{\text{pw}}$, and 
Lemma~\ref{hctenrich}.a imply $\delta_3\lesssim h_{\rm max}^{3/2}$.  The 
remaining conditions for the parameters in the {\bf (H1)}-{\bf (H2)} and {\bf (H4)}-{\bf (H6)} 
are verified as in the proof of Theorem~\ref{apriori_NS_est_NCFEM}.
For some ${\boldmath\Phi}_M\in X_M$ with $\trinl{\boldmath\Phi}_M\trinr_{\text{pw}}=1$,  
${\rm apx}(\cT)=\|\widehat{N}({\boldmath\Psi})\|_{X_h^*}=\widehat{N}({\boldmath\Psi}; {\boldmath\Phi}_M) $.
This, $N({\boldmath\Psi}; E_M{\boldmath\Phi}_M)=0$,  \eqref{Gamma3bdd},  and 
Lemmas \ref{hctenrich}-\ref{EnrichSmooth} 
lead for $(\chi_1,\chi_2):={\boldmath\chi}:={\boldmath\Phi}_M-E_M{\boldmath\Phi}_M$ to
\begin{align*}
{\rm apx}(\cT) & = \widehat{N}({\boldmath\Psi}; {\boldmath\Phi}_M-E_M {\boldmath\Phi}_M) 
= A_{\text{pw}}({\boldmath\Psi}, {\boldmath\chi} )-F({\boldmath\chi})
+\Gamma_{\text{pw}}({\boldmath\Psi},{\boldmath\Psi},{\boldmath\chi}) \\
& = A_{\text{pw}}({\boldmath\Psi}-I_M {\boldmath\Psi},{\boldmath\chi})
- (f+ [u,v], \chi_1)_{L^2(\Omega)}+\frac 12  ([u,u], \chi_2)_{L^2(\Omega)}\\
&\lesssim   \trinl   {\boldmath\Psi}-I_M {\boldmath\Psi} \trinr_{\rm pw} + \text{osc}_0(f+ [u,v],\cT)
+ \text{osc}_0([u,u],\cT)\lesssim  h_{\rm max}^{s}
\end{align*}
with arguments as in  the final part of the proof  of Theorem~\ref{apriori_NS_est_NCFEM}.
Hence Theorems~\ref{thm3.1} and \ref{err_apriori_thm_vke} 
apply and prove {\bf (A)}-{\bf (C)}.
\end{proof}

For any $K\in\cT$ and $E\in\cE$, define the volume and edge 
error estimators  by 
\begin{align*}
\eta_K^2&:= h_K^4\left\|[u_M,v_M]+f\right\|_{L^2(K)}^2+
h_K^4 \left\|[u_M,u_M]\right\|_{L^2(K)}^2,\\
\eta_E^2 &:= h_E\left\|\left[D^2 u_M\right]_E\tau_E\right\|_{L^2(E)}^2
+h_E\left\|\left[D^2 v_M\right]_E\tau_E\right\|_{L^2(E)}^2.
\end{align*} 

\begin{thm}[a posteriori]\label{reliability_NC}
If ${\boldmath\Psi}=(u,v)\in  X$ is a regular solution to $N({\boldmath\Psi})=0$, then there exist 
$\delta,\epsilon$, $C_{\rm rel}$, and $C_{\rm eff}$ such that,
for any $\cT\in\bT(\delta)$, the discrete solution
${\boldmath\Psi}_{M}=(u_{M},v_{M})\in X_M:=\cM(\cT)\times \cM(\cT)$ to \eqref{vformdNC}
with $\trinl  {\boldmath\Psi}-{\boldmath\Psi}_{ M} \trinr_{\text{\rm pw}}\le \epsilon$  satisfies
\begin{align*}
	C_{\rm rel}^{-2}
	\trinl{\boldmath\Psi}-{\boldmath\Psi}_{ M}\trinr_{\text{\rm pw}}^2
	&\leq \sum_{K\in\cT}\eta_K^2+\sum_{E\in\cE}\eta_E^2\leq C_{\rm eff}^2 
	(\trinl {\boldmath \Psi}-{\boldmath\Psi}_{M}\trinr_{\text{\rm pw}}^2+{\rm osc}_0^2(f)).
\end{align*}
\end{thm}

\begin{proof}
Let ${\boldmath\Psi}_M$ be the solution to \eqref{vformdNC} close to ${\boldmath\Psi}$ and  
apply Theorem~\ref{abs_res_thm} with  $Y=X, Y_h=X_M$,  $v_h={\boldmath\Psi}_M$, and $Q=E_M$.
Suppose that  $\epsilon,\delta$ satisfy Theorem~\ref{apriori_VKE_est_NC} and,
if necessary, are chosen smaller such that, for any $\cT\in\bT(\delta)$, exactly one discrete solution
${\boldmath\Psi}_{M}\in X_M$  to \eqref{vformdNC} satisfies 
 $\trinl  {\boldmath\Psi}-{\boldmath\Psi}_{M}  \trinr_{\text{\rm pw}} \le \epsilon
 \le \beta/( 2(1+\Lambda )\| \Gamma
 \|)$. Lemma \ref{hctenrich}.b implies
 $
\trinl {\boldmath\Psi}_{ M}- E_{M} {\boldmath\Psi}_{ M}  \trinr_{\text{\rm pw}}  
\le \Lambda \trinl  {\boldmath\Psi}-{\boldmath\Psi}_{M}  \trinr_{\text{\rm pw}}
\le  \Lambda \epsilon 
$.
This and  triangle inequalities show
\begin{align*}
\trinl E_M{\boldmath\Psi}_M\trinr + \trinl {\boldmath\Psi}_M\trinr_{\text{pw}} &\le 
 \trinl {\boldmath\Psi}_M-E_M{\boldmath\Psi}_M \trinr_{\text{pw}} + 2  \trinl {\boldmath\Psi}_M\trinr_{\text{pw}}
\le  2 \trinl {\boldmath\Psi}\trinr +(2+ \Lambda)\epsilon=:M;   \\
\trinl {\boldmath\Psi}- E_M{\boldmath\Psi}_M\trinr &\le  \trinl  {\boldmath\Psi}- {\boldmath\Psi}_M\trinr_{\text{pw}}+ \trinl {\boldmath\Psi}_M- E_M{\boldmath\Psi}_M\trinr_{\text{pw}} 
\le (1+\Lambda)\epsilon \le  \beta/( 2\|\Gamma
 \|).
\end{align*}
Consequently,  the abstract residual \eqref{relib_eqn} in Theorem~\ref{abs_res_thm} implies
\begin{align}\label{eqccverlateinMayno123}
\trinl{\boldmath\Psi}-{\boldmath\Psi}_M\trinr_{\text{pw}}\leq  2\beta^{-1} 
\|N(E_M{\boldmath\Psi}_M)\|_{ X^*}+\trinl {\boldmath\Psi}_M-E_M{\boldmath\Psi}_M\trinr_{\text{pw}}.
\end{align}
There exists ${\boldmath\Phi}\in  X$ with $\trinl {\boldmath\Phi}\trinr=1$ and 
\[
 \|N(E_M{\boldmath\Psi}_M)\|_{ X^*} =  N(E_M{\boldmath\Psi}_M;{\boldmath\Phi})=A_{\text{\rm pw}}(E_M{\boldmath\Psi}_M,{\boldmath\Phi})-F({\boldmath\Phi})+\Gamma(E_M{\boldmath\Psi}_M,E_M{\boldmath\Psi}_M,{\boldmath\Phi})
\]
with the definition of $N$. This and the definition of $\widehat{N}({\boldmath\Psi}_M;{\boldmath\Phi})$ lead to
\begin{align*}
 \|N(E_M{\boldmath\Psi}_M)\|_{ X^*}  &=\widehat{N}({\boldmath\Psi}_M;{\boldmath\Phi})+A_{\text{pw}}(E_M{\boldmath\Psi}_M-{\boldmath\Psi}_M,{\boldmath\Phi})
\nonumber  \\
 & \nonumber \; +\Gamma(E_M{\boldmath\Psi}_M,E_M{\boldmath\Psi}_M,{\boldmath\Phi})-\Gamma_{\text{pw}}({\boldmath\Psi}_M,{\boldmath\Psi}_M,{\boldmath\Phi})\\
 &\le \widehat{N}({\boldmath\Psi}_M;{\boldmath\Phi}) +
(1+M\|\Gamma_{\rm pw}\| )\trinl {\boldmath\Psi}_M-E_M{\boldmath\Psi}_M\trinr_{\text{pw}}
\end{align*}
with the  bound  of  $A_{\text{pw}}$, elementary arguments  with the trilinear form and 
its bound  $\|\Gamma_{\rm pw}\|$ 
(deduced from Lemma~\ref{lemmadiscreteembeddings}
as in Remark~\ref{remarkccnew12324boundedness}), 
and $M$  in the second step.
Since ${\boldmath\Psi}_M$ solves  \eqref{vformdNC}, 
\(
\widehat{N}({\boldmath\Psi}_M;{\boldmath\Phi})=\widehat{N}({\boldmath\Psi}_M; {\boldmath\chi})
\)
holds for ${\boldmath\chi}:=(\chi_1,\chi_2):={\boldmath\Phi}-I_M{\boldmath\Phi}$  with the  Morley interpolation $ I_M {\boldmath\Phi}$ of ${\boldmath\Phi}$.  
Since Lemma~\ref{Morley_Interpolation}.a  implies $ A_{\text{pw}}({\boldmath\Psi}_{M},{\boldmath\chi})=0$, the  
definitions of $\Gamma_{\rm pw}(\bullet,\bullet,\bullet)$  and $F(\bullet)$ lead to
\begin{align*}
\widehat{N}({\boldmath\Psi}_M; {\boldmath\Phi})&=
\Gamma_{\text{pw}}({\boldmath\Psi}_M,{\boldmath\Psi}_M,{\boldmath\chi})  - F({\boldmath\chi}) \\
&= 1/2( [u_M,u_M], \chi_2 )_{L^2(\Omega)}   - (f+[u_M,v_M], \chi_1)_{L^2(\Omega)} \\
&\le ( \sum_{K\in\cT} \eta_K^2 )^{1/2} \| h_\cT^{-2}({\boldmath\Phi}-I_M{\boldmath\Phi})\|_{L^2(\Omega)} 
\le C_A ( \sum_{K\in\cT} \eta_K^2 )^{1/2}
\end{align*}
with weighted Cauchy inequalities in the second last step and the constant $C_A\approx 1$ from Lemma~\ref{Morley_Interpolation}.b  with $\trinl {\boldmath\Phi}\trinr=1$  in the end. 
The combination with \eqref{eqccverlateinMayno123} reads
\begin{align*}
\trinl{\boldmath\Psi}-{\boldmath\Psi}_M\trinr_{\text{pw}}  \leq  2 \beta^{-1}  C_A( \sum_{K\in\cT} \eta_K^2 )^{1/2}
+(1+  2\beta^{-1} (1+M\|\Gamma _{\text{pw}}  \| )    )  
\trinl {\boldmath\Psi}_M-E_M{\boldmath\Psi}_M\trinr_{\text{pw}}.
\end{align*}
The last term  is  controlled as in \eqref{eqccverlateinMayno33} and  this concludes the proof of the 
reliability estimate with 
$C_{\rm rel}= \max\{ 2\beta^{-1} C_A ,(1+ 2\beta^{-1}(1+M \|\Gamma _{\text{pw}}  \|))C_B\}  $.

The proof of the efficiency of the volume term $\eta_K$ 
 is immediately adopted from  that of  \cite[Lemma 5.3]{CCGMNN18}.
The arguments in the proof of efficiency for the edge term $\eta_E$ are the same 
as for the (linear) biharmonic  equation and can be adopted from  \cite[p. 322]{CarstensenGallistlHu2013}.
Further details are omitted. 
\end{proof}

\begin{rem}
The adaptation of the nonconforming scheme that allows for a right-hand side $f \in H^{-1}(\Omega)$ is possible by arguments in Section \ref{gen_rhs}.
\end{rem}

\section*{Acknowledgements}

The authors thank for the comments by the anonymous referee that led to Subsection \ref{gen_rhs}.
The research of the first author has been supported by the Deutsche Forschungsgemeinschaft in the Priority Program 1748 under the project "foundation and application of generalized mixed FEM towards nonlinear problems
in solid mechanics" (CA 151/22-2).  The research of the second author is supported by the NBHM Grant 0204/58/2018/R\&D-II/14721. The finalization of this paper has been supported by   DST SERB MATRICS grant of the second author MTR/2017/000199 and SPARC project 
(id 235) entitled {\it the mathematics and computation of plates}.

\section*{Appendix}

\subsection*{Proof of Lemma~\ref{SpecLem}}
The first part of the assertion is included in \cite{SchatzWang96} 
and so is merely outlined for convenient reading of the second. 
The Rellich compact embedding theorem 
$H^1_0(\Omega)\stackrel{c}{\hookrightarrow} L^2(\Omega)$ 
leads to $\lt\stackrel{c}{\hookrightarrow} H^{-1}(\Omega)$ in the sequel. 
Hence $S:=\left\{ g\in\lt\, |\,\|g\|=1\right\}$ is pre-compact in $V^*$. The operator 
$A\in L(V;V^*)$, associated to the scalar product $a$ via $Av=a(v,\bullet)$ for all 
$v\in V$ (note $A$ in contrast to the coefficients ${\bf A}$); $A$ is 
invertible and $A^{-1}\in L(V^*;V)$ maps $S$ onto $W:=A^{-1}(S)$ pre-compact in $H^1_0(\Omega)$. The open balls $ B(z,\epsilon/6)$ in $V$ 
around  $z\in \overline{W}$  with radius $\epsilon/6$ with respect to the norm $\|\bullet\|_a$
form an open cover of the compact set $ \overline{W}$ and so have a finite sub-cover for $z_1,\dots, z_J\in  \overline{W}$, 
\begin{equation}\label{FE_Cover}
W\subset \cup_{j=1}^J  B(z_j,\epsilon/6)\subset V.
\end{equation}
Since $\cD(\Omega)$ is dense in $V$, there exists $\zeta_j\in \cD(\Omega)$ with 
$\| z_j-\zeta_j\|_a <\epsilon/6$. The smoothness of $\zeta_j$ proves 
\begin{equation}\label{F1bdd}
\| \zeta_j-I_C\zeta_j\|_a  \leq C h_{\max} \le C \delta
\end{equation}
for any triangulation $\cT\in\bT(\delta)$ and the nodal interpolation $I_C$ in $S_0^1(\cT)$; the constant $C$ depends on $\displaystyle\max_{j=1,\,\ldots,\, J}\|D^2\zeta_j\|$, 
the shape-regularity parameter $\kappa$, and on $\overline{\lambda}$. 
For any $g\in\lt\setminus \{0\}$ with 
$z=A^{-1}(g)/\|g\|\in W$ from \eqref{FE_Cover}, there exists at least one index $j\in\{1,\ldots,J\}$ with 
$z\in B(z_j,\epsilon/6)$. This, the choice of $\zeta_j$, and   \eqref{F1bdd} 
with $ \delta:=\epsilon/(6C)  $   prove
\[
\| z-I_C\zeta_j\|_a \leq  \| z-z_j\|_a+
\| z_j-\zeta_j\|_a+\| \zeta_j-I_C\zeta_j\|_a<{\epsilon}/{3}+C\delta <\epsilon/2.
\]
A rescaling of this leads to~%
$\displaystyle
\| A^{-1}(g)-\|g\| I_C\zeta_j\|_a =\|g\| \| z-I_C\zeta_j\|_a\leq\epsilon \|g\|/2.$
This proves that the first term in the asserted inequality  is bounded by the right-hand side. 
The analysis of the second term considers the pre-compact subset 
${\bf A}\nabla W=\{ {\bf A}\nabla z: Tz=g\in S\}$ of $L^2(\Omega;\bR^n)$, {where $T: L^2(\Omega) \longrightarrow H^1_0(\Omega)$ is the solution map with $z=Tg \in H^1_0(\Omega)$.}  Since 
the open balls $B(Q,\epsilon/6)$  around $Q\in \overline{{\bf A}\nabla W}$ 
in the $L^2$ norm form an open cover
of the compact closure $\overline{{\bf A}\nabla W}$ in $L^2(\Omega;\bR^n)$, 
there exists $Q_1,\dots, Q_K$ in  $\overline{{\bf A}\nabla W}$ with
\begin{equation}\label{FE_Cover2}
{\bf A}\nabla W \subset \cup_{k=1}^K  B(Q_k,\epsilon/6)\subset L^2(\Omega;\bR^n).
\end{equation}
Since $\cD(\Omega;\bR^n)$ is dense in $L^2(\Omega;\bR^n)$, 
there exists ${\Phi}_k\in \cD(\Omega;\bR^n)$ with 
$\| Q_k-{\Phi}_k\| <\epsilon/6$. The smoothness of ${\Phi}_k$ and a Poincar\'e inequality 
(on simplices with constant $h_T/\pi$) prove
\begin{equation}\label{F1bdd2}
\| {\boldmath\Phi}_k-\Pi_0 \Phi_k\| \leq ||\nabla \Phi_k||\,  h_{\max}/\pi  \le C \delta
\end{equation}
for any triangulation $\cT\in\bT(\delta)$ with the $L^2$ projection $\Pi_0$ onto 
$P_0(\cT;\bR^n)$. The constant $C=\max\{ ||\nabla \Phi_1||,\dots, ||\nabla \Phi_K||\}$ depends on
the smoothness of the  functions $\Phi_1$ ,\dots, $\Phi_K$. 

For any $g\in\lt\setminus \{0\}$ with 
$z=A^{-1}(g)/\|g\|\in W$ from \eqref{FE_Cover}, there exists at least one index $k\in\{1,\ldots,K\}$ with 
${\bf A}\nabla z\in B(Q_k,\epsilon/6)$. This, the choice of $\Phi_k$, and   \eqref{F1bdd2} 
with $ \delta:=\epsilon/(6C)  $   prove
\[
\|  (1-\Pi_0) {\bf A}\nabla z  \|\le \|   {\bf A}\nabla z - \Pi_0 \Phi_k \|
\le  \|   {\bf A}\nabla z -Q_k\|+\|Q_k -\Phi_k\|+\|(1-\Pi_0)\Phi_k\|<\epsilon/2.
\]
A rescaling of this proves
$\displaystyle\|  (1-\Pi_0) {\bf A}\nabla z  \|\le \epsilon\|g\|/2$ for all $Az=g\in L^2(\Omega)$
(with arbitrary norm $\|g\|\ge 0$).
This concludes the proof. \qed

\bibliographystyle{amsplain}
\bibliography{vKeBib}

\providecommand{\bysame}{\leavevmode\hbox to3em{\hrulefill}\thinspace}
\providecommand{\MR}{\relax\ifhmode\unskip\space\fi MR }
\providecommand{\MRhref}[2]{%
  \href{http://www.ams.org/mathscinet-getitem?mr=#1}{#2}
}
\providecommand{\href}[2]{#2}
\begin{thebibliography}{10}

\bibitem{BlumRannacher}
H.~Blum and R.~Rannacher, \emph{On the boundary value problem of the biharmonic
  operator on domains with angular corners}, Math. Methods Appl. Sci.
  \textbf{2} (1980), no.~4, 556--581.

\bibitem{MR3097958}
D.~Boffi, F.~Brezzi, and M.~Fortin, \emph{Mixed finite element methods and
  applications}, Springer Series in Computational Mathematics, vol.~44,
  Springer, Heidelberg, 2013.

\bibitem{BDGK17}
F.~{Bonaldi}, D.~A. {Di Pietro}, G.~{Geymonat}, and F.~{Krasucki}, \emph{A
  hybrid high-order method for kirchhoff-love plate bending problems}, arXiv
  e-prints (2017), arXiv:1706.06781.

\bibitem{BN10}
A.~Bonito and R.H. Nochetto, \emph{Quasi-optimal convergence rate of an
  adaptive discontinuous galerkin method}, SIAM Journal on Numerical Analysis
  \textbf{48} (2010), no.~2, 734--771.

\bibitem{Braess}
D.~Braess, \emph{Finite elements, theory, fast solvers, and applications in
  elasticity theory}, 3rd ed., Cambridge, 2007.

\bibitem{BNRS17}
S.~C. Brenner, M.~Neilan, A.~Reiser, and L.-Y Sung, \emph{A {$C^0$} interior
  penalty method for a von {K}\'{a}rm\'{a}n plate}, Numer. Math. \textbf{135}
  (2017), no.~3, 803--832. \MR{3606463}

\bibitem{Brenner}
S.~C. Brenner and L.~R. Scott, \emph{The mathematical theory of finite element
  methods}, 3rd ed., Springer, 2007.

\bibitem{BSZZ13}
S.~C. Brenner, L.-Y Sung, H.~Zhang, and Y.~Zhang, \emph{A {M}orley finite
  element method for the displacement obstacle problem of clamped {K}irchhoff
  plates}, J. Comput. Appl. Math. \textbf{254} (2013), 31--42. \MR{3061064}

\bibitem{BSZ12}
S.~C. Brenner, L.-Y Sung, and Y.~Zhang, \emph{Finite element methods for the
  displacement obstacle problem of clamped plates}, Math. Comp. \textbf{81}
  (2012), no.~279, 1247--1262. \MR{2904578}

\bibitem{brennermathcomp}
S.C. Brenner, \emph{Preconditioning complicated finite elements by simple
  finite elements}, SIAM J. Sci. Comput. \textbf{17} (1996), no.~5, 1269--1274.
  \MR{1404873}

\bibitem{MR2373954}
S.C. Brenner and L.R. Scott, \emph{The mathematical theory of finite element
  methods}, third ed., Texts in Applied Mathematics, vol.~15, Springer, New
  York, 2008.

\bibitem{BSZ2013}
S.C. Brenner, Li-yeng Sung, H.~Zhang, and Yi~Zhang, \emph{A {M}orley finite
  element method for the displacement obstacle problem of clamped {K}irchhoff
  plates}, J. Comput. Appl. Math. \textbf{254} (2013), 31--42.

\bibitem{Brezzi}
F.~Brezzi, \emph{Finite element approximations of the von {K\'{a}rm\'{a}n}
  equations}, RAIRO Anal. Num\'{e}r. \textbf{12} (1978), no.~4, 303--312.

\bibitem{BrezziRappazRaviart80}
F.~Brezzi, J.~Rappaz, and P.-A. Raviart, \emph{Finite-dimensional approximation
  of nonlinear problems. {I}. {B}ranches of nonsingular solutions}, Numer.
  Math. \textbf{36} (1980), no.~1, 1--25.

\bibitem{CCADNNAKP15}
C.~Carstensen, A.~K. Dond, N.~Nataraj, and A.~K. Pani, \emph{Error analysis of
  nonconforming and mixed {FEM}s for second-order linear non-selfadjoint and
  indefinite elliptic problems}, Numer. Math. \textbf{133} (2016), no.~3,
  557--597.

\bibitem{CCDG14_eigenvalues}
C.~Carstensen and D.~Gallistl, \emph{Guaranteed lower eigenvalue bounds for the
  biharmonic equation}, Numer. Math. \textbf{126} (2014), no.~1, 33--51.

\bibitem{CarstensenGallistlHu2013}
C.~Carstensen, D.~Gallistl, and J.~Hu, \emph{A posteriori error estimates for
  nonconforming finite element methods for fourth-order problems on
  rectangles}, Numer. Math. \textbf{124} (2013), no.~2, 309--335.

\bibitem{CCDGJH14}
\bysame, \emph{A discrete {H}elmholtz decomposition with {M}orley finite
  element functions and the optimality of adaptive finite element schemes},
  Comput. Math. Appl. \textbf{68} (2014), no.~12, part B, 2167--2181.

\bibitem{CCDGNN15}
C.~Carstensen, D.~Gallistl, and N.~Nataraj, \emph{Comparison results of
  nonstandard {$P_2$} finite element methods for the biharmonic problem}, ESAIM
  Math. Model. Numer. Anal. \textbf{49} (2015), no.~4, 977--990.

\bibitem{CCDGMS15}
C.~Carstensen, D.~Gallistl, and M.~Schedensack, \emph{Adaptive nonconforming
  {C}rouzeix-{R}aviart {FEM} for eigenvalue problems}, Math. Comp. \textbf{84}
  (2015), no.~293, 1061--1087.

\bibitem{CCKKDPMS15}
C.~Carstensen, K.~K{\"o}hler, D.~Peterseim, and M.~Schedensack,
  \emph{Comparison results for the {S}tokes equations}, Appl. Numer. Math.
  \textbf{95} (2015), 118--129.

\bibitem{CCGMNN18}
C.~Carstensen, G.~Mallik, and N.~Nataraj, \emph{A priori and a posteriori error
  control of discontinuous {G}alerkin finite element methods for the von
  {K}\'{a}rm\'{a}n equations}, IMA J. Numer. Anal. \textbf{39} (2019), no.~1,
  167--200.

\bibitem{CCNN19}
C.~Carstensen and N.~Nataraj, \emph{Adaptive {M}orley {F}{E}{M} for the von
  {K}\'{a}rm\'{a}n equations with optimal convergence rates}, 2019.

\bibitem{CC_DP_MS12}
C.~Carstensen, D.~Peterseim, and M.~Schedensack, \emph{Comparison results of
  finite element methods for the {P}oisson model problem}, SIAM J. Numer. Anal.
  \textbf{50} (2012), no.~6, 2803--2823.

\bibitem{CN86}
M.~E. Cayco and R.~A. Nicolaides, \emph{Finite element technique for optimal
  pressure recovery from stream function formulation of viscous flows}, Math.
  Comp. \textbf{46} (1986), no.~174, 371--377.

\bibitem{CN89}
\bysame, \emph{Analysis of nonconforming stream function and pressure finite
  element spaces for the {N}avier-{S}tokes equations}, Comput. Math. Appl.
  \textbf{18} (1989), no.~8, 745--760.

\bibitem{Ciarlet}
P.~G. Ciarlet, \emph{The finite element method for elliptic problems},
  North-Holland, Amsterdam, 1978.

\bibitem{CiarletPlates}
\bysame, \emph{Mathematical elasticity: Theory of plates}, vol.~II,
  North-Holland, Amsterdam, 1997.

\bibitem{Clement75}
P.~Cl{\'e}ment, \emph{Approximation by finite element functions using local
  regularization}, Rev. Française Automat. Informat. Recherche
  Op\'{e}rationnelle S\'{e}r. Rouge Anal. Num\'{e}r. \textbf{9} (1975),
  no.~R-2, 77--84.

\bibitem{dd19}
D.~A. Di~Pietro and J.~Droniou, \emph{{The Hybrid High-Order Method for
  Polytopal Meshes}}, 528 pages, June 2019.

\bibitem{DDE15}
D.~A. Di~Pietro, J.~Droniou, and A.~Ern, \emph{A discontinuous-skeletal method
  for advection-diffusion-reaction on general meshes}, SIAM J. Numer. Anal.
  \textbf{53} (2015), no.~5, 2135--2157. \MR{3395131}

\bibitem{DiPetroErn12}
D.~A. Di~Pietro and A.~Ern, \emph{Mathematical aspects of discontinuous
  {G}alerkin methods}, Math\'ematiques \& Applications (Berlin), vol.~69,
  Springer, Heidelberg, 2012.

\bibitem{DEL16}
D.~A. Di~Pietro, A.~Ern, A.~Linke, and F.~Schieweck, \emph{A discontinuous
  skeletal method for the viscosity-dependent {S}tokes problem}, Comput.
  Methods Appl. Mech. Engrg. \textbf{306} (2016), 175--195. \MR{3502564}

\bibitem{Gallistl2014Adaptive}
D.~Gallistl, \emph{Adaptive finite element computation of eigenvalues}, Ph.D.
  thesis, Humboldt-Universit\"at zu Berlin, Mathematisch-Naturwissenschaftliche
  Fakult\"at, 2014.

\bibitem{DG_Morley_Eigen}
\bysame, \emph{Morley finite element method for the eigenvalues of the
  biharmonic operator}, IMA J. Numer. Anal. \textbf{35} (2015), no.~4,
  1779--1811.

\bibitem{Georgoulis2011}
E.~H. Georgoulis, P.~Houston, and J.~Virtanen, \emph{An {\it a posteriori}
  error indicator for discontinuous {G}alerkin approximations of fourth-order
  elliptic problems}, IMA J. Numer. Anal. \textbf{31} (2011), no.~1, 281--298.

\bibitem{Gudi10}
T.~Gudi, \emph{A new error analysis for discontinuous finite element methods
  for linear elliptic problems}, Math. Comp. \textbf{79} (2010), no.~272,
  2169--2189.

\bibitem{HuShi_Morley_Apost}
J.~Hu and Z.~Shi, \emph{A new a posteriori error estimate for the {Morley}
  element}, Numer. Math. \textbf{112} (2009), no.~1, 25--40.

\bibitem{KP07}
O.A. Karakashian and F.~Pascal, \emph{Convergence of adaptive discontinuous
  galerkin approximations of second-order elliptic problems}, SIAM Journal on
  Numerical Analysis \textbf{45} (2007), no.~2, 641--665.

\bibitem{MR1344684}
C.~T. Kelley, \emph{Iterative methods for linear and nonlinear equations},
  Frontiers in Applied Mathematics, vol.~16, Society for Industrial and Applied
  Mathematics (SIAM), Philadelphia, PA, 1995.

\bibitem{kreuzer2019convergence}
C~Kreuze and E.H. Georgoulis, \emph{Convergence of adaptive discontinuous
  galerkin methods (corrected version of [math. comp. 87 (2018), no. 314,
  2611--2640])}, 2019.

\bibitem{Lions}
J.~L. Lions, \emph{Quelques m\'{e}thodes de r\'{e}solution des probl\`{e}mes
  aux limites non lin\'{e}aires}, Dunod, Paris (1969), 53--57.

\bibitem{GMNN_CFEM}
G.~Mallik and N.~Nataraj, \emph{Conforming finite element methods for the von
  {K\'{a}rm\'{a}n} equations}, Adv. Comput. Math. (2016), 1--24.

\bibitem{GMNN_NCFEM}
\bysame, \emph{A nonconforming finite element approximation for the von
  {K\'{a}rm\'{a}n} equations}, ESAIM Math. Model. Numer. Anal. \textbf{50}
  (2016), no.~2, 433--454.

\bibitem{N10}
M.~Neilan, \emph{A nonconforming {M}orley finite element method for the fully
  nonlinear {M}onge-{A}mp\`ere equation}, Numer. Math. \textbf{115} (2010),
  no.~3, 371--394. \MR{2640051}

\bibitem{SchatzWang96}
A.~H. Schatz and J.~P. Wang, \emph{Some new error estimates for
  {R}itz-{G}alerkin methods with minimal regularity assumptions}, Math. Comp.
  \textbf{65} (1996), no.~213, 19--27.

\bibitem{Verfurth}
R.~Verf\"{u}rth, \emph{A aposteriori error estimation techniques for finite
  element methods}, Oxford University Press, 2013.

\bibitem{MR816732}
E.~Zeidler, \emph{Nonlinear functional analysis and its applications. {I}},
  Springer-Verlag, New York, 1986, Fixed-point theorems, Translated from the
  German by Peter R. Wadsack.

\end{thebibliography}

\end{document}